\documentclass{amsart}
\usepackage{medstar}
\usepackage{tikz-cd}
\usepackage[unicode]{hyperref}
\usepackage{graphicx,color}
\usepackage{amssymb}
\usepackage{mathrsfs}
\usepackage{enumerate}
\usepackage{pinlabel}
\usepackage[utf8x]{inputenc}
\PrerenderUnicode{→}
\PrerenderUnicode{Σ}
\PrerenderUnicode{Ξ}
\PrerenderUnicode{₂}

\newtheorem{theorem}{Theorem}[section]
\newtheorem{proposition}[theorem]{Proposition}
\newtheorem{corollary}[theorem]{Corollary}
\newtheorem{lemma}[theorem]{Lemma}
\theoremstyle{definition}
\newtheorem{definition}[theorem]{Definition}
\newtheorem{example}[theorem]{Example}
\newtheorem{remark}[theorem]{Remark}
\newtheorem{notation}[theorem]{Notation}
\newtheorem{conjecture}[theorem]{Conjecture}

\DeclareMathOperator{\vertex}{Vt}
\DeclareMathOperator{\edge}{Ed}
\newcommand{\ord}{\operatorname{ord}}
\DeclareMathOperator{\legs}{Legs}
\DeclareMathOperator{\nbhd}{Nbhd}
\DeclareMathOperator{\inp}{in}
\DeclareMathOperator{\out}{out}
\DeclareMathOperator{\Ob}{Ob}
\DeclareMathOperator{\interior}{Int}
\DeclareMathOperator{\image}{Im}
\DeclareMathOperator{\prof}{Prof}
\newcommand{\unrootedtrees}{\mathtt{tree}}

\DeclareMathOperator{\Iso}{Iso}
\DeclareMathOperator{\Aut}{Aut}

\DeclareMathOperator{\id}{id}
\DeclareMathOperator{\sbgph}{Sbgph}

\newcommand{\Set}{\mathbf{Set}}
\newcommand{\Cyc}{\mathbf{Cyc}}
\newcommand{\Operad}{\mathbf{Op}}
\newcommand{\SSet}{\mathbf{sSet}}

\DeclareMathOperator*{\colim}{colim}
\newcommand{\fC}{\mathfrak{C}}
\newcommand{\fD}{\mathfrak{D}}

\newcommand{\uc}{{\underline c}}

\DeclareMathOperator{\colset}{Col}

\newcommand{\rootify}{\mathscr{T}}
\newcommand{\amalgamate}{\mathscr{A}}

\newcommand{\findroot}{\odot}
\newcommand{\lifting}{\mathscr{L}}

\newcommand{\cstar}{\medstar}

\newcommand{\Rr}{\mathbb{R}}

\newcommand{\rlp}[1]{{#1}^\boxslash}
\newcommand{\llp}[1]{{}^\boxslash\!{#1}}

\newcommand{\red}{\mathscr{R}}
\newcommand{\incl}{\mathscr{I}}

\newcommand{\Seg}{\mathscr{S}}

\DeclareMathOperator{\map}{map}
\DeclareMathOperator{\Sc}{Sc}

\newcommand{\boundarymap}{\eth}

\title{Higher cyclic operads}

\author[P. Hackney]{Philip Hackney}
\address{Institut f\"ur Mathematik \\ Universit\"at Osnabr\"uck \\ Osnabr\"uck, Germany}
\address{Max-Planck-Institut f\"ur Mathematik \\ Bonn, Germany}
\curraddr{Department of Mathematics\\ University of Louisiana at Lafayette\\ Lafayette, LA\\ USA}
\email{philip@phck.net} 
\urladdr{http://phck.net}

\author[M. Robertson]{Marcy Robertson}
\address{School of Mathematics and Statistics \\ The University of Melbourne \\ Melbourne, Victoria, Australia}\email{marcy.robertson@unimelb.edu.au}

\author[D. Yau]{Donald Yau}
\address{Department of Mathematics\\
The Ohio State University at Newark\\
Newark, OH \\ USA}
\email{dyau@math.ohio-state.edu}

\begin{document}

\begin{abstract}
We introduce a convenient definition for weak cyclic operads, which is based on unrooted trees and Segal conditions.
More specifically, we introduce a category $\Xi$ of trees, which carries a tight relationship to the Moerdijk--Weiss category of rooted trees $\Omega$.
We prove a nerve theorem exhibiting colored cyclic operads as presheaves on $\Xi$ which satisfy a Segal condition.
Finally, we produce a Quillen model category whose fibrant objects satisfy a weak Segal condition, and we consider these objects as an up-to-homotopy generalization of the concept of cyclic operad.
\end{abstract}

\maketitle

For certain operads, such as the moduli space of Riemann spheres with labeled punctures or the endomorphism operad of a vector space $V$ equipped with a non-degenerate bilinear form, there is not really a qualitative difference between the notion of input and output.
Indeed, in the former case, the `output' of a given element arises solely from our choice of labels and not the underlying geometry, while in the latter case we have natural isomorphisms \[
	\operatorname{End}_V(n) = \hom(V^{\otimes n}, V) = \hom(V^{\otimes n}, V^*) = \hom(V^{\otimes n+1}, k).
\]

\begin{figure}
\labellist
\small\hair 2pt
 \pinlabel {$1$} at 13 219
 \pinlabel {$2$} at 71 219
 \pinlabel {$3$} at 133 219
 \pinlabel {$4$} at 195 219
 \pinlabel {$5$} at 260 219
 \pinlabel {$2$} at 376 219
 \pinlabel {$3$} at 439 219
 \pinlabel {$4$} at 500 219
 \pinlabel {$5$} at 567 219
 \pinlabel {$0$} at 596 219
 \pinlabel {$0$} at 145 21
 \pinlabel {$1$} at 316 21
\endlabellist
\centering
\includegraphics[width=0.9\textwidth]{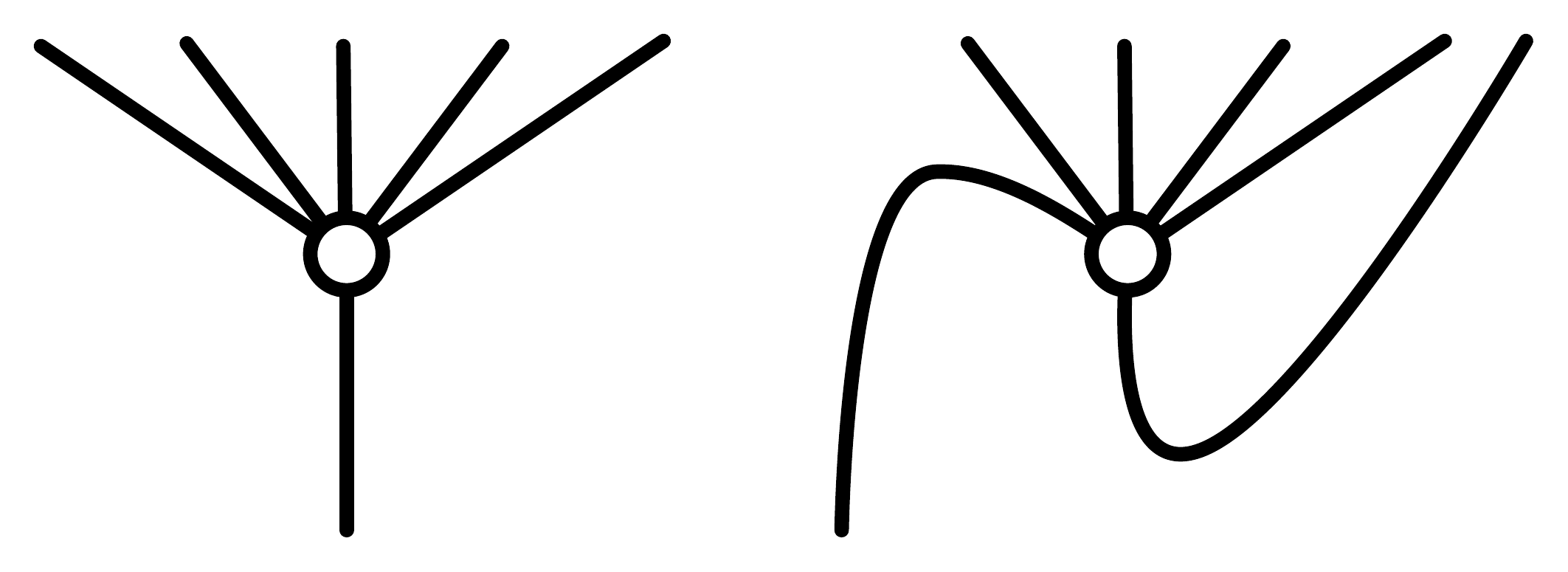}
\caption{$f$ and $f\tau$}\label{small flip flop}
\end{figure}
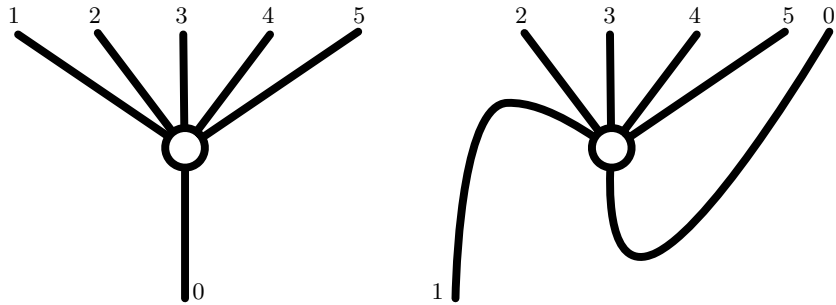
This consideration leads directly to the notion of \emph{cyclic operad}, introduced by Getzler and Kapranov in \cite{MR1358617} (although we add the axiom due to van der Laan, see \cite[\S 11]{vanderlaan} and \cite[\S II.5.1]{mss}).
A cyclic operad is an operad $O$ with extra structure, namely an action of the cyclic group $C_{n+1} = \langle \tau \rangle$ on the space $O(n)$.
Applied to an element $f \in O(n)$, we should regard $f\tau \in O(n)$ as $f$ with the first input changed to the output and the output changed to the last input, as in Figure \ref{small flip flop}.
To get a feel for how this cyclic operator should act on compositions, one should look at trees with several vertices like the one in Figure \ref{big flip flop}.
\begin{figure}
\labellist
\small\hair 2pt
 \pinlabel {$f$} [l] at 96 73
 \pinlabel {$g$} [l] at 60 126
 \pinlabel {$h$} [l] at 93 179
 \pinlabel {$f\tau$} [l] at 305 73
 \pinlabel {$g\tau$} [l] at 269 126
 \pinlabel {$h$} [l] at 304 179
 \pinlabel {$f\tau$} [l] at 528 123
 \pinlabel {$g\tau$} [l] at 491 73
 \pinlabel {$h$} at 461 163
\endlabellist
\centering
\includegraphics[width=\textwidth]{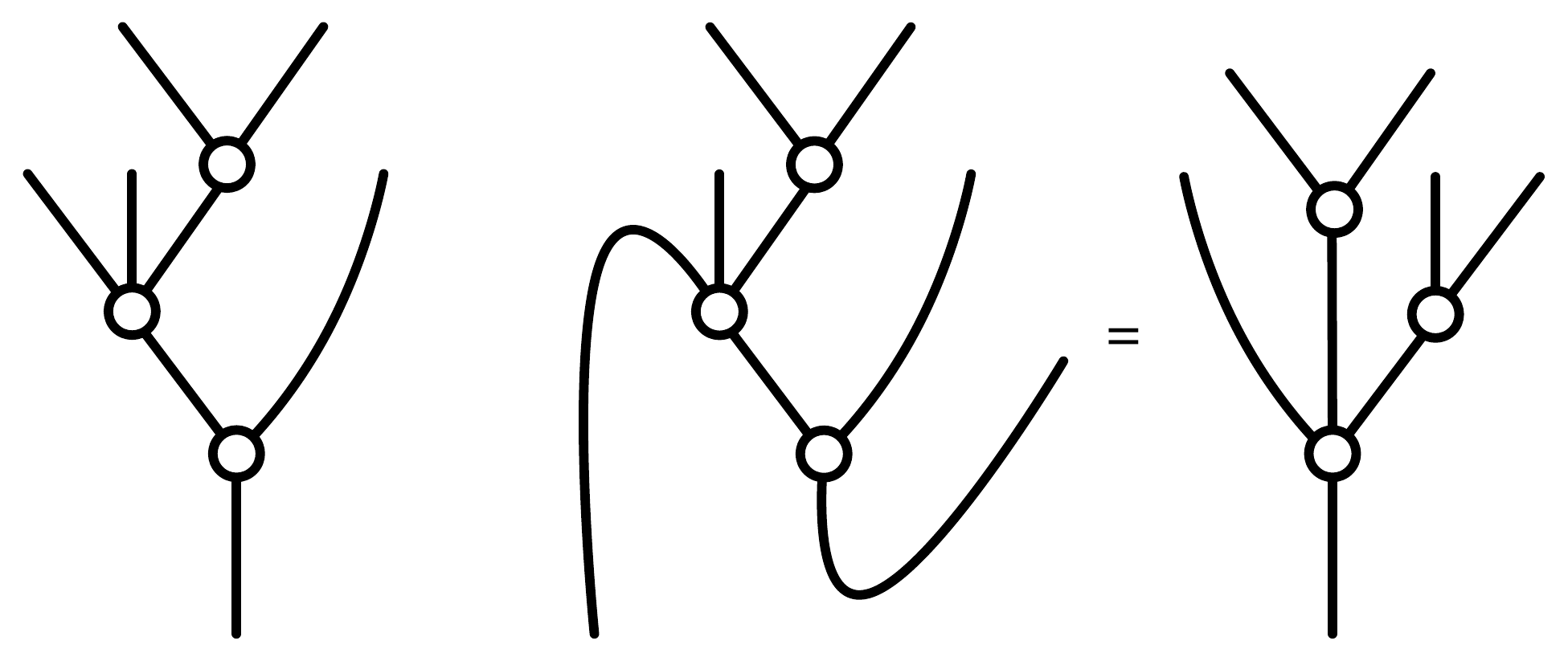}
	\caption{A composition of $f,g,h$, and the action of $\tau$ on this composition}\label{big flip flop}
\end{figure}
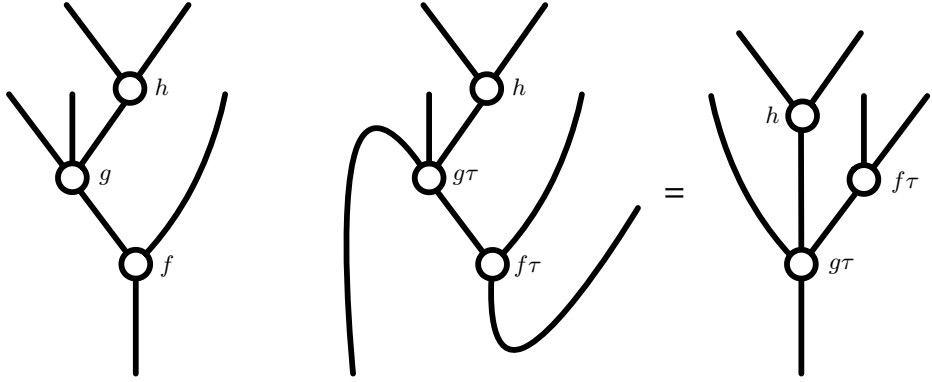
Given a cyclic operad, one can begin talking about graph homology (see \cite{kontsevich} as well as the generalizations of Conant \& Vogtmann \cite{MR2026331}), whereas algebras over cyclic operads admit a cyclic homology theory \cite{MR1358617}.

Further examples of cyclic operads include the associative, Lie, and commutative operads, (certain models for) the framed little $n$-disks operad \cite{budney,ksv}, the $A_\infty$ operad \cite{MR1358617}, and also any monoid with involution \cite{lindahl1971,MR0068138} (that is, the involution $x\mapsto x^\dagger$ satisfies $x^\dagger y^\dagger = (yx)^\dagger$) regarded as an operad concentrated in degree 1.
The last of these is useful for giving small examples (see Example \ref{cyclic example} and Proposition \ref{not q equiv}), but also gives a connection with another interesting class of mathematical objects.

A dagger category \cite[Definition 2.2]{SELINGER2007139} is a category $\mathcal C$ together with an involutive functor $\dagger : \mathcal C^{op} \to \mathcal C$ which is the identity on objects. 
In other words, a dagger structure on $\mathcal C$ is an assignment $(f : X \to Y) \mapsto (f^\dagger : Y \to X)$ satisfying $f^{\dagger \dagger} = f$ and $f^\dagger g^\dagger = (gf)^\dagger$; this is the many-objects version of a monoid with involution.
As Baez argued eloquently in \cite{quandries}, to understand the similarities between general relativity and quantum theory, one should begin by considering the natural dagger structure on the category of $n$-cobordisms and on the category of Hilbert spaces, respectively.
Other important examples of dagger categories include any groupoid, categories of relations, and categories of correspondences. 

Colored cyclic operads are a simultaneous generalization of cyclic operads (which we might term `monochrome cyclic operads') and of dagger categories.
There are additional examples in the literature (e.g., \cite[\S 5.4]{cohenvoronov} \cite[\S 3.11.1]{MR2986860}), and the concept provides a bridge to approaches to higher operads based on colored operads.
If $O$ is a $\fC$-colored operad and $n\geq 0$, then the object $O_n = \coprod_{c,c_1,\dots, c_n \in \fC} O(c_1, \dots, c_n; c)$ admits a right action by $\Sigma_n = \Aut\{1,\dots, n\}$ compatible with operadic composition. Write $\Sigma_n^+ = \Aut\{0,1,\dots, n\}$ and identify $\Sigma_n$ as the isotropy group of $0$.
\begin{definition}\label{defn cyclic operad}
A \emph{cyclic structure} on a $\fC$-colored operad $O$ is a collection of maps $-\cdot \sigma : O(c_1, \dots, c_n; c_0) \to O(c_{\sigma(1)}, \dots, c_{\sigma(n)}; c_{\sigma(0)})$ for $\sigma \in \Sigma_n^+$ and $c_i \in \fC$. These should satisfy two conditions. First, they assemble into a right $\Sigma_n^+$ action on $O_n$, which agrees with the existing $\Sigma_n$ action. For the second condition, let $\tau_{n+1} \in \Sigma_n^+$ be the element with $\tau_{n+1}(n) = 0$ and $\tau_{n+1}(i) = i +1$ for $0 \leq i < n$. We insist that if $g\in O(c_1, \dots, c_k; c_0)$ and $f \in O(d_1, \dots, d_\ell; c_i)$ are composable at position $i$, then\footnote{
	We exclude the case $i=1, \ell=0$ from \eqref{equation cyclic operad condition}, as the formula $(g\circ_1 f)\cdot \tau = (g\cdot \tau^2) \circ_k f$ follows from the first case. Indeed,
	\[
		[(g\cdot \tau_{k+1}^2) \circ_k f] \cdot \tau_k^{k-1} = (g\cdot \tau_{k+1}^{2+k-1}) \circ_{k-(k-1)} f = (g\cdot \tau_{k+1}^{k+1}) \circ_{1} f = g\circ_1 f.
	\]
}
\begin{equation}\label{equation cyclic operad condition}
(g\circ_i f) \cdot \tau_{k+\ell} = \begin{cases}
(g\cdot \tau_{k+1}) \circ_{i-1} f & \text{if $2 \leq i \leq k$} \\
(f \cdot \tau_{\ell+1}) \circ_\ell (g\cdot \tau_{k+1}) & \text {if $i=1$ and $\ell \neq 0$.} 
\end{cases}
\end{equation}
Informally, when $O$ is equipped with a cyclic structure, we will say that $O$ is a \emph{$\fC$-colored cyclic operad.}\footnote{We should note that our definition is not a symmetric version of the `cyclic multicategories' of Cheng, Gurski, and Riehl \cite{MR3189430}.
A non-symmetric colored cyclic operad is a non-symmetric colored operad $O$ together with an action of the subgroup $\langle \tau_{n+1} \rangle \leq \Sigma_n^+$ on $O_n$, so that \eqref{equation cyclic operad condition} holds. These form a reflective subcategory of the category of cyclic multicategories.}
\end{definition}

The up-to-homotopy cyclic operads that we develop in this paper are a variation on `dendroidal models' for $\infty$-operads (cf. \cite{bh1,cm-ho,cm-ds,cm-simpop,mw2}).
The dendroidal category $\Omega$ is a category of rooted trees \cite{mw}; each such rooted tree $T$ (with edge set $\edge(T)$) can be regarded as a free object in the category of $\edge(T)$-colored operads.
The dendroidal category is then defined to be the full subcategory of the category of all colored operads whose objects are the rooted trees.
Not only is $\Omega$ defined as a subcategory of colored operads, but it turns out that there is a model structure on the category of presheaves of $\Omega$ (see \cite[Theorem 2.4]{cm-ho}) that is Quillen equivalent to a model structure on the category of simplicially-enriched colored operads (see \cite{cm-simpop}).
This is an extension of the equivalence between the Joyal model structure on simplicial sets and the Bergner model structure on simplicially-enriched categories (see \cite{juliesurvey} for references). 

It would be a very ambitious project to attempt to do all of the above for colored cyclic operads, and we are skeptical that the full Cisinski--Moerdijk program can be carried out in the cyclic case. 
A key difficulty is that the adjunction between categories and dagger categories is badly behaved, in particular with respect to equivalences.
Thus, in the present paper we limit what is said about colored cyclic operads.
It is true that every unrooted tree $S$ freely generates an $\edge(S)$-colored cyclic operad $C(S)$ (see Section \ref{towards cyclic}), but we do not ever consider the full-subcategory of colored cyclic operads spanned by the unrooted trees. 
The cyclic operad $C(S)$ is nearly always infinite, even when $S$ is a linear tree, and arbitrary maps $C(S) \to C(R)$ do not admit decompositions into cofaces and codegeneracies, as they do in the dendroidal setting.
Instead, we directly construct a category $\Xi$ of unrooted trees that is reminiscent of $\Omega$. 
The assignment $S \mapsto C(S)$ gives a faithful, non-full functor from $\Xi$ to $\Cyc$ (Theorem \ref{C is faithful}, Example \ref{nonfull example}).
We use this to prove a nerve theorem for colored cyclic operads (Theorem \ref{cyclic dendroidal nerve theorem}).

Our main goal is to propose a model for weak monochrome cyclic operads. 
These are called \emph{Segal cyclic operads} in Section \ref{section segal cyclic}, and they are certain reduced presheaves satisfying a Segal condition. 
The Segal cyclic operads are patterned after the Segal operads appearing in the work of Bergner and the first author \cite{bh1}, which have become important in current work of Boavida, Horel, and the second author on profinite completions of the framed little disks operad.

The profinite completion of a product of spaces is weakly equivalent, but in general not isomorphic, to the product of the profinite completions. 
For this reason, the profinite completion of an operad does not yield an operad on the nose but rather an $\infty$-operad. 
This fact has played a crucial role in work of Horel \cite{1504.01605} when he generalized work of Fresse \cite{FresseBook1} and computed a profinite version of the Grothendieck--Teichm\"uller group $\widehat{GT}\cong \pi_0\operatorname{End}^h(\widehat{D_2})$, where $D_2$ is the little 2-disks operad. 
In work by Boavida, Horel and the second author, they show that considering the framed little 2 disks as an operad, they recover exactly the same result, i.e., $\widehat{GT}\cong \pi_0\operatorname{End}^h(\widehat{D_2})\cong \pi_0\operatorname{End}^h(\widehat{fD_2})$.  
Considering $fD_2$ as a cyclic operad would necessarily result in a smaller set of endomorphisms and conjecturally would provide refinement on these computations; of course one would expect the profinite completion of a cyclic operad to be some type of infinity cyclic operad. 
Providing a good foundation for this project is one of the major motivations for the present paper.

\subsection*{Overview} We give a brief outline of the paper. Each section begins with a more substantial summary of its contents.

The first section is dedicated to the construction of the category $\Xi$ of unrooted trees.
In the second section, we examine exactly how close $\Xi$ is to the category $\Omega$ of rooted trees.
The third and fourth sections are devoted to two structures on the category $\Xi$: a generalized Reedy structure and an active / inert (or generic / free) weak factorization system. 

The next two sections deal with the relationship of $\Xi$ to colored cyclic operads. 
In the fifth section we construct the functor $\Xi \to \Cyc$, and in the sixth we prove a nerve theorem for colored cyclic operads.

The final two sections are devoted to model-categorical matters. 
The penultimate section is about the model structure on diagrams indexed by a generalized Reedy category, and at the beginning of the section we show that this model structure usually has properties which ensure that Bousfield localizations exist.
We then restrict to the case when the base category is the category of simplicial sets. In Section \ref{reduced presheaves} we discuss certain cases when categories of \emph{reduced} presheaves of simplicial sets admit model structures.
In Section \ref{simplicial model structures} we show that these model categories are in fact simplicial model categories.

In the last section we prove the existence of a model structure on reduced $\Xi$-presheaves in simplicial sets whose fibrant objects, the Segal cyclic operads, satisfy a Segal condition.
We show that there is a Quillen adjunction (which is not a Quillen equivalence) between this model structure and the model structure for Segal operads from \cite{bh1}.

Finally, in an appendix, we discuss certain additional (co)tensorings by $\Sigma_2$-simplicial sets, which exists for $\Xi$-presheaves which vanish on non-linear trees.

\subsection*{Notational conventions}
If $\mathcal C$ is a category, we will write $\mathcal C(x, y)$ or $\hom(x,y)$ for the set of morphisms from $x$ to $y$, depending on if the name of our category is short (e.g., $\mathcal C = \Xi$) or long (e.g., $\mathcal C = \SSet^{\Xi^{op}}_\ast$).
We will write $\Iso_{\mathcal C}(x,y)$ for the isomorphisms from $x$ to $y$, $\Aut_{\mathcal C}(x) := \Iso_{\mathcal C}(x,x)$ for the invertible self-maps of $x$, and $\Iso(\mathcal C)$ for the wide subcategory of $\mathcal C$ consisting of all of the isomorphisms.
In all adjunctions $\mathcal C \rightleftarrows \mathcal D$, the top arrow denotes the left adjoint.

Throughout this paper we use freely the language of Quillen model categories and take the book of Hirschhorn \cite{hirschhorn} as our standard reference.

\subsection*{Acknowledgments}
We are grateful to the Hausdorff Research Institute for Mathematics and the Max Planck Institute for Mathematics for their hospitality during the fall of 2016.

The authors have had many interesting discussions about parts of this paper since its conception, but we would like especially to thank Clark Barwick, Julie Bergner, and Richard Garner for some helpful insights that came at precisely the right time. 
We would also like to thank the participants of the workshop \href{http://www.him.uni-bonn.de/en/programs/past-programs/past-junior-trimester-programs/topology-2016/workshop-interactions-between-operads-and-motives/}{`Interactions between operads and motives'} at HIM for useful feedback and questions.
Finally, we would like to express our appreciation to the anonymous referees of this paper for their generous and extensive feedback.

\section{The unrooted tree category \texorpdfstring{$\Xi$}{Ξ}}

The main goal of this section is to define a category of unrooted trees $\Xi$.
We will begin with a formalism for general graphs, before defining the objects of $\Xi$ in Definition \ref{def trees}. 
We give two distinct descriptions of the morphisms of $\Xi$ in Definition~\ref{definition complete morphism} and Definition \ref{maps of xi}.
Each has its own advantage: morphisms in the former sense (here called \emph{complete}) immediately form a category, while morphisms in the latter sense are specified by a smaller set of data, and are easier to work with in most situations.
We then embark on a sustained study of the nature of these morphisms; key tools are the notions of distance and a (minimal) path in a tree.
Along the way, we recover the Moerdijk--Weiss dendroidal category $\Omega$.
Finally, in Proposition~\ref{bijection complete and kernel}, we show that the two definitions of morphisms coincide.

At the heart of this work is the notion of `graph with legs'.
One can choose several formalisms; for concreteness, let us say that an \emph{undirected graph with legs} consists of two finite sets $E$ and $V$ and a function $\nbhd: V \to \mathcal P(E)$ (the set of subsets of $E$).
		This data should satisfy one axiom, namely that, for each $e \in E$,
		\[ \left|\{ v \in V \, |\, e \in \nbhd(v) \} \right| \leq 2. \]
We will package the triple $(E,V,\nbhd)$ into a single symbol $G$, and write $\edge(G) = E$ and $\vertex(G) = V$.
Edges actually come in two types, namely interior edges
\[
	\interior(G) = \{ e\in E | e\in \nbhd(v) \cap \nbhd(w) \text{ for some } v\neq w \}
\]
and the set of legs
\[
	\legs(G) = \edge(G) \setminus \interior(G)
\]
which are edges incident to at most one vertex.\footnote{If $e \in \edge(G)$ is not incident to any vertex, then one should really think that $e$ appears twice in $\legs(G)$. Since we are only concerned with connected graphs for the bulk of this paper, only one graph (see Example \ref{unrooted tree examples}) has an edge with this property, so we will just systematically single out that special case.}
If $v$ is a vertex of $G$, we also write $|v|$ for the valence of $v$, or the cardinality of the set $\nbhd(v)$.

Every graph has an underlying topological space, which can be described as follows.
See the left hand side of Figure \ref{rwb figure} for an example.
\begin{definition}[Space associated to a graph]
Fix an $\epsilon$ with $0 < \epsilon < 1$, which we can use to scale the closed unit disc $\mathbb D$ in the complex plain $\mathbb C$.
Define \[ \left|\cstar_0\right| = \epsilon \mathbb D = \{ r e^{i\theta} \,|\, 0 \leq r \leq \epsilon, 0\leq \theta \leq 2\pi \} \subsetneq \mathbb D \subsetneq \mathbb C\] and, for $n > 0$,
\[
	\left|\cstar_n \right| = \epsilon \mathbb D \cup \bigcup_{k=0}^{n-1} \{ r e^{\frac{k}{n}2\pi i} \, | \, 0\leq r \leq 1 \},
\]
considered as a subspace in the closed unit disc of the complex plane. 
If $G$ is an undirected graph with legs, fix bijections 
\[
	\kappa_v \colon \nbhd(v) \overset\cong\to \{ e^{\frac{k}{|v|}2\pi i} \} = S^1 \cap \left|\cstar_{|v|} \right| 
\]
and define 
\[
|G| = \left( \frac{\coprod\limits_{v\in \vertex(G)} \left|\cstar_{|v|} \right|}{\kappa_v(e) \sim \kappa_w(e)} \right) \amalg
Q \times [0,1]
\]
where $e\in \nbhd(v) \cap \nbhd(w)$ and $Q = \edge(G) \setminus \bigcup_{v\in \vertex(G)} \nbhd(v)$.
\end{definition}

Notice that the homeomorphism type of $|G|$ determines the isomorphism type of $G$. This would not be the case if we did not add some thickness at the centers of $\left|\cstar_n \right|$ by using the $\epsilon \mathbb D$. Indeed, a variation of realization with $\epsilon = 0$ produces the closed unit interval $[0,1]$ on both the graph $G_1$ with one edge and no vertices and on the graph $G_2$ with one vertex $v$, one edge $e$, and $\nbhd(v) = \{ e \}$.

\begin{figure}
	\includegraphics[width=\textwidth]{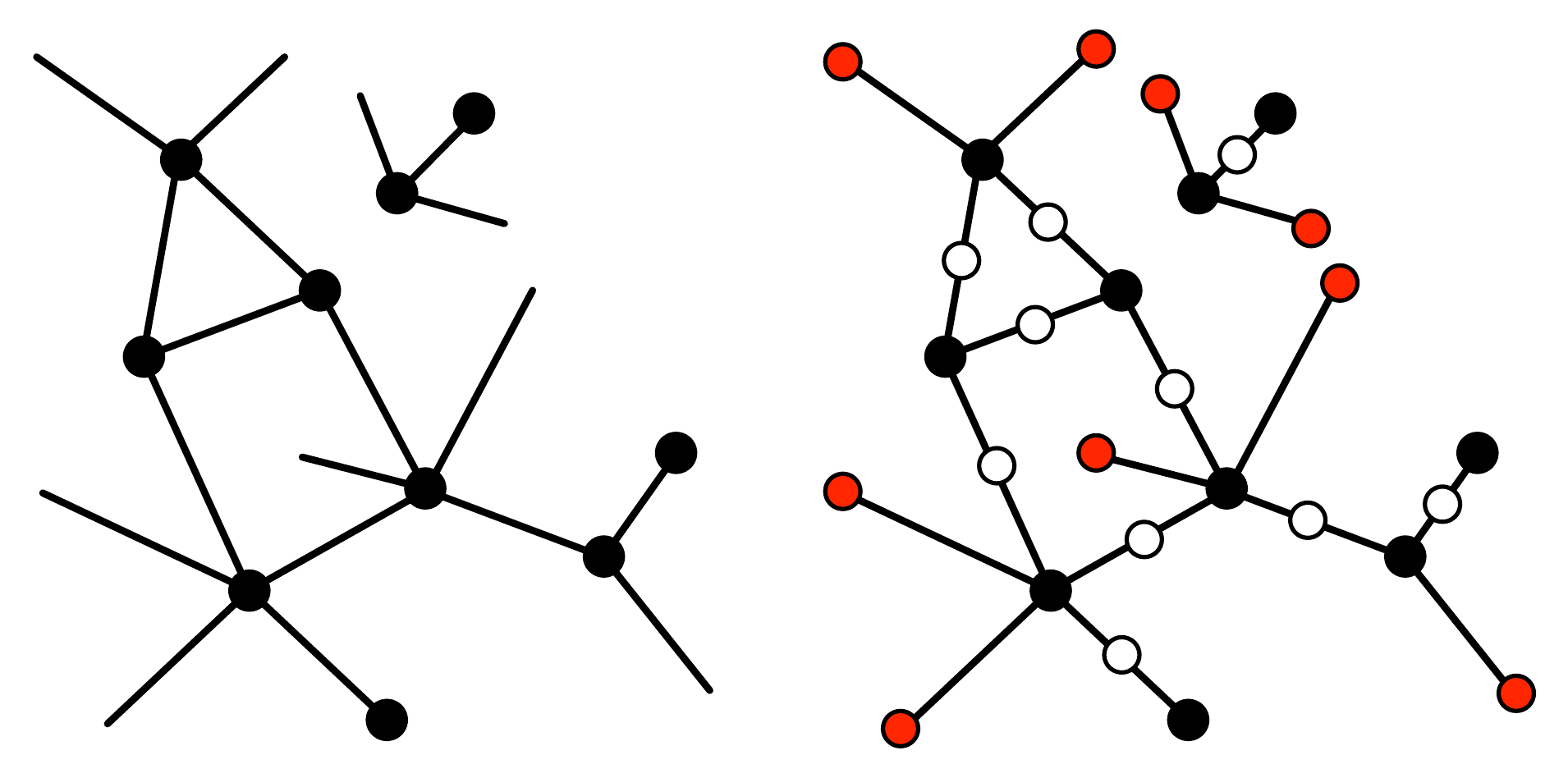}
	\caption{A graph with legs and its corresponding rwb graph}\label{rwb figure}
\end{figure}

The following is an alternative, equivalent formalism for graph with legs.
\begin{definition}[Red-white-black formalism]\label{rwb definition}
An \emph{rwb graph} is an ordinary undirected graph (see, for example \cite[\S 1.1]{diestel}) where each vertex is colored either red, white, or black and such that
\begin{itemize}
	\item red vertices are univalent,
	\item white vertices are bivalent and are only adjacent to black vertices, and
	\item a black vertex is not adjacent to any other black vertex.
\end{itemize}
\end{definition}

From a graph with legs, we can form an rwb graph by coloring all vertices black, adding a white vertex on each interior edge, and adding a red vertex to the loose end of each leg.
Each rwb graph determines a graph with legs by deleting the white vertices and joining the edges on either side and  deleting all of the red vertices. See Figure \ref{rwb figure} for an illustration of this correspondence.

\subsection{Trees}

The category $\Xi$ governing cyclic dendroidal sets has `unrooted' or `cyclic' trees as objects.

\begin{definition}\label{def trees}
	An \emph{unpinned tree} $S$ is an undirected graph with legs which is contractible, has at least one leg, 
	and is equipped with bijections 
	\[
		\ord^v : \{ 0, 1, \dots, n_v \} \overset{\cong}\to \nbhd(v),
	\]
	where $\nbhd(v) \subseteq \edge(S)$ is the set of vertices adjacent to $v$.
	A \emph{pinned tree}, or just \emph{tree}, has, in addition, a map \[ 
	\ord : \{ 0, 1, \dots, n \} \to \legs(S), \]
	where $\legs(S) \subseteq \edge(S)$ is the set of legs of $S$ which is a bijection if $S$ contains a vertex and is otherwise the unique map from $\{0,1\}$ to the single edge. 
\end{definition}

A typical example of such a graph is found in Figure \ref{figure basic tree example}.
In pictures of graphs, we will always draw the ordered set of legs $\nbhd(v)$ in a counterclockwise fashion. Using this convention, to specify the unpinned structure we only need to mark the edges $\{\ord^v({\color{red}0})\}_v$ in the figures.

\begin{figure}
\labellist
\small\hair 2pt
 \pinlabel {$u$} at 40 66
 \pinlabel {$v$} at 94 66
 \pinlabel {$w$} at 94 21
 \pinlabel {$u$} at 190 66
 \pinlabel {$v$} at 244 66
 \pinlabel {$w$} at 244 21
 \pinlabel {$a$} at 31 42
 \pinlabel {$b$} at 14 89
 \pinlabel {$c$} at 67 59
 \pinlabel {$f$} at 85 94
 \pinlabel {$d$} at 101 44
 \pinlabel {$e$} at 126 74
 \pinlabel {{\color{red}$1$}} at 185 85
 \pinlabel {{\color{red}$2$}} at 174 59
 \pinlabel {{\color{red}$0$}} at 205 59
 \pinlabel {{\color{red}$0$}} at 229 74
 \pinlabel {{\color{red}$1$}} at 238 50
 \pinlabel {{\color{red}$2$}} at 259 59
 \pinlabel {{\color{red}$3$}} at 251 85
 \pinlabel {{\color{red}$0$}} at 251 36
 \pinlabel {{\color{blue}$0$}} at 251 110
 \pinlabel {{\color{blue}$1$}} at 171 104
 \pinlabel {{\color{blue}$2$}} at 286 74
 \pinlabel {{\color{blue}$3$}} at 171 30
\endlabellist
\centering
\includegraphics[width=0.6\textwidth]{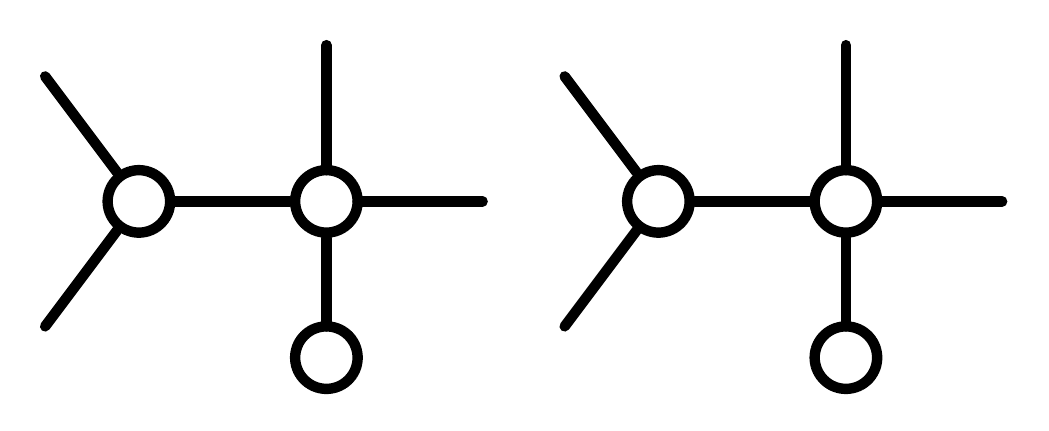}
\caption{Underlying graph (left), unpinned and pinned tree data (right)}\label{figure basic tree example}
\end{figure}
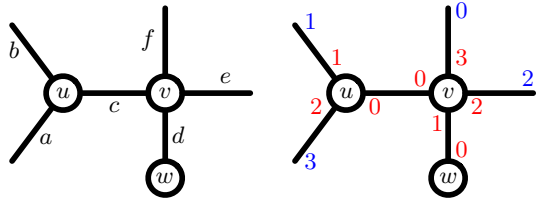

\begin{example}\label{unrooted tree examples}
Let us fix several foundational examples of trees (Figure \ref{figure unrooted trees examples}).
\begin{itemize}
	\item The graph with one edge and no vertices, which we write as $\eta$.
	\item For each $n > 0$, the graph $\cstar_n$. This graph has a single vertex $v$ and $n$ edges $\{ 0, 1, \dots, n-1 \}$ (and take $\ord^v = \ord = \id$).\footnote{Note the shift in index compared with \cite[p. 250]{mss}: they use the notation $\ast_{n}$ for what we call $\cstar_{n+1}$.}
	\item For $n \geq 0$, the linear graph $L_n$ with $n$ distinct vertices $\{v_1, \dots, v_n \}$, $n+1$ distinct edges $\{ e_0, \dots, e_n \}$, $\nbhd(v_i) = \{e_{i-1}, e_i \}$, 
	so that $\ord^{v_i}(t) = e_{i-1+t}$, $\ord(0) = e_0$, and $\ord(1) = e_n$. Note that $L_0 \cong \eta$.
\item We will call any tree with all vertices bivalent a linear graph.
\end{itemize}
\end{example}

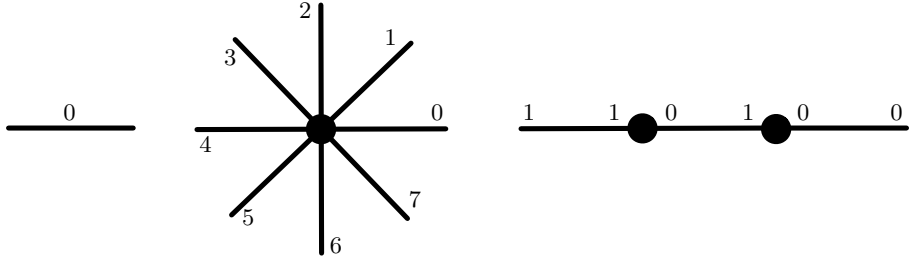
\begin{figure}
\labellist
\small\hair 2pt
 \pinlabel {$0$} at 43 92
 \pinlabel {$0$} at 224 92
 \pinlabel {$1$} at 201 129
 \pinlabel {$2$} at 159 142
 \pinlabel {$3$} at 122 119
 \pinlabel {$4$} at 110 76
 \pinlabel {$5$} at 131 40
 \pinlabel {$6$} at 174 26
 \pinlabel {$7$} at 213 49
 \pinlabel {$1$} at 269 92
 \pinlabel {$1$} at 311 92
 \pinlabel {$0$} at 339 92
 \pinlabel {$1$} at 377 92
 \pinlabel {$0$} at 404 92
 \pinlabel {$0$} at 450 92
\endlabellist
\centering
\includegraphics[width=\textwidth]{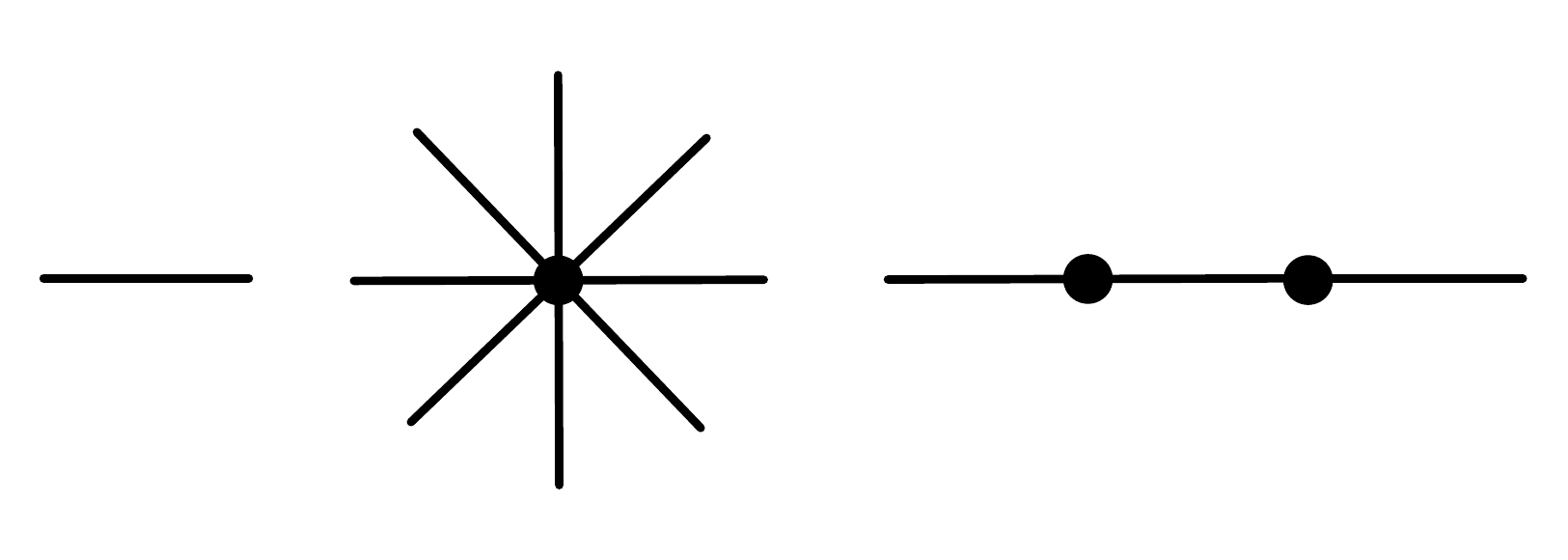}
	\caption{The trees $\eta=L_0$, $\cstar_8$, and $L_2$}\label{figure unrooted trees examples}
\end{figure}

\begin{remark}
A directed tree is a tree $S$ where each edge has an orientation (see Figure \ref{Theta figure}); another way to say this when $S\neq \eta$ is to say that there are partitions\footnote{
	Outside of this remark, we will only consider directed trees which are rooted, hence we will consider $\out(S)$ and $\out(v)$ as single edges, rather than the one element sets which contain them; cf., Definition~\ref{dendroidal category}.
}
\begin{align*}
	\nbhd(v) &= \out'(v) \amalg \inp(v) \\
	\legs(S) &= \out'(S) \amalg \inp(S) 
\end{align*}
so that $\out'(v) \cap \out'(w) = \varnothing = \inp(v) \cap \inp(w)$ for $v\neq w$ and $\out'(v) \cap \inp(S) = \varnothing = \inp(v) \cap \out'(S)$.
This description actually has a little bit more information floating around than we would like; namely a $(|\inp(S)|,|\out'(S)|)$-shuffle and, for each $v\in \vertex(S)$, a $(|\inp(v)|, |\out'(v)|)$-shuffle.
This is simply because for a directed graph we only need separate orderings on the inputs and outputs, not orderings on the entire neighborhoods.

Making a choice for the $(p,q)$-shuffles above, every directed tree determines a tree.
We will use the convention that the total order on $\nbhd(v)$ is determined by that on $\out'(v)$ and $\inp(v)$ by insisting, for $e \in \out'(v)$ and $e' \in \inp(v)$, that $e < e'$. Similarly, we get an order on $\legs(S)$ by saying $\out'(S) < \inp(S)$ (unless $S$ consists of a single edge).
This convention makes it so that for a rooted tree, the downward edge is always labeled by `$0$'.
This gives the map $\Ob(\Omega) \to \Ob(\Xi)$; 
we will actually define a variant of $\Omega$ in Definition \ref{dendroidal category}.
\end{remark}

\subsection{Morphisms of \texorpdfstring{$\Xi$}{Ξ}}

When discussing subgraphs of trees, we will always assume that they are nonempty, connected, and contain all edges incident to any of their vertices.

\begin{definition}
	A \emph{subgraph of a tree $S$} consists of 
	a pair of subsets
	\begin{align*}
		V &\subseteq \vertex(S)\\
		E &\subseteq \edge(S)
	\end{align*}
	so that
	\begin{itemize}
		\item if $v \in V$, then $\nbhd(v) \subseteq E$ (which means that $R = (V,E,\nbhd)$ constitutes the structure of an undirected graph without orderings),
		\item the underlying space of the graph $R = (V,E,\nbhd)$ is contractible.
	\end{itemize}
	Write $\sbgph(S)$ for the set of subgraphs of $S$.
\end{definition}

\begin{remark}
	Subgraphs of $S$ are naturally \emph{unpinned} trees.
	The orderings $\ord^v$ at each vertex $v$ are inherited from those in $S$.
\end{remark}

\begin{example}\label{edge subgraph}
	Each edge $e\in S$ constitutes a subgraph with $E = \{ e \}$ and $V = \varnothing$. We will write this subgraph as $|_e$.
\end{example}

\begin{example}\label{star subgraph}
	For each $v\in \vertex(S)$, there is a subgraph $\cstar_v$ with $V = \{ v \}$ and $E = \nbhd(v)$.
	Thus we have an inclusion $\cstar : \vertex(S) \hookrightarrow \sbgph(S)$.
	Notice that $\cstar_v$ has a preferred ordering with
	\[
		\ord^v_{\cstar_v} = \ord_{\cstar_v} = \ord^v_{S} \colon \{ 0, 1, \dots, n \} \overset\cong\to \nbhd(v) = \legs(\cstar_v).
	\]
\end{example}

\begin{proposition}\label{unions of subgraphs}
	If $R$ and $R'$ are subgraphs of $S$ and $R \cap R' \neq \varnothing$, then $R\cup R'$ is also a subgraph of $S$.
\end{proposition}
\begin{proof}
Write $R = (V,E)$ and $R'=(V',E')$.
The first condition we need to check for $R \cup R' = (V\cup V', E \cup E')$ is immediate, and does not require the hypothesis.
The hypothesis $R\cap R' \neq \varnothing$ means $(V\cap V') \cup (E \cap E') \neq \varnothing$, which implies that the underlying space of $R \cup R'$ is connected (since it is the union of the underlying spaces of $R$ and $R'$).
Thus it is a connected subspace of a contractible graph, hence is contractible as well.
\end{proof}

\begin{definition}[Boundary of a subgraph]
	Suppose that $S$ is a tree.
	\begin{itemize}
		\item If $X$ is a set, let $\mathbf{M}X = \coprod_{n\geq 0} X^{\times n} / \Sigma_n$ be the free commutative unital monoid on $X$ (that is, the set of unordered lists of elements of $X$).
		\item There is a function $\boundarymap : \sbgph(S) \to \mathbf{M}(\edge(S))$ with 
		\[
			\boundarymap(R) = \begin{cases}
				e^2 & \text{ if $R = |_e$ } \\
				\prod\limits_{e\in \legs(R)} e & \text{ otherwise.}
			\end{cases}
		\]
		We say that $\boundarymap(R)$ is the \emph{boundary} of the subgraph $R$.
	\end{itemize}
\end{definition}

If $R, T \in \sbgph(S)$, then the graph $R\cap T$ is either empty or it is also in $\sbgph(S)$. Further, $R\cup T\in \sbgph(S)$ if and only if $R\cap T$ is nonempty.
We will say that $R$ and $T$ \emph{overlap} if $R\cap T$ is nonempty (equivalently, if $R\cup T$ is connected).

\begin{definition}\label{definition complete morphism}
	Suppose that $S$ and $R$ are two trees. A \emph{complete morphism} $R \to S$ consists of two functions
	\begin{itemize}
		\item $\alpha_0 : \edge(R) \to \edge(S)$
		\item $\alpha_1 : \sbgph(R) \to \sbgph(S)$
	\end{itemize}
	that satisfy the following conditions.
	\begin{enumerate}
		\item The equation $\boundarymap \circ \alpha_1 = (\mathbf{M} \alpha_0) \circ \boundarymap$ holds.
		\item If $T, T' \in \sbgph(R)$ overlap, then so do $\alpha_1(T)$ and $\alpha_1(T')$. Furthermore,
		\begin{enumerate}
			\item $\alpha_1(T\cap T') = \alpha_1(T) \cap \alpha_1(T')$ and
			\item $\alpha_1(T\cup T') = \alpha_1(T) \cup \alpha_1(T')$.
		\end{enumerate}
	\end{enumerate}
\end{definition}

The set $\sbgph(R)$ is actually a partial lattice and the second condition just states that $\alpha_1$ is a map of partial lattices.
As the properties above are closed under function composition, there is a category of trees whose morphisms are complete morphisms.

Notice in particular that the existence of the function $\medstar : \vertex(R) \hookrightarrow \sbgph(R)$ means that every complete morphism has an associated function $\vertex(R) \to \sbgph(S)$.

\begin{definition}\label{maps of xi}
Suppose that $R$ and $S$ are trees.
\begin{itemize}
\item A \emph{morphism} $\phi : R \to S$ is defined to be a pair of maps
\begin{align*}
	\phi_0 \colon \edge(R) &\to \edge(S) \\
	\phi_1 \colon \vertex(R) &\to \sbgph(S)
\end{align*}
satisfying the following:
\begin{enumerate}
	\item \label{maps of xi bivalent} If $v$ is not bivalent (that is, $|\nbhd(v)| \neq 2$), then $\phi_0|_{\nbhd(v)}$ is injective. 
	\item \label{maps of xi interchange} For each vertex $v$, $\phi_0(\nbhd(v)) = \legs(\phi_1(v))$ (as unordered sets).
	\item \label{maps of xi separation} $\vertex(\phi_1(v)) \cap \vertex(\phi_1(w)) = \varnothing$ for $v\neq w$.
\end{enumerate}
\item The \emph{identity map} $\id_R: R \to R$ is given by letting $(\id_R)_0 = \id_{\edge(R)}$ and letting $(\id_R)_1$ be the inclusion $\cstar \colon \vertex(R) \to \sbgph(R)$.
\item More generally, a morphism $\phi : R \to S$ is an \emph{isomorphism} if $\phi_0$ is a bijection and if $\phi_1$ factors through $\vertex(S)$ as $\phi_1 = \cstar \widetilde \phi_1$ with $\widetilde \phi_1$ is a bijection.
\[ \begin{tikzcd}
\vertex(R) \rar{\phi_1} \arrow[dr, dashed, "\widetilde \phi_1" swap, "\cong" description] & \sbgph(S) \\
& \vertex(S) \uar{\cstar}
\end{tikzcd} \]
\end{itemize}
\end{definition}

In Proposition~\ref{bijection complete and kernel}, we show that precomposition with $\medstar$ constitutes a bijection between complete morphisms and morphisms. 
We will also transfer the composition of complete morphisms back to morphisms in Definition \ref{xi composition}, after which the reader may wish to verify that this definition of isomorphism is correct from a categorical standpoint.

\begin{example}\label{example unique iso}
If $R = R'$ except for orderings, then 
there is a unique isomorphism $\phi : R \to R'$ with $\phi_0 = \id$. 
We first note that since Definition \ref{maps of xi} does not mention orderings, the pair $(\id_{\edge(R)}, \cstar)$ constitutes a morphism $R \to R'$ of $\Xi$.

We now show that the only automorphism $\phi$ of $R$ which fixes the edges is $\id_R$. For this, we induct on the number of vertices $n$ of $R$; the case $n=0$ is clear since there is only one map $\eta \to \eta$.
Suppose that uniqueness has been established for all $m < n$; pick any $e_0\in \legs(R)$.
There is a unique $v_0 \in \vertex(R)$ with $e_0\in \nbhd(v_0)$.
By assumption $\phi_0(e_0) = e_0$, hence $\phi_1(v_0) = \cstar_{v_0}$. 
For each $e\in \nbhd(v_0) \setminus e_0$, there is a subgraph $R_e$ of $R$ consisting of all vertices and edges on all paths not containing $v$ but beginning at $e$.
By the induction hypothesis, $\phi|_{R_e} = \id_{R_e}$ is uniquely determined by the fact that $\phi_0$ fixes the edges. It follows that $\phi_1(v) = \cstar_v$ for all $v\in \vertex(R)$, so $\phi = \id_R$.
\end{example}

\begin{remark}
The argument for uniqueness in the previous example fails if we allow graphs without legs. Indeed, the graph \textbullet\!---\!\textbullet \, admits two distinct automorphisms $\phi$ with $\phi_0 = \id$.
\end{remark}

\begin{definition}[Cofaces and codegeneracies]\label{cofaces and codegens} We describe three basic types of morphisms of $\Xi$. Throughout, $S$ and $R$ will denote objects of $\Xi$, and $d(R)$ will refer to the number of vertices of $R$.
	\begin{itemize}
	\item Suppose that $S$ is a subgraph of $R$ and $d(S) = d(R) - 1$. Then we say the inclusion $S \to R$ is an \emph{outer coface}.
	Up to orderings, the tree $S$ is obtained from $R$ by selecting a pair $(v,e)$ with $v \in \vertex(R)$ and $e\in \nbhd(v)$ such that $\nbhd(v) \setminus \{e\} \subseteq \legs(R)$, and then deleting $v$ and all of the legs in $\nbhd(v) \setminus \{e\}$. We will write $\delta^v : S \to R$ for such a coface map if $d(R) > 1$, or $i : \eta \to R$ for the map that hits $\ord(i)$ when $d(R) = 1$.
	\item 
	A map $\phi : S \to R$ is an \emph{inner coface} if $d(S) = d(R) - 1$ and there is a vertex $v_0$ so that $\phi_1(v_0)$ has exactly two vertices and $\phi_1(v)$ has exactly one vertex for $v\in \vertex(S) \setminus \{ v_0\}$. The subgraph $\phi_1(v_0)$ has exactly one inner edge $e$, and we will often write $\delta^e : S \to R$ for such an inner coface map. The tree $S$ is obtained from $R$ by contracting an inner edge.
	\item 
	A map $\phi : S \to R$ is a \emph{codegeneracy} if $d(S) = d(R) + 1$ and there is a vertex $v_0$ (necessarily with $|v_0| = 2$) so that $\phi_1(v_0)$ is an edge and $\phi_1(v)$ has exactly one vertex for $v\in \vertex(S) \setminus \{ v_0\}$.
\end{itemize}
A \emph{coface} is a map which is either an inner coface or an outer coface.
\end{definition}

An example of each type of map is given in Figure \ref{outer inner degen}.

\begin{figure}
\labellist
\small\hair 2pt
 \pinlabel {$u$} at 40 170
 \pinlabel {$v$} at 94 170
 \pinlabel {$w$} at 157 170
 \pinlabel {$v$} at 289 274
 \pinlabel {$w$} at 352 274
 \pinlabel {$w$} at 352 168
 \pinlabel {$u$} at 235 59
 \pinlabel {$v$} at 289 59
 \pinlabel {$x$} at 328 59
 \pinlabel {$w$} at 369 59
 \pinlabel {$a$} at 26 139
 \pinlabel {$b$} at 14 192
 \pinlabel {$c$} at 68 163
 \pinlabel {$d$} at 128 180
 \pinlabel {$e$} at 86 199
 \pinlabel {$f$} at 101 140
 \pinlabel {$c$} at 260 267
 \pinlabel {$d$} at 323 283
 \pinlabel {$e$} at 282 304
 \pinlabel {$f$} at 296 247
 \pinlabel {$a$} at 266 152
 \pinlabel {$b$} at 262 191
 \pinlabel {$e$} at 296 199
 \pinlabel {$f$} at 296 136
 \pinlabel {$d$} at 323 177
 \pinlabel {$a$} at 225 31
 \pinlabel {$b$} at 208 81
 \pinlabel {$c$} at 263 52
 \pinlabel {$d'$} at 309 68
 \pinlabel {$d''$} at 348 68
 \pinlabel {$e$} at 282 93
 \pinlabel {$f$} at 296 27
 \pinlabel {{\color{blue}\huge$\delta^u$}} at 183 288
 \pinlabel {{\color{blue}\huge$\delta^c$}} at 220 189
\endlabellist
\centering
\includegraphics[width=0.8\textwidth]{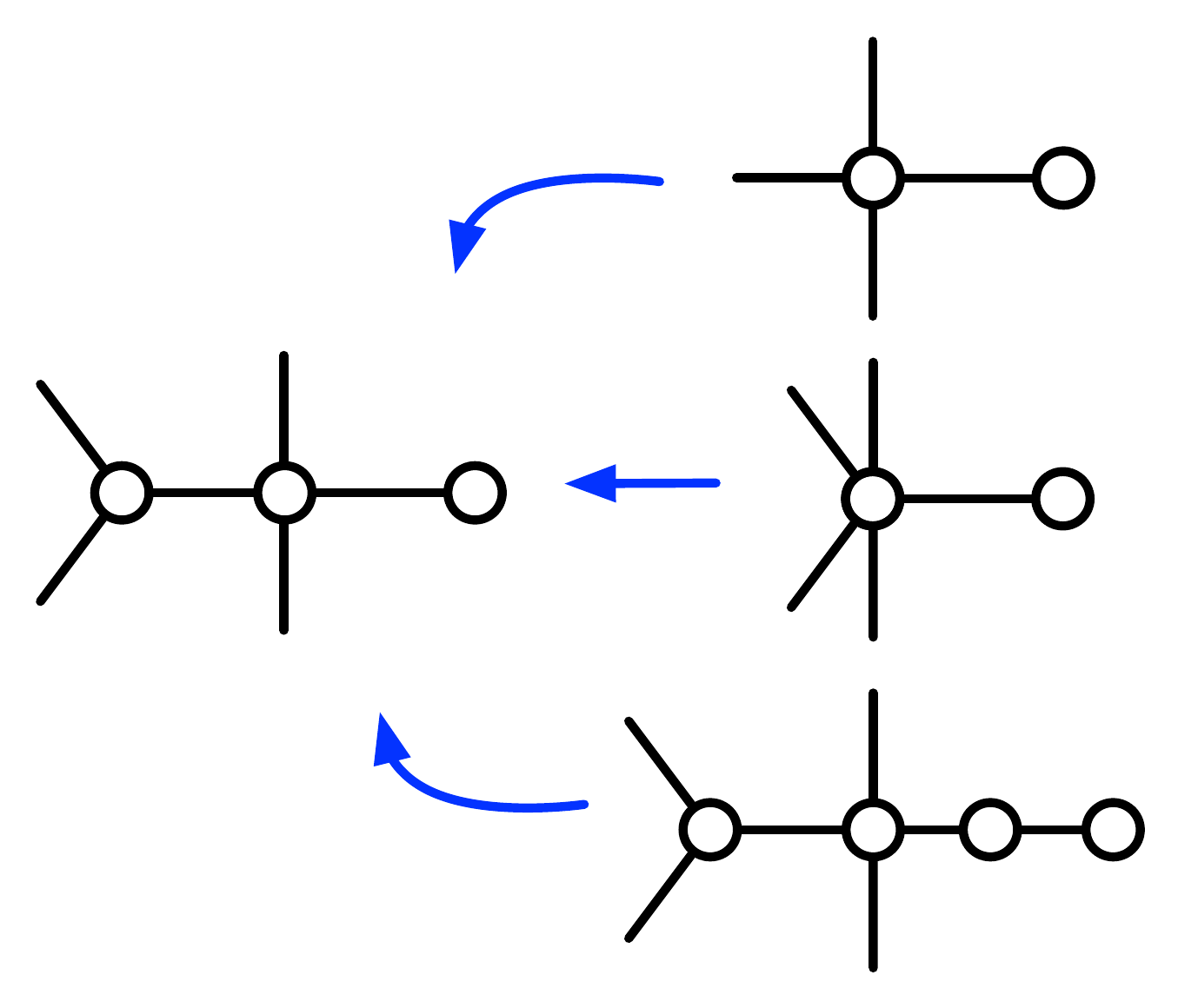}
	\caption{An outer coface, an inner coface, and a codegeneracy}\label{outer inner degen}
\end{figure}
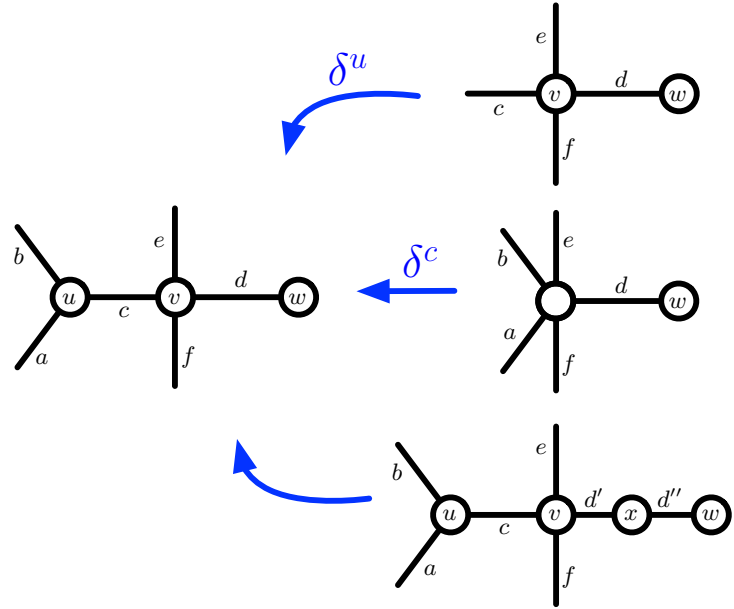

A \emph{path} in a graph $G$ is an alternating word in the alphabet $\edge(G)\sqcup \vertex(G)$ which can only contain a subword $ve$ or $ev$ if $e\in \nbhd(v)$.
A path from a vertex $v$ to a vertex $w$ is a path of the form
\[
	P = ve_1v_1e_2\dots e_{n-1}v_{n-1}e_nw
\]
while a path from an edge $e$ to an edge $e'$ is a path of the form 
\[
	P = ev_1e_1v_2\dots v_{n-1}e_{n-1}v_ne'.
\]
The length of a path, denoted by $|P|$, is the length of the word.

We can concatenate paths $P$ and $P'$ if the last letter in $P$ is the first letter in $P'$ or if the last letter in $P$ is adjacent to the first letter in $P'$.
In the former case, we will remove the duplicate letter.

\begin{definition}
Let $G$ be a graph and $v,w \in \vertex(G)$.
Define the \emph{distance} from $v$ to $w$ by $d_G(v,w) = d(v,w) = \min_P \frac{|P| - 1}{2}$ 
where $P$ ranges over all paths in $G$ from $v$ to $w$. 
Any path realizing the distance is said to be a \emph{minimal path} (and such exists as long as $d(v,w)$ is defined).
Similarly, if $e,f \in \edge(G)$, one defines $d_G(e,f)$ and minimal paths between edges in $G$.
\end{definition}
If we consider the rwb graph $G_{rwb}$ associated to a graph $G$, then $\vertex(G)$ is the same as the set of black vertices of $G_{rwb}$ and $\edge(G)$ is the union of the sets of red and white vertices of $G_{rwb}$;
these distances are then half of the usual graph distance between vertices (see, e.g., \cite[\S 1.3]{diestel}) in $G_{rwb}$.

\begin{proposition}\label{minimal path prop}
Let $S$ be a tree.
If $v,w\in \vertex(S)$, then a minimal path from $v$ to $w$ exists and is unique. 
A similar statement applies to minimal paths between edges.
Minimal paths are characterized as those containing no repeated entries.
\end{proposition}
\begin{proof}
Every vertex and edge of $S$ is a vertex in the corresponding rwb graph $S_{rwb}$.
This statement is then \cite[Theorem 1.5.1]{diestel} applied to the tree $S_{rwb}$.
\end{proof}

Existence and uniqueness of minimal paths leads to the following result.
\begin{corollary}
	If $R$ is a subgraph of $S$, then $d_R = d_S|R$, i.e., the distance in $R$ is the restriction of the distance in $S$, for all edges and vertices in $R$. \qed
\end{corollary}

We now have the necessary tools to define the objects and morphisms of the dendroidal category $\Omega$. 
Though we do not make substantial use of $\Omega$ until Section \ref{section rooting}, we include this definition here, rather than after Definition~\ref{xi composition}, to indicate the usefulness of the notion of distance.

\begin{definition}[Dendroidal category]\label{dendroidal category} We now define (a variant of) $\Omega$ as a subcategory of $\Xi$.
\begin{itemize}
\item A \emph{rooted tree} is a tree $R$ satisfying the following condition:
Suppose $r_0 = \ord(0) \in \legs(R)$. If $v\in \vertex(R)$ and $k > 0$, then 
\[
	d(\ord^v(0), r_0) < d(\ord^v(k), r_0).
\]
If $R\neq \eta$, we set 
\begin{align*}
	\inp(v) &= \nbhd(v) \setminus \ord^v(0) & \out(v) &= \ord^v(0) \\
	\inp(R) &= \legs(R) \setminus \ord(0) & \out(R) &=\ord(0) = r_0
\end{align*}
while if $R=\eta$ we set $\inp(R) = \{ r_0 \}$ and $\out(R) = r_0$.
\item If $R$ and $S$ are rooted trees and $\phi : R \to S$ is a map in $\Xi$, we say that $\phi$ is \emph{oriented} if for each $v \in \vertex(R)$ and each $k > 0$,
\[
	d(\phi_0(\ord^v(0)), s_0) \leq d(\phi_0(\ord^v(k)), s_0).
\]
\item The category $\Omega$ is the subcategory of $\Xi$ whose objects are rooted trees and whose morphisms are the oriented maps between rooted trees.
\item We write $\iota : \Omega \to \Xi$ for the subcategory inclusion.
\end{itemize}
\end{definition}

Notice that if $\phi$ is an oriented map, then $\phi_1(v)$ is a rooted tree (without ordering of the leaves) with root $\phi_0(\ord^v(0)) = \phi_0(\out(v))$.

\begin{remark}
This definition of $\Omega$ is analogous to the equivalent category $\Omega'$ from \cite[Example 2.8]{bm}.
A rooted tree in our sense is equivalent to a rooted tree together with a planar structure and an ordering of the input edges, and morphisms do not need to preserve the planar structure.
\end{remark}

\begin{remark}
In the above definition we were able to \emph{recognize} rooted trees among all trees; we cannot do something similar for general directed trees (and hence for the category $\Theta$ from \cite[Remark 6.55]{hrybook}). Indeed, graphs which are linear as undirected graphs generally possess many directed structures, even controlling for the number of inputs and outputs. See Figure \ref{Theta figure}.
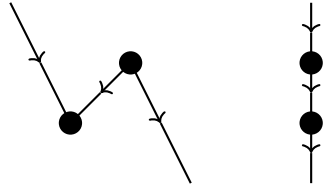
\begin{figure}
\begin{tikzpicture}[scale=2,thick]
\draw[->] (0,0) -- (.2,-.4);
\draw (.2,-.4) -- (.4,-.8); 
\draw[->] (.8,-.4) -- (.6, -.6);
\draw (.6, -.6) -- (.4,-.8);
\draw[->] (.8, -.4) -- (1, -.8);
\draw (1, -.8) -- (1.2, -1.2);
\filldraw 	(.4,-.8) circle [radius=2pt]
			(.8,-.4) circle [radius=2pt];
\draw[->] (2,0) -- (2,-.2);
\draw (2,-.2) -- (2,-.4);
\draw[->] (2,-.4) -- (2,-.6);
\draw (2,-.6) -- (2,-.8);
\draw[->] (2,-.8) -- (2,-1);
\draw (2,-1) -- (2,-1.2);
\filldraw 	(2,-.8) circle [radius=2pt]
			(2,-.4) circle [radius=2pt];
\end{tikzpicture}
\caption{Two different directed structures on the same undirected tree}\label{Theta figure}
\end{figure}
In short, there is a functor from (a legged-variant of) $\Theta$ to $\Xi$, but it is not injective on objects.
\end{remark}

\begin{lemma}\label{neighborhoods of vertices}
Let $\phi : R \to S$ be a morphism of $\Xi$. Then
	$\interior(\phi_1(v)) \cap \edge(\phi_1(w)) = \varnothing$ if $v\neq w$.
\end{lemma}
\begin{proof}
Suppose that $\phi_1(w)$ contains a vertex and $e\in \interior(\phi_1(v)) \cap \edge(\phi_1(w))$. Since $\phi_1(w) \neq |_e$, there is a vertex $w' \in \vertex(\phi_1(w))$ with $e\in \nbhd(w')$. Since $e \in \interior(\phi_1(v))$, we know that $w'\in \vertex(\phi_1(v))$, which implies $v=w$ by Definition \ref{maps of xi}\eqref{maps of xi separation}.

In the general case, we induct on $d(v,w)$. 
	Suppose we have vertices $v$ and $w$ with $d(v,w) = n > 0$ and let $ve_1v_1e_2v_2\dots e_{n-1}v_{n-1}e_nw$ be the shortest path from $v$ to $w$.
If $\phi_1(w)$ contains a vertex then we know $\interior(\phi_1(v)) \cap \edge(\phi_1(w)) = \varnothing$ by the first paragraph. 
	Suppose that $\phi_1(w) = |_e$ is a single edge. 
If $n=1$, then $e = \phi_0(e_1) \in \legs(\phi_1(v))$, which implies $e\notin \interior(\phi_1(v))$ by Definition~\ref{maps of xi}\eqref{maps of xi interchange}.
	Assume the statement of the lemma is true for vertices of distance equal to $n - 1$.
	Then $e = \phi_0(e_n) \in \legs(\phi_1(v_{n-1}))$, which implies that $e\notin \interior(\phi_1(v))$ by the induction hypothesis.
\end{proof}

\begin{lemma}\label{interiors disjoint from image}
Let $\phi : R \to S$ be a morphism of $\Xi$.
	For each vertex $v \in \vertex(R)$, $\interior(\phi_1(v)) \subseteq \edge(S) \setminus \image(\phi_0)$.
\end{lemma}
\begin{proof}
	Suppose that $\phi_0(e) \in \interior(\phi_1(v))$.
	Since the graph has a vertex $v$ and is connected, every edge is adjacent to at least one vertex. If $e$ is adjacent to $v$, then $\phi_0(e) \in \phi_0(\nbhd(v)) = \legs(\phi_1(v)) \subseteq \edge(S) \setminus \interior(\phi_1(v))$, so we conclude that $e$ is not adjacent to $v$.
	Thus there exists a $w\neq v$ with $e\in \nbhd(w)$.
	But now 
	\[ \phi_0(e) \in \interior(\phi_1(v)) \cap \legs(\phi_1(w)) \subseteq \interior(\phi_1(v)) \cap \edge(\phi_1(w)),\] 
	which is empty by Lemma \ref{neighborhoods of vertices}.
\end{proof}

\begin{lemma}\label{legs intersection}
Let $\phi : R \to S$ be a morphism of $\Xi$.
	If 
	\[ |\legs(\phi_1(v)) \cap \legs(\phi_1(w))| > 1,\] 
	then $v = w$.
\end{lemma}
\begin{proof}
Suppose $v\neq w$.
Let $e, e' \in \legs(\phi_1(v)) \cap \legs(\phi_1(w))$.
Let $P$ be the shortest path in $\phi_1(v)$ from $e$ to $e'$ and let $P'$ be the shortest path in $\phi_1(w)$ from $e$ to $e'$.
Since $P$ and $P'$ are also distance minimizing paths in $S$, uniqueness implies that $P= P'$.
If $e\neq e'$, this path contains a vertex, hence $\varnothing \neq \vertex(\phi_1(v)) \cap \vertex(\phi_1(w))$ and we see that $v=w$ by Definition \ref{maps of xi}\eqref{maps of xi separation}.
\end{proof}

\begin{lemma}\label{same value}
	Suppose that $\phi : R \to S$ is a morphism of $\Xi$.
	If $\phi_0(e) = \phi_0(e')$, then $e$ and $e'$ lie on a common linear subgraph (which may just mean $e=e'$), all of whose edges map to a common value.
\end{lemma}
\begin{proof}
	Induct on $d(e,e')$. If $d(e,e') = 0$ then $e = e'$ and the result follows.
	If $d(e,e') = 1$ and $\phi_0(e) = \phi_0(e')$, then the vertex adjacent to both $e$ and $e'$ must be bivalent by Definition~\ref{maps of xi}\eqref{maps of xi bivalent}.
	Assume the result is known for $d(e,e') < n$.
	Suppose $\phi_0(e) = \phi_0(e') = s \in \edge(S)$ with $d(e,e') = n > 1$.
	Let $e_0v_1e_1v_2\dots v_{n-1}e_{n-1}v_ne_n$ be the distance minimizing path in $R$ from $e = e_0$ to $e' = e_n$.

	For each $i$, let $P_i$ be the shortest path in $\phi_1(v_i)$ from $\phi_0(e_{i-1})$ to $\phi_0(e_i)$.
	The path $P_1 \dots P_{n-1}$ contains no repeated entries by Lemma \ref{neighborhoods of vertices} and Definition \ref{maps of xi}\eqref{maps of xi separation}, hence is the unique length minimizing path (Proposition \ref{minimal path prop}) from $s$ to $\phi_0(e_{n-1})$.
	Both of these edges are in $\phi_1(v_n)$, hence $P_1\dots P_{n-1}$ is a path in $\phi_1(v_n)$.
	If $s \neq \phi_0(e_{n-1})$, then $P_1 \dots P_{n-1}$ contains a vertex, violating Definition \ref{maps of xi}\eqref{maps of xi separation}. Thus $\phi_0(e_n) = s = \phi_0(e_{n-1})$, so $v_n$ is bivalent.
	The result now follows from the induction hypothesis since $d(e, e_{n-1}) < d(e, e')$.
\end{proof}

We will momentarily (in \ref{image definition}) define the \emph{image} of a map, which is essentially the union of all of the subgraphs $\phi_1(v)$.
We first check that this union actually is a subgraph.

\begin{proposition} Suppose that $\phi : R \to S$ is a morphism in $\Xi$ and $R \neq \eta$. Then 
\[ \bigcup_{v\in \vertex(R)} \phi_1(v)\] is a subgraph of $S$.
\end{proposition}
\begin{proof}
Suppose that $R$ contains a vertex.
Let $P = v_1 e_1 \dots v_{n-1} e_{n-1} v_n$ be a path in $R$ containing all vertices at least once. 
Then we have
\[
	\image(\phi) = \phi_1(v_1) \cup \phi_1(v_2) \cup \dots \cup \phi_1(v_n).
\]
Use induction. By Lemma \ref{unions of subgraphs} we know that
\[
	\Big( \phi_1(v_1) \cup \dots \cup \phi_1(v_{k}) \Big) \cup \phi_1(v_{k+1})
\]
is a subgraph since $\phi_1(v_1) \cup \dots \cup \phi_1(v_{k})$ and $\phi_1(v_{k+1})$ are (induction hypothesis) and 
\[
	\phi_0(e_k) \in \phi_1(v_{k}) \cap \phi_1(v_{k+1}).
\]

\end{proof}

\begin{definition}\label{image definition}
Let $\phi : R \to S$ be a map in $\Xi$.
Define the \emph{image} of $\phi$, denoted
$\image(\phi)$, to be the subgraph
\[
	\image(\phi) = \begin{cases} \, |_{\phi_0(e)} & R = \eta \\
	\bigcup_{v\in \vertex(R)} \phi_1(v) & \vertex(R) \neq \varnothing \end{cases}
\]
of $S$.
\end{definition}

\begin{proposition}\label{legs interchange}
 Suppose that $\phi : R \to S$ is in $\Xi$. Then $\phi_0(\legs(R)) = \legs(\image(\phi))$.
\end{proposition}
\begin{proof}
The desired identity is clear when $R = \eta$ is an edge.
We show that the desired equality holds when $R$ contains a vertex. If $s\in \edge(\image(\phi))$ and $\phi_0^{-1}(s) = \varnothing$, then $s\in \interior(\phi_1(v)) \subseteq \interior(\image(\phi))$ for some $v$, and also $s\notin \phi_0(\legs(R))$. 

Hence, for the remainder of the proof, we will only consider edges $s\in \edge(\image(\phi))$ so that $\phi_0^{-1}(s)$ is nonempty.
We will write $L_s$ for the linear subgraph of $R$, guaranteed by Lemma~\ref{same value}, with $\edge(L_s) = \phi_0^{-1}(s)$.
Note that if $\nbhd(v) \cap \edge(L_s)$ is not empty, then $v\notin \vertex(L_s)$ if and only if $\phi_1(v)$ is not an edge.

To show that $\legs(\image(\phi)) \subseteq \phi_0(\legs(R))$, we prove the equivalent statement: if $s\in \edge(\image(\phi))$ and $\phi_0^{-1}(s) \subseteq \interior(R)$, then $s\in \interior(\image(\phi))$. We have already established this when $\phi_0^{-1}(s) = \varnothing$.
If $\varnothing \neq \phi_0^{-1}(s) \subseteq \interior(R)$, then there exist distinct vertices $v_1,v_2 \in \vertex(R)$ with $v_i \notin \vertex(L_s)$ and $\nbhd(v_i) \cap \edge(L_s) \neq \varnothing$.
As $\phi_1(v_i)$ is not equal to $|_s$, it contains a vertex $w_i$ adjacent to $s\in \legs(\phi_1(v_i))$.
Since $\vertex(\phi_1(v_1)) \cap \vertex(\phi_1(v_2)) = \varnothing$, we know that $w_1\neq w_2$, hence $s \in \interior(\image(\phi))$.

Let us turn to the reverse inclusion $\phi_0(\legs(R)) \subseteq \legs(\image(\phi))$.
Suppose $r\in \edge(R)$ and suppose that $s= \phi_0(r) \in \interior(\image(\phi))$.
We must show $r\in \interior(R)$.
By assumption, there exist distinct vertices $w_1,w_2 \in \vertex(\image(\phi))$ with $s\in \nbhd(w_i)$.
There exist unique $v_1, v_2 \in \vertex(R)$ with $w_i \in \vertex(\phi_1(v_i))$. 
If $v_1$ were equal to $v_2$, then we would have $s\in \interior(\phi_1(v_1))$.
This is impossible by Lemma~\ref{interiors disjoint from image}, hence $v_1 \neq v_2$.
Now $\nbhd(v_i) \cap \edge(L_s) \neq \varnothing$ and $v_i \notin \vertex(L_s)$, hence we have $\edge(L_s) \subseteq \interior(R)$.
Since $r\in \edge(L_s)\subseteq \interior(R)$, we have completed our proof that $\phi_0^{-1}(\interior(\image(\phi)) \subseteq \interior(R)$.
\end{proof}

Notice that if $T\subseteq T'$ are subgraphs of $R$, then $\image(\phi|_T) \subseteq \image(\phi|_{T'})$.
Also, note that if $T,T'$ are two subgraphs of $R$ with $\vertex(T) = \vertex(T') \neq \varnothing$, then $T=T'$.

\begin{lemma}
\label{lemma intersection union}
	Let $\phi : R \to S$ be a morphism.
	Suppose that $T,T' \in \sbgph(R)$ are two subgraphs which overlap.
	Then the following hold:
\begin{align}
	\image(\phi|_{T\cap T'}) &= \image(\phi|_T) \cap \image(\phi|_{T'})\label{images intersection} \\
	\image(\phi|_{T\cup T'}) &= \image(\phi|_T) \cup \image(\phi|_{T'})\label{images union}.
\end{align}
\end{lemma}
\begin{proof}

	We start with the case when $T$ (or $T'$) is a single edge $|_e$. Since $T$ and $T'$ are assumed to overlap, $T\subseteq T'$. 
	Then $T\cap T' = |_e$ and $T\cup T' = T'$.
	Both sides of \eqref{images intersection} are $|_{\phi_0(e)}$, while both sides of \eqref{images union} are $\image(\phi|_{T'})$.

	It remains to prove the result when $T$ and $T'$ each contain at least one vertex. 
	For \eqref{images union}, we have 
	\begin{align*}
		\image(\phi|_T) \cup \image(\phi|_{T'}) &= \left(\bigcup_{v\in\vertex(T)} \phi_1(v)\right) \cup \left(\bigcup_{w\in \vertex(T')} \phi_1(w)\right) 
	\\ &= \bigcup_{v\in \vertex(T) \cup \vertex(T')} \phi_1(v) = \image(\phi|_{T\cup T'}).
	\end{align*}
	Thus \eqref{images union} holds.

	Since $T\cap T'$ is contained in $T$ (and in $T'$), it is automatic that $\subseteq$ of \eqref{images intersection} holds.
	We have \begin{align*}
		\vertex( \image(\phi|_T) ) \cap \vertex( \image(\phi|_{T'}) ) &= \left(\bigcup_{v\in\vertex(T)} \vertex(\phi_1(v))\right) \cap \left(\bigcup_{w\in \vertex(T')} \vertex(\phi_1(w))\right) \\
		&=\bigcup_{v\in \vertex(T) \cap \vertex(T')} \vertex(\phi_1(v))
	\end{align*}
	by Definition~\ref{maps of xi}\eqref{maps of xi separation}. This last set is of course just $\vertex(\image(\phi|_{T\cap T'}))$.
	We have two cases to consider:
	\begin{itemize}
		\item If $\vertex(\image(\phi|_{T\cap T'})) \neq \varnothing$, then \eqref{images intersection} holds since both sides are subgraphs that have the same non-empty set of vertices.
		\item If $\vertex(\image(\phi|_{T\cap T'})) = \varnothing$, then both sides of \eqref{images intersection} are a single edge.
		But we already saw that $\subseteq$ of \eqref{images intersection} holds, so \eqref{images intersection} holds.
	\end{itemize}
\end{proof}

\begin{lemma}
	If $(\alpha_0, \alpha_1)$ is a complete morphism (Definition~\ref{definition complete morphism}) from $R$ to $S$, then $(\alpha_0, \alpha_1 \circ \medstar)$ is a morphism (Definition~\ref{maps of xi}) from $R$ to $S$. 
\end{lemma}
\begin{proof}
For concision, we write $\check \alpha_1 := \alpha_1 \circ \medstar$ in this proof and the next.
Conditions \eqref{maps of xi bivalent} and \eqref{maps of xi interchange} of Definition~\ref{maps of xi} follow immediately from $\boundarymap \circ \alpha_1 = (\mathbf{M} \alpha_0) \circ \boundarymap$.
For \eqref{maps of xi separation}, suppose that $v\neq w$ and induct on $d(v,w)$. If $d(v,w) = 1$, then $\medstar_v \cap \medstar_w$ is an edge, hence $\alpha_1(\medstar_v \cap \medstar_w) = \alpha_1(\medstar_v) \cap \alpha_1(\medstar_w) = \check \alpha_1(v) \cap \check \alpha_1(w)$ is an edge, thus \eqref{maps of xi separation} holds. Assume the result is known for distances less than $n$, and suppose $ve_1v_1e_2\dots e_{n-1}v_{n-1}e_nw$ is a minimal path from $v$ to $w$.
Let $T_v = \medstar_v \cup \bigcup_{i=1}^{n-1} \medstar_{v_i}$ and $T_w = \medstar_w \cup \bigcup_{i=1}^{n-1} \medstar_{v_i}$; by the induction hypothesis, $\vertex(\check \alpha_1(w))$ and $\vertex(\check \alpha_1(v))$ are both disjoint from $\bigcup_{i=1}^{n-1} \vertex(\check \alpha_1(v_i))$.
On the other hand, $\vertex(\check \alpha_1(v)) \cap \vertex(\check \alpha_1(w)) \subseteq \vertex(\alpha_1(T_v)) \cap \vertex(\alpha_1(T_w)) = \vertex(\alpha_1(T_v \cap T_w)) = \bigcup_{i=1}^{n-1} \vertex(\check \alpha_1(v_i))$, proving the result. 
\end{proof}

\begin{proposition}\label{bijection complete and kernel}
	Morphisms and complete morphisms from $R$ to $S$ are in bijective correspondence.
	Precisely, the assignment that sends $(\alpha_0, \alpha_1)$ to $(\alpha_0, \alpha_1 \circ \medstar)$ constitutes a bijective function from the complete morphisms to the the morphisms.
\end{proposition}
\begin{proof}
We saw in the previous lemma that this function is well-defined. We now construct an inverse.
Suppose that $\phi = (\phi_0, \phi_1) \in \Xi(R,S)$.
If $T\in \sbgph(R)$, define $\hat \phi_1(T) = \image(\phi|_{T})$. 
Let us verify that $(\phi_0, \hat \phi_1)$ is a complete morphism.

If $T = |_e$, then $\boundarymap(T) = e^2$, so 
\[
	\boundarymap \hat \phi_1(|_e) = \boundarymap(|_{\phi_0(e)}) = (\phi_0(e))^2 = (\mathbf{M}\phi_0) e^2 = (\mathbf{M}\phi_0) \boundarymap (|_e).
\]
If $T$ contains a vertex, then $\boundarymap \hat \phi_1(T) = (\mathbf{M}\phi_0) \boundarymap (T)$ by Proposition~\ref{legs interchange} applied to $\phi|_T$.
Thus the first condition of Definition~\ref{definition complete morphism} holds for the pair $(\phi_0, \hat \phi_1).$ The second condition of this definition is guaranteed by Lemma~\ref{lemma intersection union}. 

We only need to check that $\phi \mapsto (\phi_0, \hat \phi_1)$ is inverse to $(\alpha_0, \alpha_1) \mapsto (\alpha_0, \check \alpha_1) = (\alpha_0, \alpha_1 \circ \medstar)$.
But 
\[
	\check{\hat{\phi}}_1(v) = \hat{\phi}_1(\medstar_v) = \image(\phi|_{\medstar_v}) = \phi_1(v)
\]
so this function is a right inverse.
It only remains to check that this function is a left inverse as well. 
If $T = |_e$ is an edge, we have
\[
	\hat{\check{\alpha}}_1(T) = \image((\Upsilon{\alpha})|_{T}) = \begin{cases} \, |_{\alpha_0(e)} & T = \eta \\
	\bigcup_{v\in \vertex(T)} \check{\alpha}_1(v) & \vertex(T) \neq \varnothing \end{cases}
\]
so if $T$ is an edge we are done because $\alpha_1(|_e) = |_{\alpha_0(e)}$.
If $T$ contains a vertex, then
\[
	\hat{\check{\alpha}}_1(T) = \bigcup_{v\in \vertex(T)} \alpha_1(\medstar_v) = \alpha_1\left(\bigcup_{v\in \vertex(T)} \medstar_v \right) = \alpha_1(T)
\]
since $\alpha_1$ preserves unions. Thus $\hat{\check{\alpha}}_1 = \alpha_1$.
\end{proof}

As the collection of trees together with complete morphisms obviously forms a category, there is then an induced operation $\circ$ so that graphs and morphisms form a category.
If one traces through the construction of the inverse in Proposition~\ref{bijection complete and kernel}, we see that this operation takes the following form.

\begin{definition}[Composition in $\Xi$]\label{xi composition}
Let $\phi : R \to S$ and $\psi: S \to T$ be morphisms of $\Xi$. Define two functions
\begin{align*}
	(\psi \circ \phi)_0 &\colon \edge(R) \to \edge(T) \\
	(\psi \circ \phi)_1 &\colon \vertex(R) \to \sbgph(T)
\end{align*}
by $(\psi \circ \phi)_0 = \psi_0 \circ \phi_0$ and $(\psi \circ \phi)_1(v) = \image(\psi|_{\phi_1(v)}).$
\end{definition}

\subsection{Cyclic dendroidal sets}\label{section cyclic dendroidal sets}

Define the category of \emph{cyclic dendroidal sets} to be the presheaf category $\Set^{\Xi^{op}}$.
Precomposition with the functor $\iota: \Omega \to \Xi$ induces a functor 
\[
	\iota^* : \Set^{\Xi^{op}} \to \Set^{\Omega^{op}}
\]
from cyclic dendroidal sets to dendroidal sets. This functor has both adjoints, though we will only need the left adjoint $\iota_! : \Set^{\Omega^{op}} \to \Set^{\Xi^{op}}$ in this paper.

Given a tree $S$, we will write $\Xi[S] = \hom_\Xi (-, S)$ for the object represented by $S$.

\begin{definition}\label{horns and cores} Let $S \in \Xi$ be a tree.
	\begin{itemize}
	\item Suppose that $\delta$ is a coface map (Definition \ref{cofaces and codegens}) with codomain $S$. Then the $\delta$-horn of $S$ is the subobject
	\[
		\Lambda^\delta\Xi[S] = \bigcup_{\substack{d : R \to S \\ \delta \not\cong d}} d_* \Xi[R] \subseteq \Xi[S]
	\]
	where the union is over all coface maps which are not isomorphic (over $S$) to $\delta$. This is an \emph{inner horn} if $\delta$ is an inner coface; otherwise it is an \emph{outer horn}.
	\item Suppose $S\neq \eta$. The \emph{Segal core} of $\Xi[S]$, denoted $\Sc[S]$, is defined to be the union
\[
	\bigcup_{v\in \vertex(S)} \Xi[\cstar_v] \subseteq \Xi[S],
\]
where $\cstar_v$ is regarded as a subobject of $S$. If $S = \eta$, it is convenient to also define $\Sc[\eta] = \Xi[\eta]$.
	\end{itemize}
Likewise, if $T\in \Omega$ is a rooted tree, we have horns $\Lambda^\delta\Omega[T] \subseteq \Omega[T]$ and Segal cores $\Sc[T] = \bigcup \Omega[\cstar_v] \subseteq \Omega[T]$.
\end{definition}

\section{Using rooting to orient maps in \texorpdfstring{$\Xi$}{Ξ}}\label{section rooting}

In this section we give a careful comparison of the morphism sets of $\Xi$ and $\Omega$.
Morally, the category $\Xi$ is built up from the dendroidal category $\Omega$ by adding isomorphisms which rotate trees.
Thus, every morphism of $\Xi$ should decompose into an oriented map (in $\Omega$) along with some rotation data.

In the present section we make this precise. 
To each tree $S$ and a choice of root $s_0 \in \legs(S)$, there is a rooted tree $\rootify(S,s_0) \in \Omega$ (see Definition \ref{treeificiation} and Figure \ref{rootings figure}).
Further, given a morphism $\phi : R \to S$ we can transform $\phi$ into an oriented map (Lemma \ref{amalgamate splitting})
\[
	\lifting(\phi) : \rootify(R,r_\phi) \to \rootify(S,s_0)
\]
for some particular choice of root $r_\phi \in \legs(R)$ (see Definition \ref{def find root}).
We show that $\lifting$ respects composition in a certain sense (Proposition \ref{functoriality of lifting}, Remark~\ref{remark rootify as functor}).
Finally, in Theorem \ref{structure from rooting} (see also Corollary \ref{structure from rooting nonlinear} and Corollary \ref{structure from rooting linear}) we realize our goal and make clear the idea that (non-constant) maps $R\to S$ are just certain maps in $\Omega$ along with rooting data for $S$.

\begin{definition}[Rooting of trees]\label{treeificiation}
Suppose we are given a pair $(S, s_0)$ with $S \in \Ob(\Xi)$ and $s_0 \in \legs(S)$. 
\begin{itemize}
\item Assume that $S\neq \eta$. We now define a rooted tree $T\in \Ob(\Omega)$ with $\edge(T) = \edge(S)$ and $\vertex(T) = \vertex(S)$.
For each $v\in \vertex(S)$, let $\ord^v(k_v) = \out(v) \in \nbhd(v)$ be the element which minimizes the function $d(-,s_0)|_{\nbhd(v)}$.
Setting $\inp(v) = \nbhd(v) \setminus \out(v)$, we have an induced ordering on $\inp(v)$ via
\[ \begin{tikzcd}
\{ 0, 1, \dots, n_v \} \rar{\ord^v}[swap]{\cong} & \nbhd(v) \dar[equal]\\
\{ 0, 1, \dots, n_v \} \uar{+ k_v \mod (n_v + 1)} & \{ \out(v) \} \amalg \inp(v)\\
\{ 1, \dots, n_v \} \uar[hook] \rar[dotted]{\cong} & \inp(v). \uar[hook]
\end{tikzcd} \]
Similarly, we have an induced ordering on $\inp(T) = \legs(S) \setminus \{s_0\}$ via
\[ \begin{tikzcd}
\{ 0, 1, \dots, n \} \rar{\ord}[swap]{\cong} & \legs(S) \dar[equal]\\
\{ 0, 1, \dots, n \} \uar{+ k \mod (n + 1)} & \{ s_0 \} \amalg \inp(T)\\
\{ 1, \dots, n \} \uar[hook] \rar[dotted]{\cong} & \inp(T), \uar[hook]
\end{tikzcd} \]
where $k = \ord^{-1}(s_0)$.
Although $\iota T$ may be different from $S$, since they have different total orderings (though the same \emph{cyclic} orderings \cite{MR669787}) $\ord$ and $\ord^v$, there is a \emph{unique} isomorphism $f : \iota T \to S$ of $\Xi$ with $f_0 = \id$ by Example \ref{example unique iso}. 

\item We write
\[
	\rootify(S,s_0) = (T, f: \iota T \overset\cong\to S)
\]
for this construction.
The second component is redundant (since we insist that $f_0 = \id$), so we will usually abuse notation and just write $\rootify(S,s_0) = T$.
Since the trivial tree $\eta$ is already rooted, we also set $\rootify(\eta, e) = \eta$.
\item 
If $r_0 \in \legs(R)$ and $s_0 \in \legs(S)$, define $\amalgamate_{r_0,s_0}$ to be the composite
\[
\begin{tikzcd}[column sep=large]
\Omega(\rootify(R,r_0), \rootify(S,s_0)) \rar[hook, "\iota"] & 
\Xi(\iota\rootify(R,r_0), \iota\rootify(S,s_0))
\rar["\cong"]
\arrow[r, phantom, shift right=4,"\scriptstyle{\Xi(f_{R}^{-1}, f_{S})}" swap]
& \Xi(R, S). 
\end{tikzcd} 
\]
\end{itemize}
\end{definition}

An example is given in Figure \ref{rootings figure}.

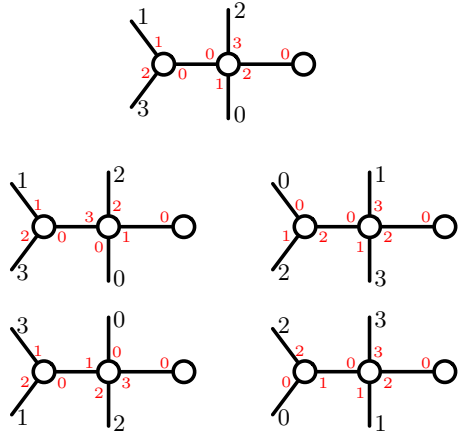
\begin{figure}
\labellist
\small\hair 2pt
 \pinlabel {$1$} at 124 357
 \pinlabel {$2$} at 205 362
 \pinlabel {$3$} at 124 280
 \pinlabel {$0$} at 205 275
 \pinlabel {{\tiny\color{red}$0$}} at 157 309
 \pinlabel {{\tiny\color{red}$1$}} at 137 336
 \pinlabel {{\tiny\color{red}$2$}} at 126 310
 \pinlabel {{\tiny\color{red}$0$}} at 180 326
 \pinlabel {{\tiny\color{red}$1$}} at 188 301
 \pinlabel {{\tiny\color{red}$2$}} at 211 309
 \pinlabel {{\tiny\color{red}$3$}} at 203 335
 \pinlabel {{\tiny\color{red}$0$}} at 243 326
 \pinlabel {$1$} at 23 220
 \pinlabel {$2$} at 104 225
 \pinlabel {$3$} at 23 143
 \pinlabel {$0$} at 104 138
 \pinlabel {$0$} at 242 220
 \pinlabel {$1$} at 323 225
 \pinlabel {$2$} at 242 143
 \pinlabel {$3$} at 323 138
 \pinlabel {$3$} at 23 98
 \pinlabel {$0$} at 104 103
 \pinlabel {$1$} at 23 21
 \pinlabel {$2$} at 104 16
 \pinlabel {$2$} at 242 98
 \pinlabel {$3$} at 323 103
 \pinlabel {$0$} at 242 21
 \pinlabel {$1$} at 323 16
 \pinlabel {{\tiny\color{red}$0$}} at  56  172
 \pinlabel {{\tiny\color{red}$1$}} at  36  199
 \pinlabel {{\tiny\color{red}$2$}} at  25  173
 \pinlabel {{\tiny\color{red}$3$}} at  79  189
 \pinlabel {{\tiny\color{red}$0$}} at  87  164
 \pinlabel {{\tiny\color{red}$1$}} at 110  172
 \pinlabel {{\tiny\color{red}$2$}} at 102  198
 \pinlabel {{\tiny\color{red}$0$}} at 142  189
 \pinlabel {{\tiny\color{red}$2$}} at 275  172
 \pinlabel {{\tiny\color{red}$0$}} at 255  199
 \pinlabel {{\tiny\color{red}$1$}} at 244  173
 \pinlabel {{\tiny\color{red}$0$}} at 298  189
 \pinlabel {{\tiny\color{red}$1$}} at 306  164
 \pinlabel {{\tiny\color{red}$2$}} at 329  172
 \pinlabel {{\tiny\color{red}$3$}} at 321  198
 \pinlabel {{\tiny\color{red}$0$}} at 361  189
 \pinlabel {{\tiny\color{red}$0$}} at  56  50
 \pinlabel {{\tiny\color{red}$1$}} at  36  77
 \pinlabel {{\tiny\color{red}$2$}} at  25  51
 \pinlabel {{\tiny\color{red}$1$}} at  79  67
 \pinlabel {{\tiny\color{red}$2$}} at  87  42
 \pinlabel {{\tiny\color{red}$3$}} at 110  50
 \pinlabel {{\tiny\color{red}$0$}} at 102  76
 \pinlabel {{\tiny\color{red}$0$}} at 142  67
 \pinlabel {{\tiny\color{red}$1$}} at 275  50
 \pinlabel {{\tiny\color{red}$2$}} at 255  77
 \pinlabel {{\tiny\color{red}$0$}} at 244  51
 \pinlabel {{\tiny\color{red}$0$}} at 298  67
 \pinlabel {{\tiny\color{red}$1$}} at 306  42
 \pinlabel {{\tiny\color{red}$2$}} at 329  50
 \pinlabel {{\tiny\color{red}$3$}} at 321  76
 \pinlabel {{\tiny\color{red}$0$}} at 361  67
\endlabellist
\centering
\includegraphics[width=0.5\textwidth]{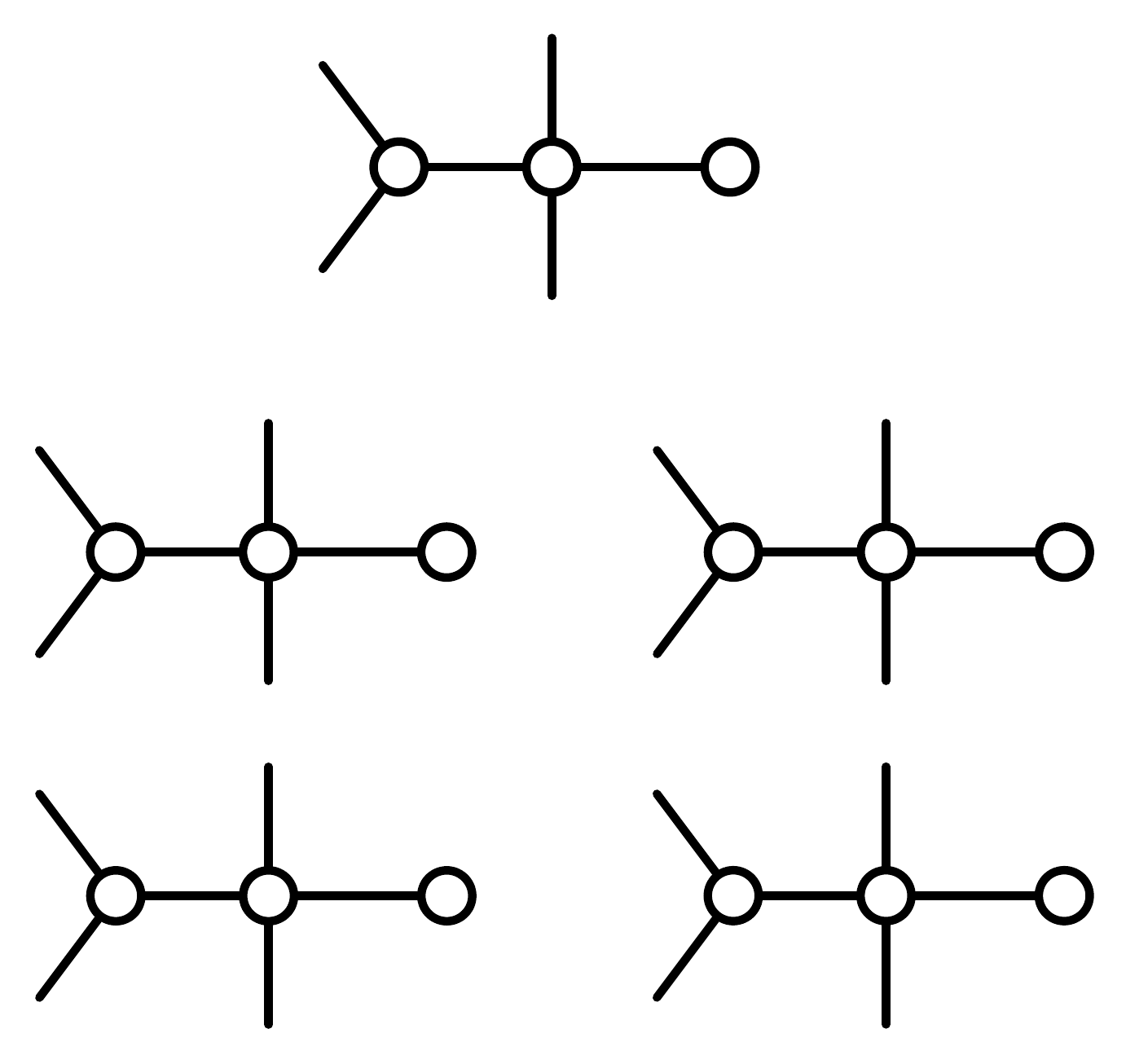}
	\caption{A tree $S$ together with its four associated rooted trees}\label{rootings figure}
\end{figure}

\begin{proposition}\label{functoriality of amalgamate}
Let $R, U, S \in \Ob(\Xi)$, $r_0 \in \legs(R)$, $u_0 \in \legs(U)$, and $s_0\in \legs(S)$. Then the diagram
\[ \begin{tikzcd}
\Omega(\rootify(U,u_0), \rootify(S,s_0)) \times \Omega(\rootify(R,r_0), \rootify(U,u_0)) \rar{\circ} \dar{\amalgamate_{u_0,s_0} \times \amalgamate_{r_0,u_0}}& \Omega(\rootify(R,r_0), \rootify(S,s_0))\dar{\amalgamate_{r_0,s_0}} \\
\Xi(U, S) \times \Xi(R, U) \rar{\circ} & \Xi(R, S)
\end{tikzcd} \]
commutes.
\end{proposition}
\begin{proof}
Let $g \in \Omega(\rootify(U,u_0), \rootify(S,s_0))$ and $h\in \Omega(\rootify(R,r_0), \rootify(U,u_0))$.
The diagram
\[ \begin{tikzcd}
\iota (\rootify(R,r_0)) 
  \rar{\iota (h)} 
  \dar{\cong}[swap]{f_R}
  \arrow[rr, bend left, "\iota (g\circ h)"]& 
\iota (\rootify(U,u_0)) 
  \dar{\cong}[swap]{f_U} 
  \rar{\iota (g)} & 
\iota (\rootify(S,s_0)) \dar{\cong}[swap]{f_S} \\
R \rar{\amalgamate_{r_0,u_0}(h)} & 
U \rar{\amalgamate_{u_0,s_0}(g)} & S
\end{tikzcd} \]
commutes,
hence
\[
	\amalgamate_{u_0,s_0}(g) \circ \amalgamate_{r_0,u_0}(h) = f_S \circ \iota(g\circ h) \circ f_R^{-1} = \amalgamate_{r_0,s_0}(g \circ h).
\]
\end{proof}

\begin{lemma}\label{subgraph distance addititvity}
	Suppose that $R$ is a subgraph of $S$, $s_0 \in \legs(S)$, and $r_0 \in \edge(R)$ minimizes $d(-,s_0)|_{\edge(R)}$. Then $r_0 \in \legs(R)$ and 
	\[
		d(r,s_0) = d(r,r_0) + d(r_0,s_0)
	\]
	for any $r \in \edge(R)$. In particular, the minimizing element $r_0$ is unique.
\end{lemma}
\begin{proof}
Let $P$ be the minimal path from $r_0$ to $s_0$.
Then $P$ contains no edges of $R$ other than $r_0$ by assumption.
Let $P'$ be the minimal path from $r$ to $r_0$.
Since $P'$ is actually a path in $R$, we have that $P$ and $P'$ only have the single edge $r_0$ in common.
Thus $P'P$ has no repeated edges, and thus is a minimal path from $r$ to $s_0$ by the characterization in Proposition \ref{minimal path prop}.
Thus $|P'P|-1 = (|P'|-1) + (|P|-1)$ which is the claimed equality.
\end{proof}

Given $R, S \in \Ob(\Xi)$, we write
$\Xi^0(R,S) \subseteq \Xi(R,S)$ for the set of maps which factor through the vertex-free graph $\eta$. 
If $R$ is linear then $\Xi^0(R,S) \cong \edge(S)$, otherwise $\Xi^0(R,S) = \varnothing$.
A particular case of Lemma~\ref{same value} says that if $\phi : R \to S$ is a map in $\Xi$ with $\phi_0|_{\legs(R)}$ not injective, then $\phi \in \Xi^0(R,S)$.
Combining this fact with $\phi_0(\legs(R)) = \legs(\image(\phi))$ (Proposition~\ref{legs interchange}) and Lemma~\ref{subgraph distance addititvity} applied to $\image(\phi) \hookrightarrow S$ gives that the following function is well-defined.

\begin{definition}[`Find root' function]\label{def find root}
Let $s_0 \in \legs(S)$.
Define a function $\findroot_{s_0}$
\[ \begin{tikzcd}[row sep = tiny]
		\Xi(R,S) \setminus \Xi^0(R,S) \rar{\findroot_{s_0}} &  \legs(R) \\
		\phi \rar[mapsto] & r_\phi,
\end{tikzcd}\]
where $r_\phi$ minimizes the function $d(\phi_0(-), s_0)|_{\legs(R)}$.
\end{definition}

Lemma~\ref{subgraph distance addititvity} also guarantees that $\findroot_{s_0}(\phi)$ minimizes the function $d(\phi_0(-),s_0)$, though of course there may be internal edges which also minimize this function since $\phi_0$ need not be injective.

\begin{remark}
	If $\phi \in \Xi^0(R,S)$, then $R$ must be a linear tree. The function $d(\phi_0(-),s_0)$ is \emph{constant}, hence is minimized by each of the two extremal edges when $R \cong L_n$ for $n > 0$.
\end{remark}

We now turn to several observations which give us effective tools for computing the function $\findroot_{s_0}$, in particular with respect to subgraph inclusions and certain compositions.

\begin{remark}\label{image find root}
	Let $\phi : R \to S$ be a map in $\Xi$ which does not factor through $\eta$. If $s_0 \in \legs(S)$,
	then
	$\phi_0(\findroot_{s_0}(\phi)) = \findroot_{s_0}(\image(\phi) \hookrightarrow S).$	
\end{remark}

In a similar vein, we have the following lemma.

\begin{lemma}\label{yet another findroot lemma}
	Let $\phi : R \to S \neq \eta$ be a morphism of $\Xi$ with $\image(\phi) = S$.
	Suppose $s_0 \in \legs(S)$ and write $r_0 = \findroot_{s_0}(\phi)$ (i.e., $\phi_0^{-1}(s_0) \cap \legs(R) = \{ r_0\})$.
	If $v\in \vertex(R)$ is such that $\phi_1(v)$ is not an edge, then 
	\[
		\phi_0\Big(\findroot_{s_0}\Big(\cstar_v \hookrightarrow R \overset\phi\to S \Big) \Big) = \phi_0 (\findroot_{r_0}(\cstar_v \hookrightarrow R)).
	\]
	In particular,
	\[
		\findroot_{s_0}\Big(\cstar_v \hookrightarrow R \overset\phi\to S \Big)  = \findroot_{r_0}(\cstar_v \hookrightarrow R).
	\]
\end{lemma}
\begin{proof}
The second statement follows from the first since $\phi_0|_{\nbhd(v)}$ is injective whenever $\phi_1(v)$ is not an edge.

Suppose that the elements $e_0 = \findroot_{s_0}(\cstar_v \to R \to S)$ and $e_1 = \findroot_{r_0}(\cstar_v \to R)$ are distinct.
Let $e_1 v_2 e_2 \dots e_{n-1} v_n e_n$ be the minimal path from $e_1$ to $e_n = r_0 = \findroot_{s_0}(\phi)$.
Since $d(e_1,r_0) < d(e_0, r_0)$, the path $e_0 v_1 e_1 v_2 \dots v_n e_n$, where $v_1 = v$, is a minimal path from $e_0$ to $e_n$.
For $i = 1, \dots, n$, let $P_i$ be the minimal path from $\phi_0(e_{i-1})$ to $\phi_0(e_i)$; the path $P_i$ is a path in $\phi_1(v_i)$.
Since the $v_i$ are distinct, $P_1 P_2 \dots P_n$ contains no repeated entries, hence is the minimal path from $\phi_0(e_0)$ to $\phi_0(r_0) = s_0$.
Likewise, $P_2 \dots P_n$ is the minimal path from $\phi_0(e_1)$ to $s_0$. Since $d(\phi_0(e_0), s_0)) \leq d(\phi_0(e_1), s_0)$, the paths $P_1P_2 \dots P_n$ and $P_2\dots P_n$ are equal. This implies that the path $P_1$ from $\phi_0(e_0)$ to $\phi_0(e_1)$ does not contain a vertex, hence $\phi_0(e_0) = \phi_0(e_1)$.
\end{proof}

\begin{lemma}\label{nested subgraphs lemma}
Suppose that $A\neq \eta$ and that
\[
	A \hookrightarrow B \hookrightarrow C
\]
is a pair of subgraph inclusions. If $c_0 \in \legs(C)$ and $b_0 = \findroot_{c_0}(B \hookrightarrow C) \in \legs(B)$,
then $\findroot_{b_0}(A\hookrightarrow B) = \findroot_{c_0}(A \hookrightarrow C)$.
\end{lemma}
\begin{proof}
Write $a_0 = \findroot_{b_0}(A\hookrightarrow B)$ and apply Lemma \ref{subgraph distance addititvity} twice to get 
\begin{align}
	d(a, b_0) &= d(a,a_0) + d(a_0,b_0) \label{line onzess} \\
	d(a, c_0) &= d(a,b_0) + d(b_0,c_0)\label{line thrzzss}
\end{align}
for any $a\in \edge(A) \subseteq \edge(B)$. 
For the particular case when $a = a_0$, \eqref{line thrzzss} becomes $d(a_0,c_0) = d(a_0, b_0) + d(b_0, c_0)$. Combining with \eqref{line onzess} we have
\[
	d(a, b_0) = d(a, a_0) + d(a_0, c_0) - d(b_0,c_0),
\]
hence
\begin{equation*}
	d(a, c_0) 
	= d(a, b_0) + d(b_0, c_0) = d(a, a_0) + d(a_0, c_0).
\end{equation*}
Then $a_0$ is the element of $\edge(A)$ which minimizes this function, hence
\[a_0 = \findroot_{c_0}(A \hookrightarrow C). \]
\end{proof}

\subsection{Orientation of maps}

\begin{lemma}\label{amalgamate splitting}
Let $r_0 \in \legs(R)$ and $s_0 \in \legs(S)$. 
There is a (unique) function $\lifting = \lifting_{r_0,s_0}$ so that the following diagram commutes.
\[ \begin{tikzcd}
\findroot_{s_0}^{-1}(r_0) \dar[hook]  \rar{\lifting}
&
\Omega(\rootify(R,r_0), \rootify(S,s_0))
\dar{\amalgamate_{r_0,s_0}}
\\
\Xi(R,S) \setminus \Xi^0(R,S) \rar[hook] & \Xi(R,S)
\end{tikzcd} \]
\end{lemma}
If $\phi \in \Xi(R,S) \setminus \Xi^0(R,S)$, we will use $\lifting_{s_0}(\phi)$ as shorthand for $\lifting_{\findroot_{s_0}(\phi), s_0}(\phi)$.
\begin{proof}
Let $\phi : R\to S$ be in $\findroot_{s_0}^{-1}(r_0)$.
To see that the functions 
\begin{align*}
	\phi_0 : \edge(R) &\to \edge(S) \\
	\phi_1 : \vertex(R) &\to \sbgph(S)
\end{align*}
determine a map $\rootify(R,r_0) \to \rootify(S,s_0)$ in $\Omega$, we just need to establish that, for each $v\in \vertex(R)$, we have $\phi_0(\out(v)) = \out(\phi_1(v))$. 
There is nothing to prove in the case $\phi_1(v)$ is a single edge.
For concision, write $\bar r_0$ for the element (see Remark~\ref{image find root})
\[ \bar r_0 := \phi_0(r_0) = \phi_0(\findroot_{s_0}(\phi)) = \findroot_{s_0}(\image(\phi) \hookrightarrow S) \in \legs(\image(\phi)) = \phi_0(\legs(R)). \]
In the case when $\phi_1(v)$ is not an edge, we have
\begin{align*}
	\out(\phi_1(v)) &= \findroot_{s_0}(\phi_1(v) \hookrightarrow S) \\
	&= \findroot_{\bar r_0}(\phi_1(v) \hookrightarrow \image(\phi)) & \text{Lemma \ref{nested subgraphs lemma}} \\
	&= \phi_0(\findroot_{\bar r_0}(\cstar_v \hookrightarrow R \to \image(\phi)) & \text{Remark \ref{image find root}} \\
	&= \phi_0(\findroot_{r_0} (\cstar_v \hookrightarrow R)) & \text{Lemma \ref{yet another findroot lemma}} \\
	&= \phi_0(\out(v)). & \text{Definition~\ref{treeificiation}}
\end{align*}
\end{proof}

\begin{lemma}\label{amalgamate different summands}
	If $r_0 \neq r_1 \in \legs(R)$, then
	\[
		\amalgamate_{r_0,s_0} \Big( \Omega(\rootify(R,r_0), \rootify(S,s_0)) \Big) \cap \amalgamate_{r_1,s_0} \Big( \Omega(\rootify(R,r_1), \rootify(S,s_0)) \Big) \subseteq \Xi^0(R,S).
	\]
\end{lemma}
\begin{proof}
If $\phi \in \amalgamate_{r_i,s_0} \Big( \Omega(\rootify(R,r_i), \rootify(S,s_0)) \Big) \setminus \Xi^0(R,S)$, then $\findroot_{s_0}(\phi) = r_i$.
\end{proof}

The following proposition describes precisely how the functions $\lifting_{r_0,s_0}$ behave with respect to composition in $\Xi$.
Special cases of the first part have already appeared in Lemma \ref{yet another findroot lemma} and Lemma \ref{nested subgraphs lemma}.

\begin{proposition}\label{functoriality of lifting}
Consider a composite
\[
	R \overset\psi\to U \overset\phi\to S
\]
in $\Xi$, and let $s_0 \in \legs(S)$.
Suppose that $\phi \circ \psi \notin \Xi^0(R,S)$.
If $u_0 = \findroot_{s_0}(\phi)$, then
\[
	\findroot_{u_0}(\psi) = \findroot_{s_0}(\phi \circ \psi)
\]
and 
\[
	\lifting_{s_0}(\phi) \circ \lifting_{u_0}(\psi) = \lifting_{s_0}(\phi \circ \psi).
\]
\end{proposition}
\begin{proof}
Set $u_0 = \findroot_{s_0}(\phi)$ and $r_0 = \findroot_{u_0}(\psi)$. Omitting the subsccripts, we know that the diagram 
\[ \begin{tikzcd}
\findroot_{s_0}^{-1}(u_0) \times \findroot_{u_0}^{-1} (r_0) \arrow[dr, bend left=15] \arrow[d, "\lifting \times \lifting" swap]\\
\Omega(\rootify(U,u_0), \rootify(S,s_0)) \times \Omega(\rootify(R,r_0), \rootify(U,u_0)) \dar{\circ} \arrow{r}{\amalgamate \times \amalgamate} & 
\Xi(U, S) \times \Xi(R, U) \dar{\circ} \\
\Omega(\rootify(R,r_0), \rootify(S,s_0))\arrow{r}{\amalgamate}  & \Xi(R, S)
\end{tikzcd} \]
commutes by Proposition \ref{functoriality of amalgamate} and Lemma \ref{amalgamate splitting}. 
Further, $(\phi, \psi)$ is an element in the apex.
Thus \[ \phi \circ \psi \in \amalgamate_{r_0,s_0}\Big( \Omega(\rootify(R,r_0), \rootify(S,s_0))\Big),\] so $\findroot_{s_0} (\phi \circ \psi) = r_0$.
Thus we have established the first statement.

The second statement is now immediate since the underlying maps of $\lifting_{s_0}(\phi)$, $\lifting_{u_0}(\psi)$, and $\lifting_{s_0}(\phi \circ \psi)$ are just $\phi_i$, $\psi_i$, and $(\phi \circ \psi)_i$ by the proof of Lemma \ref{amalgamate splitting}.
\end{proof}

\begin{remark}\label{remark rootify as functor}
	Let $S\in \Xi$, and let $\mathcal C \subseteq \Xi \downarrow S$ be the full subcategory with object set
	\[
		\coprod_{R\in\Xi} \Xi(R,S) \setminus \Xi^0(R,S).
	\]
	A restatement of the last part of the proof of Proposition~\ref{functoriality of lifting} is that for each $s_0\in S$, there is a functor $\mathcal C \to \Omega \downarrow \rootify(S,s_0)$ which on objects sends $\phi$ to $\lifting_{s_0}(\phi)$.
	It is possible to extend this functor to the larger full subcategory $\mathcal C \subseteq \mathcal C' \subseteq \Xi \downarrow S$ which includes the objects $\Xi(\eta, S)$, so that $\mathcal C' \to \Omega \downarrow \rootify(S,s_0) \overset{\iota}\to \Xi \downarrow S$ is isomorphic to the inclusion. It is not generally possible to extend the functor to all of $\Xi \downarrow S$.
\end{remark}

If $T$ and $T'$ are rooted trees, let $\Omega^0(T,T') \subseteq \Omega(T,T')$ denote the subset of oriented maps which factor through $\eta$. Notice that every morphism of $\Xi$ from $T$ to $T'$ that factors through $\eta$ is automatically oriented, so $\Omega^0(T,T') = \Xi^0(T,T')$; we make the distinction in notation only for emphasis.

For the remainder of the section, if $s_0 \in \legs(S)$, we will write 
\[
\amalgamate_{s_0}: \coprod\limits_{r \in \legs(R)} \Omega(\rootify(R,r), \rootify(S,s_0)) \to \Xi(R, S) \]
for the coproduct of the $\amalgamate_{r,s_0}$.

\begin{theorem}\label{structure from rooting}
Suppose that $s_0 \in \legs(S)$.
The function $\amalgamate_{s_0}$
restricts to a bijection 
\begin{align*}
	\amalgamate_{s_0}^{nc} \colon \coprod_{\legs(R)} \Big( \Omega(\rootify(R,r), \rootify(S,s_0)) \setminus \Omega^0(\rootify(R,r), \rootify(S,s_0)) \Big) & \to  \Xi(R,S) \setminus \Xi^0(R,S) \\
	&=  \coprod_{\legs(R)} \findroot_{s_0}^{-1}(r).
\end{align*}
\end{theorem}
\begin{proof}
There is a diagram 
\[ \begin{tikzcd}
 \coprod\limits_{\legs(R)} \Big( \Omega(\rootify(R,r), \rootify(S,s_0)) \setminus \Omega^0(\rootify(R,r), \rootify(S,s_0)) \Big)  \rar{\amalgamate_{s_0}^{nc}}\dar[hook]  &
 \Xi(R,S) \setminus \Xi^0(R,S) \arrow[dl, "\lifting_{s_0}" description] \dar[hook]
 \\
 \coprod\limits_{\legs(R)} \Omega(\rootify(R,r), \rootify(S,s_0)) \rar{\amalgamate_{s_0}} 
 & \Xi(R,S).
\end{tikzcd} \]
The bottom triangle commutes by Lemma \ref{amalgamate splitting}. Given $\psi : R \to S$, $\lifting_{s_0}(\psi) : \rootify(R,r_0) \to \rootify(S,s_0)$ is the unique map of $\Omega$ so that $\amalgamate_{s_0} (\lifting_{s_0}(\psi)) = \psi$. If $\psi = \amalgamate_{s_0}(\phi : \rootify(R,r_0) \to \rootify(S,s_0))$, then certainly $\phi$ satisfies the condition to be $\lifting_{s_0}(\amalgamate_{s_0}(\phi))$. Thus the top triangle commutes.

Since the left vertical map is injective, so is $\amalgamate_{s_0}^{nc}$. Further, if $\phi : R \to S$ is not constant, then 
\[
	\lifting_{s_0}(\phi) \in \Omega(\rootify(R, \findroot_{s_0}(\phi)), \rootify(S,s_0))
\]
is also not constant, so $\amalgamate_{s_0}^{nc}$ is surjective.
\end{proof}

\begin{corollary}\label{iso decompositions}
	We have, for each $s_0 \in \legs(S)$
	\[
		\Iso_{\Xi}(R,S) \cong \coprod_{\legs(R)} \Iso_{\Omega}(\rootify(R,r), \rootify(S,s_0))
	\]
	via $\amalgamate_{s_0}$.
	Specializing to the case $R = S$, we have
	\[
		\Aut_{\Xi}(S) \cong \coprod_{s\in\legs(S)} \Iso_{\Omega}(\rootify(S,s), \rootify(S,s_0)).
	\]
\end{corollary}
\begin{proof}
	The statement is trivial if $S$ (and hence $R$) does not have a vertex.
	Otherwise, isomorphisms are not constant, so this follows from Theorem \ref{structure from rooting} by taking subsets.
\end{proof}

\begin{example}
	Consider the tree $S$ with $\edge(S) = \legs(S) = \{ e_1, e_2, e_3 \}$ and $\vertex(S) = \{ v\}$.
We can apply Corollary \ref{iso decompositions} to reveal some of the structure of $\Aut_{\Xi}(S) = \Sigma_3$, but $\amalgamate_{e_i}$ as $i$ varies do not behave well together.
For example, under the composite
\[ \begin{tikzcd}
\Iso_\Omega\Big(\rootify(S,e_1), \rootify(S,e_2)\Big) \dar[hook] \rar[equal] & \Big\{ \phi, \psi \Big\} \\
\coprod\limits_{i=1}^3 \Iso_\Omega\Big(\rootify(S,e_i), \rootify(S,e_2)\Big) \rar{\amalgamate_{e_2}}[swap]{\cong}
& 
\Aut_{\Xi}(S) \rar{\amalgamate_{e_1}^{-1}}[swap]{\cong}
& 
\coprod\limits_{j=1}^3 \Iso_\Omega\Big(\rootify(S,e_j), \rootify(S,e_1)\Big)
\end{tikzcd} \]
(where, say, $\phi_0(e_1) = e_2, \phi_0(e_2) = e_1, \phi_0(e_3) = e_3, \psi_0(e_1) = e_2, \psi_0(e_2) = e_3, \psi_0(e_3) = e_1$) $\phi$ and $\psi$ map to different coproduct summands.
The morphism $\phi$ lands in the $j=2$ component and $\psi$ lands in the $j=3$ component.

Let us rephrase this.
The bottom line of this diagram may be identified with the following:
\[ \begin{tikzcd}
\displaystyle \coprod_{[\pi] \in (13)\backslash \Sigma_3 } [\pi] \rar{\cong} & \Sigma_3 \rar{\cong} & \displaystyle \coprod_{[\pi'] \in (23)\backslash \Sigma_3 } [\pi']
\end{tikzcd} \]
where the coproducts are indexed by right cosets.
So the preceding paragraph reflects that $\amalgamate_{e_2}$ and $\amalgamate_{e_1}$ perform right coset decompositions for different stabilizer subgroups.
\end{example}

\begin{corollary}\label{structure from rooting nonlinear}
	If $R$ is non-linear, then \[
		\amalgamate_{s_0} \colon \coprod_{\legs(R)} \Omega(\rootify(R,r), \rootify(S,s_0)) \to \Xi(R,S)
	\]
	is an isomorphism, with inverse given by $\lifting_{s_0}$.
\end{corollary}
\begin{proof}
	If $R$ is non-linear, then $\Xi^0(R,S) = \varnothing$.
\end{proof}

\begin{corollary}\label{structure from rooting linear}
	If $R \neq \eta$ is a linear graph with at least one vertex, then $\legs(R) = \{ r_0 ,r_1 \}$
	and
	\[
		\amalgamate_{s_0} \colon \Omega(\rootify(R,r_0), \rootify(S,s_0)) \amalg \Omega(\rootify(R,r_1), \rootify(S,s_0)) \to \Xi(R,S)
	\]
	satisfies
	\[
		| \amalgamate_{s_0}^{-1}(\phi) | = \begin{cases}
			2 & \phi \in \Xi^0(R,S) \\
			1 & \phi \notin \Xi^0(R,S).
		\end{cases}
	\]
	If $R = \eta = L_0$, then $\amalgamate_{s_0}$ is a bijection.
\end{corollary}
\begin{proof}
In general, $\amalgamate_{s_0}$ splits as a coproduct of $\amalgamate_{s_0}^{nc}$ with 
\[ 
\amalgamate_{s_0}^c \colon 
\coprod\limits_{\legs(R)} \Omega^0 (\rootify(R,r), \rootify(S,s_0)) \to
 \Xi^0(R,S).
 \]
If $R$ is linear and has at least one vertex, then
\[ \begin{tikzcd}
\Omega^0(\rootify(R,r_0), \rootify(S,s_0)) \amalg \Omega^0(\rootify(R,r_1), \rootify(S,s_0)) \dar{\cong} \rar{\amalgamate_{s_0}^c} &  \Xi^0(R,S) \dar{\cong} \\
\edge(S) \amalg \edge(S) \rar{\id \amalg \id} & \edge(S)
\end{tikzcd} \]
is two-to-one
and $\amalgamate_{s_0}^{nc}$ is injective by Theorem \ref{structure from rooting}.
Hence the first statement is proved.

For the second statement, if $R = \eta$ then all maps $R \to S$ are constant, $\legs(R) = \{r\}$ has one element, and
\[ \amalgamate_{s_0} = \amalgamate_{s_0}^c \colon \Omega^0(\rootify(R,r), \rootify(S,s_0)) \to  \Xi^0(R,S) \]
is isomorphic to the identity map on $\edge(S)$.

\end{proof}

\section{A generalized Reedy structure on \texorpdfstring{$\Xi$}{Ξ}}

The category $\Delta$ of nonempty finite ordered sets is the prototypical example of a \emph{Reedy category}. The surjective (resp.\ injective) maps form a wide subcategory (i.e., a subcategory which contains all of the objects of the ambient category) $\Delta^-$ (resp.\ $\Delta^+$) of morphisms which lower (resp.\ raise) degrees, such that any map has a unique factorization $f = f^+ f^-$.
Numerous inductive techniques used in the theory of (co)simplicial objects actually work in diagrams indexed by arbitrary Reedy categories. 

Generalized Reedy categories were introduced in \cite{bm} and capture the dendroidal category $\Omega$ as an example, highlighting its similarities to $\Delta$. 
We will return to the theory of model structures on diagram categories index by a generalized Reedy category $\Rr$ in Section \ref{section bmr model structure}, but for now we show that $\Xi$ admits such a structure.

\begin{definition}[{\cite[Definition 1.1]{bm}}]\label{definition greedy}
A \emph{generalized Reedy structure} on a small category $\Rr$ consists of
\begin{itemize}
\item
wide subcategories $\Rr^+$ and $\Rr^-$, and
\item
a degree function $d: \Ob(\Rr) \to \mathbb{N}$
\end{itemize}
satisfying the following four axioms.
\begin{enumerate}[(i)]
\item
Non-invertible morphisms in $\Rr^+$ (resp., $\Rr^-$) raise (resp., lower) the degree.  Isomorphisms in $\Rr$ preserve the degree. \label{Reedy def degrees}
\item
$\Rr^+ \cap \Rr^- = \Iso(\Rr)$. \label{Reedy def iso}
\item
Every morphism $f$ of $\Rr$ factors as $f = gh$ with $g  \in \Rr^+$ and $h \in \Rr^-$. This
factorization is unique up to isomorphism in the sense that if $g'h'$ is another such factorization, then there is an isomorphism $\theta $ so that $\theta h=h'$ and $g=g'\theta$.\label{Reedy def factorization}
\item
If $\theta f=f$ for $\theta \in \Iso(\Rr)$ and $f\in \Rr^-$, then $\theta$ is an identity. \label{Reedy def left iso}
\end{enumerate} If, moreover, the condition
\begin{enumerate}[(iv')]
\item If $f \theta=f$ for $\theta \in \Iso(\Rr)$ and $f\in \Rr^+$, then $\theta$ is an identity \label{Reedy def right iso}
\end{enumerate}
holds, then we call this a generalized \emph{dualizable}  Reedy structure.
\end{definition}

An ordinary Reedy category is a generalized Reedy category where there are no isomorphisms other than the identity maps.

\begin{definition}\label{reedy structure def}
Consider the following structures on $\Xi$.
	\begin{itemize}
		\item The degree function $d : \Ob(\Xi) \to \mathbb N$ with $d(S) = |\vertex(S)|$.
		\item The wide subcategory $\Xi^+$ consisting of all maps $\phi \colon R \to S$ so that $\phi_0 \colon \edge(R) \to \edge(S)$ is injective.
		\item The wide subcategory $\Xi^-$ consisting of all maps $\phi \colon R \to S$ so that $\phi_0$ is surjective and, for each vertex $v \in \vertex(S)$, there exists a vertex $w\in \vertex(R)$ with $v\in \vertex(\phi_1(w))$.
	\end{itemize}
\end{definition}

This definition is chosen to be compatible with the known generalized Reedy structure\footnote{See \cite[Example 2.8]{bm} and the minor correction in \cite[p.~216]{bh2}.} on the dendroidal category $\Omega$ from Definition \ref{dendroidal category}, in the sense that the equalities $\Omega^+ = \Omega \cap \Xi^+$ and $\Omega^- = \Omega \cap \Xi^-$ hold, and the degree functions agree.
Both inner and outer cofaces from Definition \ref{cofaces and codegens} are in $\Xi^+$, while codegeneracies are in $\Xi^-$.
In fact, one can show that $\Xi^+$ is generated by the cofaces and $\Xi^-$ is generated by the codegeneracies, though we do not need this here.
Notice that
\begin{equation} \Xi^0(R,S) \cap \Xi^+(R,S) \cong \begin{cases} \edge(S) & \text{ if } R = \eta \\ \varnothing & \text{otherwise}\end{cases} \label{equation about constants and plus}\end{equation} 
is nonempty if and only if $R = \eta$. The following lemma is immediate.

\begin{lemma}\label{lifting for plus minus} Given $s_0 \in \edge(S)$, the map
\[
	\lifting : \Xi(R, S) \setminus \Xi^0(R,S) \to \coprod_{r\in \legs(R)} \Omega(\rootify(R,r), \rootify(S,s_0)) 
\]
restricts to maps
\begin{align*}
	\Xi^+(R, S) \setminus \Xi^0(R,S) &\to \coprod_{r\in \legs(R)} \Omega^+(\rootify(R,r), \rootify(S,s_0)) \\
	\Xi^-(R, S) \setminus \Xi^0(R,S) &\to \coprod_{r\in \legs(R)} \Omega^-(\rootify(R,r), \rootify(S,s_0)).
\end{align*}
\end{lemma}

\begin{proposition}\label{proposition xi is reedy}
With the structure from Definition \ref{reedy structure def}, $\Xi$ is a dualizable generalized Reedy category.
\end{proposition}
\begin{proof}
For \eqref{Reedy def degrees}: note that isomorphisms preserve degree.
If $\phi \colon R \to S \in \Xi^+$, then $\{ \vertex(\phi_1(v)) \}_{\vertex(R)}$ is a collection of pairwise disjoint, non-empty subsets of $\vertex(S)$. Thus
\[
	d(R) = |\vertex(R)| = \sum_{\vertex(R)} 1 \leq \sum_{v\in \vertex(R)} |\vertex(\phi_1(v))| \leq |\vertex(S)| = d(S).
\]
If $\phi \colon R \to S \in \Xi^-$, 
then surjectivity of $\phi_0$ implies (by Lemma \ref{neighborhoods of vertices}) that $|\vertex(\phi_1(v))| \leq 1$ for all $v \in \vertex(R)$.
Let $\vertex(S) \to \vertex(R)$ be the map which sends $v \in \vertex(S)$ to the (unique, by Definition \ref{maps of xi}\eqref{maps of xi separation}) vertex $w$ with $v \in \vertex(\phi_1(w))$.
Since there is at most one $v$ in a given $\phi_1(w)$, the map $\vertex(S) \to \vertex(R)$ is injective, hence $d(S) \leq d(R)$.

For \eqref{Reedy def iso}, it is clear that $\Iso(\Xi)$ is contained in $\Xi^+ \cap \Xi^-$.
For the reverse inclusion, suppose that $\phi : R\to S$ is in both $\Xi^+$ and $\Xi^-$.
If $\phi$ is constant, then by \eqref{equation about constants and plus} we must have $R = \eta$; since $\phi$ is also in $\Xi^-$, $S$ is also an edge, and $\phi$ is an isomorphism.
If $\phi$ is not constant, choose a root $s_0$ for $S$.
Then by Lemma \ref{lifting for plus minus} we know that $\lifting(\phi)$ is in \[ \Omega^+(\rootify(R,r_0), \rootify(S,s_0)) \cap \Omega^-(\rootify(R,r_0), \rootify(S,s_0)) = \Iso_{\Omega}(\rootify(R,r_0), \rootify(S,s_0))\]  where $r_0 = \findroot_{s_0}(\phi)$.
By Corollary \ref{iso decompositions}, we thus have that $\phi$ is an isomorphism in $\Xi$. Since $\phi$ was arbitrary, we have $\Xi^+ \cap \Xi^- \subseteq \Iso(\Xi)$ as well.

For \eqref{Reedy def left iso}, note that if $\theta$ is an isomorphism, $\phi \colon R \to S \in \Xi^-$, and $\theta \phi = \phi$, then $\theta_0 \phi_0 = \phi_0$. 
Since $\theta_0$ is a bijection of sets and $\phi_0$ is a surjection of sets, it follows that $\theta_0$ is an identity. 
There is only one isomorphism $S \to S$ in $\Xi$ which is the identity on edges (Example \ref{example unique iso}), hence $\theta = \id_S$. The proof that $\Xi$ satisfies (iv') follows similarly.

We finally turn to \eqref{Reedy def factorization}.
We first \emph{construct} a factorization of a given morphism of $\Xi$.
	We may assume that $\phi \in \Xi(R,S) \setminus \Xi^0(R,S)$.
	Pick a root $s_0$ for $S$, and consider
	\[
		\lifting(\phi) \in \Omega(\rootify(R,r_0), \rootify(S,s_0))
	\]
	where $r_0 = \findroot_{s_0}(\phi)$.
	Then there is a decomposition $\lifting(\phi) = g \circ h$
	with 
$g \in \Omega^+(T, \rootify(S,s_0))$ and $ h\in \Omega^-(\rootify(R,r_0), T)$.
We have
\[ \begin{tikzcd}
\rootify(R,r_0) \arrow[rr, bend left, "\lifting(\phi)"] \rar{h} & T \rar{g} & \rootify(S,s_0);
\end{tikzcd} \]
apply the functor $\iota$ to this diagram to get
\[ \begin{tikzcd}
\iota(\rootify(R,r_0)) \dar{\cong} \arrow[rr, bend left, "\iota(\lifting(\phi))"] \rar{\iota h} & \iota T \rar{\iota g} & \iota(\rootify(S,s_0)) \dar{\cong} \\
R \arrow[rr, "\amalgamate(\lifting(\phi)) = \phi"] & & S.
\end{tikzcd} \]
Since $\iota h \in \Xi^-$, $\iota g \in \Xi^+$, and isomorphisms are in $\Xi^+ \cap \Xi^-$, we have provided the desired decomposition of $\phi$.

Suppose that $\phi^1 \circ \psi^1 = \phi^2 \circ \psi^2$ with $\phi^i \in \Xi^+(U_i,S)$ and $\psi^i \in \Xi^-(R, U_i)$.
Let $u_i = \findroot_{s_0}(\psi^i)$ and $r_0 = \findroot_{u_1}(\phi^1) \overset{\ref{functoriality of lifting}}= \findroot_{u_2}(\phi^2)$.
We have, by Proposition \ref{functoriality of lifting},
\[
	\lifting_{u_1}(\phi^1) \circ \lifting_{s_0}(\psi^1) = \lifting_{u_2}(\phi^2) \circ \lifting_{s_0}(\psi^2).
\]
Now $\lifting_{u_i}(\phi^i) \in \Omega^+$ and $\lifting_{s_0}(\psi^i) \in \Omega^-$, so there exists an isomorphism $a$ making the diagram
\[ \begin{tikzcd}
\rootify(R,r_0) 
	\arrow[r, "\lifting_{u_1}(\psi^1)"]  
	\arrow[d, "\id", "=" swap] & 
\rootify(U_1,u_1) 
	\arrow[r, "\lifting_{s_0}(\phi^1)"]  
	\arrow[d, "a", "\cong" swap] &
\rootify(S,S_0)  
	\arrow[d, "\id", "=" swap] \\
\rootify(R,r_0) 
	\arrow[r, "\lifting_{u_2}(\psi^2)"]  & 
\rootify(U_2,u_2) 
	\arrow[r, "\lifting_{s_0}(\phi^2)"]  &
\rootify(S,S_0)  
\end{tikzcd} \]
commute.
Applying $\iota$ gives the back square of the diagram:
\[ \begin{tikzcd}[row sep=small, column sep=small]
\iota(\rootify(R,r_0)) 
	\arrow[rrrr, "\iota(\lifting_{u_1}(\psi^1))"] 
	\arrow[dr, "\cong"] 
	\arrow[ddd, "=" swap] &&&& 
\iota(\rootify(U_1,u_1)) 
	\arrow[rrrr, "\iota(\lifting_{s_0}(\phi^1))"] 
	\arrow[dr, "\cong"] 
	\arrow[ddd, "\iota(a)", "\cong" swap, near end] &&&&
\iota(\rootify(S,S_0)) 
	\arrow[dr, "\cong"] 
	\arrow[ddd, "=" swap, near end] 
\\ & 
R 
	\arrow[rrrr, "\psi^1", crossing over, near start] &&&& 
U_1 
	\arrow[rrrr, "\phi^1", crossing over, near start] &&&&
S
\\ \\
\iota(\rootify(R,r_0)) 
	\arrow[rrrr, "\iota(\lifting_{u_2}(\psi^2))", near end] 
	\arrow[dr, "\cong"] &&&& 
\iota(\rootify(U_2,u_2)) 
	\arrow[rrrr, "\iota(\lifting_{s_0}(\phi^2))", near end] 
	\arrow[dr, "\cong"] &&&&
\iota(\rootify(S,S_0)) 
	\arrow[dr, "\cong"] 
\\ & 
R 
	\arrow[rrrr, "\psi^2"] 
	\arrow[from=uuu, crossing over, "=" swap, near start] &&&& 
U_2 
	\arrow[rrrr, "\phi^2"] 
	\arrow[from=uuu, crossing over, dashed] &&&&
S
	\arrow[from=uuu, crossing over, "="] 
\end{tikzcd} \]
and there exists a dashed map making the diagram commute. This map is necessarily an isomorphism. Thus, the front face establishes uniqueness of decompositions in $\Xi$.

\end{proof}

A stronger statement is true, namely that $\Xi$ is an EZ-category in the sense of \cite[Definition 6.7]{bm}.
We delay a proof of this fact (Theorem~\ref{xi is ez category}) until the end of the next section, as we should first learn a bit more about maps in $\Xi^-$.

\section{The active / inert weak factorization system on \texorpdfstring{$\Xi$}{Ξ}}

In this section we exhibit a weak factorization system on the category $\Xi$.
Given a class $I$ of morphisms in a category $\mathcal C$, write $\rlp{I}$ for the maps which have the right lifting property with respect to every element of $I$.
In other words, $f: X \to Y$ is in  $\rlp{I}$ if, and only if, every commutative square
\[ \begin{tikzcd}
A \dar{i} \rar & X \dar{f} \\
B \rar \arrow[dotted]{ur} & Y  
\end{tikzcd} \]
with $i \in I$ admits a lift $B \to X$.
Similarly, $\llp{I}$ is the class of maps having the left lifting property with respect to every element of $I$.

A weak factorization system (see \cite[Definition 11.2.1]{riehl}) consists of two classes of maps $\mathcal L$ and $\mathcal R$ so that every morphism factors into a map in $\mathcal L$ followed by one in $\mathcal R$, and so that
$\llp{\mathcal R} = \mathcal L$ and $\rlp{\mathcal L} = \mathcal R$.

\begin{remark}
	We have actually already encountered one weak factorization system in this paper, namely the one whose left class is $\Xi^-$ and whose right class is $\Xi^+$.
	In fact, the following is true: 
	if $\Rr$ is a generalized Reedy category, then $(\Rr^-, \Rr^+)$ is an \emph{orthogonal factorization system} (see, e.g., \cite[2.2]{MR0322004}), that is, a weak factorization system in which all of the liftings are unique.
	This is mentioned in Remark 8.28 of \cite{shulman_reedy}; its proof is an exercise using only axioms \eqref{Reedy def factorization}, \eqref{Reedy def left iso}, and closure of $\Rr^+$, $\Rr^-$ under composition.
\end{remark}

Any weak factorization system whose right class $\mathcal R$ is contained in the monomorphisms will be orthogonal.
This will be true, in particular, of the weak factorization system in Proposition~\ref{proposition active inert} (using Theorem~\ref{xi is ez category} and Remark~\ref{remark inert in xi plus}).

\begin{definition}
A morphism $\phi : R \to S$ in $\Xi$ is called \emph{active} if $\image(\phi) = S$.  It is called  \emph{inert} if $\phi_0 : \edge(R) \to \edge(S)$ is injective and if $w \in \vertex(\phi_1(v))$ then $\nbhd(w) \subseteq \image(\phi_0)$.
\end{definition}

Outer cofaces (see Definition \ref{cofaces and codegens}) are inert maps, while codegeneracies and inner cofaces are active maps.
Notice that every map in $\Xi^-$ (see Definition \ref{reedy structure def}) is an active map. 
Further, every inert map is contained in $\Xi^+$ by the following remark.

\begin{remark}\label{remark inert in xi plus}
Suppose that $\phi$ is inert. 
Then, since  $\phi_0$ is injective, we know $ 0 < |\vertex(\phi_1(v))|$ for any $v\in \vertex(R)$.
	Further, if $\vertex(\phi_1(v))$ contains two (adjacent) vertices $w_1$ and $w_2$, connected by an edge $e$, then $e\in \interior(\phi_1(v)) \cap \nbhd(w_1) \subseteq \interior(\phi_1(v)) \cap \image(\phi_0) = \varnothing$ by
	Lemma \ref{interiors disjoint from image}.
	Thus $|\vertex(\phi_1(v))| \leq 1$.
	In other words, $\phi$ is (isomorphic to) a subgraph inclusion.
\end{remark}

Let us spell out an alternative characterization of the active maps. The following lemma will be useful in the proof of Proposition~\ref{active characterization}.

\begin{lemma}\label{lemma for active characterization}
	Suppose that $R \in \sbgph(S)$ is not an edge. If $\legs(R) = \legs(S)$, then $R=S$.
\end{lemma}
\begin{proof}
	Suppose that $v\in \vertex(S)$ is connected to a different vertex $w\in \vertex(R)$ by an edge $e$.
	Then $e\notin \legs(S) = \legs(R)$, so $v\in \vertex(R)$. Since $S$ is connected and $\vertex(R) \neq \varnothing$, we have $\vertex(R) = \vertex(S)$. 
	The definition of subgraph implies $\edge(R) = \edge(S)$ as well.
\end{proof}

\begin{proposition}\label{active characterization}
	Suppose $\phi : R \to S$ is in $\Xi$.
	\begin{enumerate}
		\item Suppose $\phi \notin \Xi^0(R,S)$. Then $\phi$ is active if and only if $\phi_0(\legs(R)) = \legs(S)$.\label{active characterization nonconstant}
		\item Suppose $\phi \in \Xi^0(R,S)$. Then $\phi$ is active if and only if $S = \eta$.\label{active characterization constant}
	\end{enumerate}
\end{proposition}
\begin{proof}

The forward direction of \eqref{active characterization nonconstant} follows immediately from Proposition~\ref{legs interchange}.
For the reverse direction, by hypothesis and Proposition~\ref{legs interchange} we have $\legs(S) = \legs(\image(\phi))$, so by Lemma~\ref{lemma for active characterization} we have $\image(\phi) = S$.
For \eqref{active characterization constant}, simply note that if $\phi \in \Xi^0(R,S)$, then $\image(\phi)$ is an edge.
\end{proof}

\begin{lemma}\label{lemma one of the inclusions}
The set of active morphisms has the left lifting property with respect to the set of inert morphisms.
\end{lemma}
\begin{proof}
Suppose we are given a commutative diagram in $\Xi$
\begin{equation} \label{diagram in question} \begin{tikzcd}
R \dar{\phi} \rar{\alpha} & P \dar{\psi} \\
S \arrow[ur, color=magenta, dashed, "\gamma" description] \rar{\beta} & Q
\end{tikzcd} \end{equation}
with $\phi$ active and $\psi$ inert (so, in particular,
$P$ is a subgraph of $Q$); we wish to show that a lift $\gamma : S \to P$ exists. Temporarily write $\alpha_0^L : \legs(R) \to \edge(P)$ and $\phi_0^L : \legs(R) \to \legs(S)$ for the restrictions of $\alpha_0$ and $\phi_0$ to legs.

If $S =\eta$ consists of the single edge $e$, then $R$ is linear.
We have a diagram
\[ \begin{tikzcd}
\legs(R) \dar{\phi_0} \rar{\alpha_0^L} & \edge(P)  \dar[hook]{\psi_0} \\
\{ e\} \rar{\beta_0} & \edge(Q)
\end{tikzcd} \]
with $\psi_0$ an injection.
Thus $\image(\alpha_0^L) = |_{p} $ is a single edge, and we define $\gamma : S \to P$ by $\gamma_0(e) = p\in \edge(P)$.

If $S$ contains a vertex, then $\phi_0^L : \legs(R) \to \legs(S) $ is a bijection.
Define $\gamma_0^L = \alpha_0^L \circ (\phi_0^L)^{-1} : \legs(S) \to \edge(P)$.
We wish to extend this to $\interior(S)$, which we will do in a moment.
Since $\phi$ is active, every $w\in \vertex(S)$ is in $\vertex(\phi_1(v))$ for some $v\in \vertex(R)$.
Notice that 
\[
\bigcup_{w\in \phi_1(v)} \beta_1(w) = 
(\beta \circ \phi)_1(v) = (\psi \circ \alpha)_1(v) = \bigcup_{t \in \alpha_1(v)} \psi_1(t) \in \sbgph( \image(\psi)) \]
and, 
since $P \to \image(\psi)$ is an isomorphism, $\sbgph( \image(\psi)) \cong \sbgph(P)$. 
Thus, for $w\in \vertex(\phi_1(v))$, there is a unique subgraph $G_w \in \sbgph(P)$ which maps to $\beta_1(w) \in \sbgph(\image(\psi)) \subseteq \sbgph(Q)$ under $\psi$.

Suppose that $s\in \interior(S)$ is an interior edge. The edge $s$ is adjacent to two distinct vertices $w$ and $w'$.
We have $\beta_1(w) \cap \beta_1(w') = \{ \beta_0(s) \}$, hence $G_w \cap G_{w'} = \{ p \}$ for some edge $p$. Set $\gamma_0(s) = p$; by definition we have $\phi_0 \gamma_0 (s) = \beta_0(s)$.
Further, since $\phi$ is active it sends legs to legs, so if $s = \phi_0(r)$, then $r$ is an internal edge between distinct vertices $v$ and $v'$, with $w\in \vertex(\phi_1(v))$ and $w'\in \vertex(\phi_1(v'))$.
Since $G_w \subseteq \alpha_1(v)$ and $G_{w'} \subseteq \alpha_1(v')$, we have $\{ \gamma_0(s) \} = G_w \cap G_{w'} \subseteq \alpha_1(v) \cap \alpha_1(v') = \{ \alpha_0(r) \}$. Thus
$\gamma_0 \phi_0(r) = \gamma_0(s) = \alpha_0(r)$.
In conclusion, $\alpha_0 = \gamma_0 \phi_0$ and $\psi_0 \gamma_0 = \beta_0$.

Next, define $\gamma_1(w) = G_w \in \sbgph(\alpha_1(v))$. 
The pair $\gamma = (\gamma_0, \gamma_1)$ is a morphism of $\Xi$, which quickly follows from inertness of $\psi$ and the fact that $\beta$ is a morphism of $\Xi$.
By definition, we have
$(\psi \circ \gamma)_1(w) = \bigcup_{v\in \gamma_1(w)} \psi_1(v) = \beta_1(w)$,
hence $\psi \circ \gamma = \beta$.
Then $\psi \circ \gamma \circ \phi = \beta \circ \phi = \psi \circ \alpha$, hence the injective map $\psi : \sbgph(P) \hookrightarrow \sbgph(Q)$ takes 
$(\gamma \circ \phi)_1(v)$ and $\alpha_1(v)$ to the same element. It follows that $(\gamma \circ \phi)_1(v) = \alpha_1(v)$, so $\gamma \circ \phi = \alpha$.
Thus we have shown that \eqref{diagram in question} always has a lift.
\end{proof}

\begin{proposition}\label{proposition active inert}
The active and inert morphisms form a weak factorization system.	
\end{proposition}
\begin{proof}
Given a map $\phi : R\to S$, we can factor $\phi$ as $R \to \image(\phi) \to S$. The map $R \to \image(\phi)$ is active since $\legs(\image(\phi)) = \phi_0(\legs(R))$, while $\image(\phi) \to S$ is a subgraph inclusion, hence inert.

For the remainder of the proof, write $\mathcal L$ for the set of active morphisms and $\mathcal R$ for the set of inert morphisms.
In Lemma \ref{lemma one of the inclusions} we showed that $\mathcal R \subseteq \rlp{\mathcal L}$ (or, equivalently, $\mathcal L \subseteq \llp{\mathcal R}$).

For the reverse inclusions, suppose that $\phi : R \to S$ is in $\rlp{\mathcal L}$. Consider the decomposition 
\[
	R \xrightarrow{\phi^-} T \xrightarrow{\phi^+} S
\]
coming from the generalized Reedy structure on $\Xi$.
We know that $\phi^-$ is an active morphism, hence we can lift in the diagram
\[ \begin{tikzcd}
R \dar{\phi^-} \rar{\id} & R \dar{\phi} \\
T\arrow[ur, color=magenta, dashed, "\psi" description] \rar{\phi^+} & S
\end{tikzcd} \]
which implies $\psi \phi^- = \id$, hence $\phi^-$ is an isomorphism. Thus $\phi_0$ is injective.
Further, the diagram
\[ \begin{tikzcd}
R \dar{\bar \phi} \rar{\id} & R \dar{\phi} \\
\image(\phi) \rar{i} & S
\end{tikzcd} \]
admits a lift; if 
$w$ is a vertex in $\phi_1(v) \subseteq \image(\phi)$, then
\[
	\nbhd(w) = i_0(\nbhd(w)) = \phi_0 \psi_0(\nbhd(w)) \subseteq \image(\phi_0).
\]
Hence $\phi$ is inert, and we see that $\rlp{\mathcal L} \subseteq \mathcal R$.

Now suppose that $\phi : R \to S$ is in $\llp{\mathcal R}$.
Then
\[ \begin{tikzcd}
R \dar{\phi} \rar{\bar \phi} & \image(\phi) \dar{i} \\
S \rar{\id} \arrow[ur, color=magenta, dashed, "\psi" description] & S
\end{tikzcd} \]
admits a lift $\psi : S \to \image(\phi)$. 
Factor $\psi = \psi^+\psi^-$ with $\psi^+ \in \Xi^+$, $\psi^- \in \Xi^-$. Since $\id_S = (i\psi^+)\psi^-$ with $i\psi^+ \in \Xi^+$, we have $i\psi^+, \psi^- \in \Iso(\Xi)$ by Definition~\ref{definition greedy}\eqref{Reedy def factorization}. It follows that $i\in \Xi^+$ preserves degree, hence is an isomorphism.
Thus $\phi$ is active, and we find that $\llp{\mathcal R} \subseteq \mathcal L$.

\end{proof}

The active and inert maps which are oriented form a weak factorization system on the category $\Omega$.
This was proved in \cite[Proposition 1.3.13]{kocktrees}, where inert maps are called `free' and active maps are called `boundary preserving'.
\begin{proposition}
	The inclusion $\iota : \Omega \to \Xi$ respects the weak factorization structures. \qed
\end{proposition}

As promised at the conclusion of the previous section, we have the following.

\begin{theorem}\label{xi is ez category}
	The category $\Xi$, together with the degree function $d(S) = |\vertex(S)|$ from Definition~\ref{reedy structure def}, is an EZ-category in the sense of \cite[Definition 6.7]{bm}.
\end{theorem}
\begin{proof}
	As we have already established that $\Xi$ is a dualizable generalized Reedy category (Proposition~\ref{proposition xi is reedy}), it is enough to show that 
\begin{enumerate}[(a)]
	\item  $\Xi^+$ is the subcategory of monomorphisms, \label{xi ez one}
	\item  $\Xi^-$ is the subcategory of split epimorphisms, and \label{xi ez two}
	\item  any pair of split epimorphisms with common domain has an absolute pushout \cite{MR0280565} (that is, can be extended to a commutative square which becomes a pushout square after applying any functor). \label{xi ez three}
\end{enumerate}
	For efficiency, we rely on the fact that $\Omega$ is an EZ-category \cite[Examples 6.8]{bm}.
	Every map $\phi : R \to S$ in $\Xi$ is isomorphic (in the arrow category $\Xi^{[1]}$) to at least one map $\iota(\tilde \phi : \tilde R \to \tilde S)$ by Corollary~\ref{structure from rooting nonlinear}, Corollary~\ref{structure from rooting linear}, and the fact that $\legs(S) \neq \varnothing$. 
	Suppose that $\phi$ is a map and we have chosen an isomorphism $\gamma: \iota \tilde \phi \cong \phi$ in $\Xi^{[1]}$. 
	The following hold, which then imply that \eqref{xi ez one} and \eqref{xi ez two} hold by the corresponding properties of $\Omega$:
	\begin{enumerate}[(i)]
		\item $\phi$ is a monomorphism (resp.\ split epimorphism) if and only if $\tilde \phi$ is a monomorphism (resp.\ split epimorphism)
		\item $\phi$ is in $\Xi^+$ (resp.\ in $\Xi^-$) if and only if $\tilde \phi$ is in $\Omega^+$ (resp.\ in $\Omega^-$).
	\end{enumerate}
	The only point that is perhaps not immediate is that if $\phi$ is a split epimorphism, so is $\tilde \phi$.
	Suppose $\alpha$ is a section of $\phi$.
	If $\phi$ is constant then $S = \tilde S = \eta$ and the composite $\eta \xrightarrow{\alpha} R \xrightarrow{\gamma^{-1}} \iota \tilde R$ is automatically oriented and is a section for $\tilde \phi$.
	If $\phi$ is not constant, then $\tilde \phi$ is isomorphic in the arrow category $\Omega^{[1]}$ to $\lifting_{s_0,r_0}(\phi) : \rootify(R,r_0) \to \rootify(S, s_0)$ where $s_0$ and $r_0$ are the images under $\gamma$ of the roots of $\tilde S$ and $\tilde R$, respectively.
	We then have
	\[
		\id_{\rootify(S,s_0)} = \lifting_{s_0} (\id_{S}) = \lifting_{s_0} (\phi \circ \alpha) = \lifting_{s_0} (\phi) \circ \lifting_{r_0}(\alpha)
	\]
	by Proposition~\ref{functoriality of lifting}. Since $\lifting_{s_0} (\phi)$ admits a section, so does $\tilde \phi$.

	Let us turn to point \eqref{xi ez three}.
	Suppose that $\phi : R \to S$ and $\psi : R \to U$ are two maps in $\Xi^-$ (i.e., two split epimorphisms).
	Pick $r_0 \in \legs(R)$, and let $s_0 = \phi_0(r_0) \in \legs(S)$ and $u_0 = \psi_0(r_0) \in \legs(U)$ (using that $\phi$ and $\psi$ are active).
	There are maps $\tilde \phi : \rootify(R,r_0) \to \rootify(S,s_0)$ and $\tilde \psi : \rootify(R,r_0) \to \rootify(U,u_0)$ and an isomorphism of diagrams of shape $\bullet \leftarrow \bullet \rightarrow \bullet$
	\[
	\begin{tikzcd}
		\iota \rootify(U,u_0)\dar["\gamma_U","\cong" swap] & \iota \rootify(R,r_0)\dar["\gamma_R","\cong" swap] \lar[swap]{\iota \tilde \psi} \rar{\iota \tilde \phi} & \iota \rootify(S,s_0)\dar["\gamma_S","\cong" swap] \\
		U & R \lar[swap]{\psi} \rar{\phi} & S
	\end{tikzcd}
	\]
	in $\Xi$.
	Let
	\begin{equation}\label{abs pushout in omega}
		\begin{tikzcd}
			\rootify(R,r_0) \dar[swap]{\tilde \psi} \rar{\tilde \phi} & \rootify(S,s_0) \dar{h} \\
			\rootify(U,u_0) \rar{g} & T\\
		\end{tikzcd}
	\end{equation}
	be an absolute pushout of $\tilde \phi, \tilde \psi \in \Omega^-$.
	The diagram \eqref{abs pushout in omega} remains an absolute pushout after applying $\iota$, and the resulting square is isomorphic to the square
	\begin{equation}\label{abs pushout in xi}
		\begin{tikzcd}
			R \dar[swap]{\psi} \rar{\phi} & S \dar{(\iota h) \gamma_S^{-1}} \\
			U \rar{(\iota g) \gamma_U^{-1}} & \iota T.\\
		\end{tikzcd}
	\end{equation}
	Thus \eqref{abs pushout in xi} is an absolute pushout as well, and \eqref{xi ez three} is established.
\end{proof}

\section{A functor from \texorpdfstring{$\Xi$}{Ξ} to the category of cyclic operads}\label{towards cyclic}

In this section, we will show that given a tree $R$, there is a cyclic (colored) operad $C(R) \in \Cyc$ with $\colset(C(R)) = \edge(R)$.
Further, the assignment $R\mapsto C(R)$ is the object part of a functor $C: \Xi \to \Cyc$.
This functor is faithful but not full.

\subsection{A monadic description}
Fix a color set $\fC$. 
A $\fC$-colored tree is a (pinned) tree $S$ together with a function $\xi \colon \edge(S) \to \fC$.
If $\uc = c_0, c_1, \dots, c_n$ is a profile in $\fC$, we will write $\unrootedtrees(\uc) = \unrootedtrees^{\fC}(\uc)$ for the groupoid of all $\fC$-colored trees $S$ so that 
\[ \begin{tikzcd}
	\{0, 1, \dots, n\} \rar{\ord} & \legs(S) \rar[hook] & \edge(S) \rar{\xi} & \fC 
\end{tikzcd} \]
takes $i$ to $c_i$ for $0\leq i \leq n$.
The isomorphisms $(S,\xi) \to (S', \xi')$ in $\unrootedtrees(\uc)$ are those isomorphisms $\phi : S \to S'$ in $\Xi$ so that $\xi = \xi'\circ \phi_0$ and $\phi_0 \circ \ord_S = \ord_{S'}$.

Consider the groupoid $\Sigma_\fC^+$ whose objects are finite, non-empty, ordered lists $\uc = c_0, c_1, \dots, c_n$ ($n\geq 0$) of elements of $\fC$, and morphisms are 
$\uc\sigma \overset\sigma\to \uc$, where $\sigma  \in \Sigma_m^+ = \Aut\{0,1,\dots, m\}$, and $\uc \sigma = (c_{\sigma(0)}, c_{\sigma(1)}, \dots, c_{\sigma(m)})$.
Such a morphism of $\Sigma_\fC^+$ determines a morphism $\unrootedtrees(\uc) \to \unrootedtrees(\uc\sigma)$ sending $S$ to $S\sigma$, which has all of the same structure as $S$ except that the leg ordering $\ord_{S\sigma}$ is the composite
\[ \begin{tikzcd}
\{ 0, 1, \dots, m \} \rar{\sigma}
& \{ 0, 1, \dots, m \} \rar{\ord_S} & \legs(S).
\end{tikzcd} \]
This is, in fact, a contravariant functor $\unrootedtrees(-) = \unrootedtrees^{\fC}(-) \colon \Sigma_\fC^+ \to \mathbf{Gpd}$.

Every object $X \in \mathcal \Set^{\Sigma_\fC^+}$ determines a functor $\unrootedtrees^{\fC}(\uc) \to \Set$, given on an object $S\in \unrootedtrees^{\fC}(\uc)$ by
\[
X[S] = \prod_{v\in \vertex(S)} X({\xi(\nbhd(v))}). 	
\] 

\begin{remark}
	There is a monad $T^+ : \Set^{\Sigma_\fC^+} \to \Set^{\Sigma_\fC^+}$ given on objects by 
	\[
	T^+(X)(\uc) = \colim_{S \in \unrootedtrees(\uc)} X[S].\]
	The category of algebras for $T^+$ is $\Cyc_\fC$, the subcategory of $\Cyc$ consisting of those cyclic operads with color set $\fC$ and morphisms which are the identity on colors. An analogous statement appears in \cite[\S 5.1]{mss} in the monochrome case (see also \cite{JOYAL2011105} for the colored case). 
\end{remark}

Notice that if $f : \fC \to \fD$ is a map of sets, then restriction along $\Sigma_\fC^+ \to \Sigma_\fD^+$ induces a functor $f^* : \mathbf{Cyc}_{\fD} \to \mathbf{Cyc}_{\fC}$.
A map $\alpha: P \to Q$ in $\Cyc$ is the same thing as a pair $(\alpha_0,\alpha_1)$, where $\alpha_0 : \colset(P) \to \colset(Q)$ is a map of sets and $\alpha_1 : P \to \alpha_0^*Q$ is a map in $\mathbf{Cyc}_{\fC}$.

\subsection{The functor \texorpdfstring{$C: \Xi \to \Cyc$}{C:Ξ → Cyc}}

Let $\prof(\fC)$ be the set of non-empty ordered lists of elements in $\fC$; it is the set of objects of the groupoid $\Sigma_\fC^+$. 
There is a forgetful functor $\Cyc_\fC \to \Set^{\Sigma_\fC^+} \to \Set^{\prof(\fC)}$.
Write
\[
	F_\fC : \Set^{\prof(\fC)} \rightleftarrows \Cyc_\fC : U_\fC
\]
for the corresponding adjunction.

\begin{definition}Let $S$ be a tree and $\fC = \edge(S)$.
Then $S$ determines an object $Z = Z^S$ of $\Set^{\prof(\fC)}$ with
\begin{equation}\label{Z def}
	Z_{\uc} = \begin{cases}
		\{ v \} & \nbhd(v) = \uc \text{ as ordered lists} \\
		\varnothing & \text{otherwise.}
	\end{cases}
\end{equation}
The cyclic operad $C(S)$ is defined as $F_\fC(Z)$.
\end{definition}
Applying the left adjoint of $\Set^{\Sigma_\fC^+} \to \Set^{\prof(\fC)}$ (for $\fC = \edge(S)$) to the object $Z=Z^S$ gives a new object $Z'$. This object has the property that $|Z'_{\uc}| = 1$ if and only if $\uc$ contains no repetitions and $\uc = \nbhd(v)$ as \emph{unordered} lists for some $v$. Otherwise, $|Z'_{\uc}| = 0$.

Notice that if $T\in \Omega$ is a rooted tree and $F: \Operad \to \Cyc$ is left adjoint to the forgetful functor, then 
\begin{equation*}
	C(\iota(T)) = F(\Omega(T)).
\end{equation*}

\begin{remark}\label{rmk pinned sub to elements}
If $R \in \sbgph(S)$ and $R$ is a pinned graph, with $\{0, 1\, \dots, n\} \xrightarrow{\ord_R} \legs(R) \hookrightarrow \edge(S)$ sending $i$ to $c_i$, then $R$ determines an element of $C(S)(c_1, \dots, c_n; c_0)$. 
There may, of course, be many elements of this set that do not come from subgraphs of $S$.
Varying the vertex orders $\ord^v$ on $R$ does not change the element produced in this way.
The cyclic operad $C(S)$ is nearly always infinite (see Example \ref{nonfull example} below). Specifically, $C(S)$ is infinite whenever $S \notin \{\eta, \cstar_1\}$, as then $S$ has a vertex $v$ with $|v| > 1$.
\end{remark}

\begin{example}
Let $S$ be the graph from Figure \ref{figure basic tree example}. Then $C(S)$ is the cyclic operad $O$ generated by three operations $u\in O(b,a;c)$, $v \in O(d,e,f;c)$, and $w\in O(\, ; d)$. We see there is an element $v \circ_1 w \in O(e,f;c)$.
Notice that there are no elements of the form $q\circ_i r$ where $r\in \{u,v\}$, $q\in \{u,v,w\}$, as $c$ is not an input for such a $q$.
But we can first apply the rotation to get, say $v\cdot \tau \in O(e,f,c;d)$, and then compose to get $(v\cdot \tau) \circ_3 u \in O(e,f,b,a;d)$. Finally, there is an element $((v\cdot \tau^2) \circ_2 u) \circ_4 w \in O(f,b,a;e)$; these are all of the elements given by subgraphs of $S$. Notice there are many other elements, for example $u \circ_1 (u\cdot \tau) \in O(a,c,a;c)$.
\end{example}

We wish to extend $S\mapsto C(S)$ to a functor $\Xi \to \Cyc$. 
Defining maps out of $C(S)$ is easy, as this object is free in $\Cyc_{\edge(S)}$. We utilize Remark~\ref{rmk pinned sub to elements} in order to regard subgraphs of $S$ as elements in $C(S)$.

\begin{definition}[$C$ as a functor]\label{functor xi to cyc} 
Suppose $\phi \in \Xi(R,S)$ and write $\fC = \colset(C(R)) = \edge(R)$.
Set $C(\phi)_0 = \phi_0 : \fC = \edge(R) \to \edge(S)$.
Since $C(R)$ is free in $\Cyc_\fC$, it is enough to define $C(\phi)_1$ on generators. 
Let $\phi_1^\natural : Z^R \to U_\fC \phi_0^* C(S)$ in $\Set^{\prof(\fC)}$ by $v \mapsto \phi_1(v) \in \sbgph(S)$, endowed with the ordering $\{0,1, \dots, n_v \} \xrightarrow{\ord^v} \nbhd(v) \overset{\phi_0}\to \legs(\phi_1(v))$. Define $C(\phi)_1 : C(R) \to \phi_0^*C(S)$ to be the adjoint of $\phi_1^\natural$.
\end{definition}

\begin{theorem}\label{C is faithful} 
	The functor $C: \Xi \to \Cyc$
	is faithful.
\end{theorem}
\begin{proof}
Let $\phi, \psi \in \Xi(R,S)$, and suppose that $C(\phi) = C(\psi)$.
Write $f = \phi_0 = \psi_0$ for the common function on color sets.
By assumption, for each $v \in \vertex(R)$, we have that $C(\phi)_1(v)$ is equal to $C(\psi)_1(v)$.
But $C(\phi)_1(v)$ comes from $\phi_1(v)$ (with appropriate choice of pinned structure) as in Remark~\ref{rmk pinned sub to elements}, and likewise for $C(\psi)_1(v)$.
It follows then that the subgraphs $\phi_1(v)$ and $\psi_1(v)$ are the same, hence $\phi = \psi$.
\end{proof}

\begin{example}[$C : \Xi \to \Cyc$ is not full]\label{nonfull example}
Consider the graph $L_1 = $ ---\!\textbullet\!--- with edges $0$ and $1$ and vertex $v$.
There are exactly four elements in $\Xi(L_1,L_1)$, corresponding to the four maps of edge sets $\{0,1\} \to \{0,1\}$.
But there are infinitely many maps in $\hom_{\Cyc}(C(L_1), C(L_1))$. 
One example which is not in $C(\Xi(L_1,L_1))$ is the map $f: C(L_1) \to C(L_1)$ which on color sets is $f(i) = 0$ and on morphisms is specified by $f(v) = v \circ_1 v$:
\[
	\text{$\overset{0}{\text{---}}$\!\textbullet\!$\overset{1}{\text{---}}$\!\textbullet\!$\overset{0}{\text{---}}$}.
\]
\end{example}

In Section \ref{section cyclic dendroidal sets}, we saw several examples of cyclic dendroidal sets. Another class of examples are the nerves of cyclic operads.

\begin{definition}[Nerve] There is a functor $N_c : \Cyc \to \Set^{\Xi^{op}}$ defined, on an object $O \in \Cyc$, by
	\[ N_c(O) = \Cyc(C(-), O) \in \Set^{\Xi^{op}}.\]
We will refer to $N_c$ as the \emph{cyclic dendroidal nerve}.
\end{definition}

Recall there is an analogous \emph{dendroidal} nerve $N_d: \Operad \to \Set^{\Omega^{op}}$ defined by $N_d(O)_T = \Operad(\Omega(T), O)$, where $\Omega(T)$ is the $\edge(T)$-colored operad freely generated by $T$.
This functor is a fully-faithful embedding, and the essential image may be characterized using \emph{inner horns} and \emph{Segal cores} from Definition \ref{horns and cores}.
\begin{theorem}[Moerdijk--Weiss; Cisinski--Moerdijk]\label{dendroidal nerve theorem}
	Let $X\in \Set^{\Omega^{op}}$ be a dendroidal set. The following are equivalent.
	\begin{enumerate}[(i')]
		\item $X\cong N_d(O)$ for some operad $O$.
		\item $X_T = \hom(\Omega[T], X) \to \hom(\Lambda^\delta\Omega[T], X)$ is a bijection for every inner coface map $\delta$.
		\item $X_T = \hom(\Omega[T], X) \to \hom(\Sc[T], X)$ is a bijection for every rooted tree $T$. 
	\end{enumerate}
\end{theorem}
\begin{proof}
	The equivalence of the first two was established in Proposition 5.3 and Theorem 6.1 of \cite{mw2}, while the third was shown to be equivalent to the first two in \cite[Corollary 2.6]{cm-ds}.
\end{proof}

The next section is dedicated to establishing an analogous theorem for the cyclic dendroidal nerve.

\begin{remark}
Example \ref{nonfull example}, shows, in particular, that the composite
\[ \begin{tikzcd}[row sep=tiny]
\Xi \rar & \Cyc \rar & \Set^{\Xi^{op}}	 \\	
S \arrow[mapsto]{rr} & & N_cC(S)
\end{tikzcd}
\]	
is not the Yoneda embedding $S\mapsto \Xi[S]$, since the Yoneda embedding is fully-faithful.
\end{remark}

\section{Cyclic operads and the nerve theorem}

Our goal in this section is to prove Theorem \ref{cyclic dendroidal nerve theorem}.
The method of proof is to reduce to the dendroidal case and apply Theorem \ref{dendroidal nerve theorem}.
To do this, we first need to understand the relationship between dendroidal Segal cores (resp.\ dendroidal inner horns) and their cyclic dendroidal analogues from Definition \ref{horns and cores}. 
The other main ingredient is to show that if $X$ is a cyclic dendroidal set so that $\iota^* X$ is the nerve of an operad, then $X$ was already the nerve of a cyclic operad. This is Theorem \ref{theorem creating cyclic operads}.

\begin{lemma}\label{segal inclusion different category lemma}
	If $T\in \Omega$ is a rooted tree, then
	\[
		\iota_!(\Sc[T] \to \Omega[T]) \cong (\Sc[\iota T] \to \Xi[\iota T]).
	\]
\end{lemma}
\begin{proof}
Given $S\in \Xi$, define a category $\mathcal C^S$ with 
$\Ob(\mathcal C^S) = \edge(S) \amalg \vertex(S)$, and non-identity maps $\{ e \to v \mid e\in \nbhd(v) \}$.
There is a functor $F : \mathcal C^S \to \Set^{\Xi^{op}}$ which on objects is given by $e \mapsto \Xi[ \eta ]$ and $v \mapsto \Xi[ \cstar_{|v|} ]$.
Define $F(\ord^v(k) \to v)$ to be the inclusion $k_* \colon  \Xi[ \eta ] \to \Xi[ \cstar_{|v|} ]$ which hits $k\in \Xi[\cstar_{|v|}]_\eta = \{0, 1, \dots, |v| - 1 \}$.
Then $\colim_{\mathcal C^S} F \cong \Sc[S]$.
A similar consideration applies for $\Sc[T] \subseteq \Omega[T]$; namely if $T$ is a rooted tree in $\Omega$, then there is a functor (using the same domain category from above) $F' : \mathcal C^T \to \Set^{\Omega^{op}}$ with $\colim_{\mathcal C^T} F' \cong \Sc[T] \subseteq \Omega[T]$.

Notice that $\iota_! \Omega[T] \cong \Xi[\iota T]$
since for any $X \in \SSet^{\Xi^{op}}$, they are the same on $\hom(-, X)$:
\[
	\hom(\iota_! \Omega[T], X) = \hom(\Omega[T], \iota^* X) = (\iota^* X)_T = X_{\iota T} = \hom(\Xi[T], X).
\]
In particular, $\iota_!(\Omega[\cstar_v]) = \Xi[\cstar_v]$ and $\iota_!(\Omega[\eta]) = \Xi[\eta]$.
Since $\iota_!$ is a left adjoint, it commutes with colimits, and we have
\begin{equation*}
	\iota_!(\Sc[T] \to \Omega[T]) \cong (\Sc[\iota T] \to \Xi[\iota T]).
\end{equation*}
\end{proof}

\begin{lemma}\label{coface interchange}
	Let $\delta : R \to T$ be any coface map in $\Omega$.
	Then \[ \iota_!(\Lambda^\delta\Omega[T] \hookrightarrow \Omega[T]) \cong (\Lambda^\delta\Xi[T] \hookrightarrow \Xi[T]).\]
\end{lemma}
\begin{proof}
Our strategy is the same as in the previous lemma -- we write both sides as a colimit of representables and then use that $\iota_!$ commutes with colimits and takes representables to representables.

Suppose that $S\in \Xi$. 
Let $V^S$ be a skeleton of the comma category $\Xi^+ \downarrow S$; there is an evident functor $A^S : V^S \to \Set^{\Xi^{op}}$ which sends $\phi : R \to S$ to the object $\Xi[R]$ and
\[ \begin{tikzcd}[column sep=tiny]
R' \arrow[dr,"\phi'"] \arrow[rr,"\alpha"] & & R \arrow[dl,"\phi"] \\
& S
\end{tikzcd} \] 
to $\alpha_* : \Xi[R'] \to \Xi[R]$.
Since $V^S$ has a terminal object (the one isomorphic to $\id_S$), $\colim_{V^S} A^S = \Xi[S]$.
If $\delta : R \to S$ is a coface map, write $V^S_\delta$ for the full subcategory of $V^S$ that excludes the two objects isomorphic to $\id_S$ and $\delta$.
Then there is an isomorphism
\[ \begin{tikzcd}
\colim\limits_{V^S_\delta} A^S|_{V^S_\delta} \dar \rar{\cong} & \Lambda^\delta\Xi[S] \dar[hook] \\
\colim\limits_{V^S} A^S \rar{\cong} & \Xi[S]. 
\end{tikzcd} \]
This isomorphism arises from the fact that $\Xi[R] \to \Xi[S]$ is a monomorphism whenever $R\to S$ is in $\Xi^+$ (Theorem~\ref{xi is ez category}) and the fact that all elements of $\Xi[S]$ (in particular, all elements of $\Lambda^\delta\Xi[S]$) factor through a minimal face by Definition~\ref{definition greedy}\eqref{Reedy def factorization}.

Suppose that $T$ is a rooted tree (with root $t_0$), and let $\phi : R \to T$ be any map in $\Xi^+$. 
If $R = \eta$, then $\phi$ is automatically oriented.
Assume that $R$ has at least one vertex. 
Then $\lifting_{t_0}(\phi) : \rootify(R, \findroot_{t_0}(\phi)) \to \rootify(T, t_0) = T$ is isomorphic to $\phi$ over $T$ and is also in $\Xi^+$. Thus we may as well have assumed that $\phi$ was oriented in the first place.
In particular, we may assume that all objects in $V^T$ are oriented maps.
With this assumption, all morphisms in $V^T$ are also oriented maps, for if $\phi \alpha = \phi'$ in $\Xi^+$ with $\phi$ and $\phi'$ oriented maps, then $\alpha$ is also oriented by Proposition \ref{functoriality of lifting}.
Thus we have a functor $B^T : V^T \to \Set^{\Omega^{op}}$ sending $R \to T$ to $\Omega[R]$, and with $\iota_! \circ B^T \cong A^T$. Further, as above, $\colim_{V^T_\delta} B^T|_{V^T_\delta}$ is isomorphic to $\Lambda^\delta\Omega[T]$ over $\Omega[T]$.
Concluding the proof, we have \[ \iota_! \Lambda^\delta\Omega[T] \cong \iota_!\left(  \colim_{V^T_\delta} B^T|_{V^T_\delta} \right) \cong \colim_{V^T_\delta} \left( \iota_! \circ B^T|_{V^T_\delta} \right) \cong  \colim_{V^T_\delta} A^T|_{V^T_\delta} \cong \Lambda^\delta\Xi[T].\]
\end{proof}

\begin{theorem}\label{theorem creating cyclic operads}
	Let $X \in \Set^{\Xi^{op}}$ be a cyclic dendroidal set and $O$ be a colored operad.
	If $\alpha : \iota^* X \overset{\cong}{\to} N_d(O)$ is an isomorphism of dendroidal sets, then there is a unique cyclic structure on $O$ so that $\alpha$ lifts to an isomorphism between $X$ and the cyclic dendroidal nerve of $O$. 
	In other words, there is a unique object $\tilde O \in \Cyc$ and an isomorphism $\tilde \alpha : X \to N_c(\tilde O)$ such that $U(\tilde O) = O$ and $\tilde \alpha_{\iota T} = \alpha_T : X_{\iota T} = (\iota_* X)_T \to N_d(O)_T = N_d(U\tilde O)_T = N_c(\tilde O)_{\iota T}$ for every rooted tree $T\in \Omega$.

	Further, if $\beta : \iota^* Y \to N_d(P)$ is another such isomorphism and $f :X \to Y$ is a morphism of $\Set^{\Xi^{op}}$, then there is a (unique) morphism $f' : \tilde O \to \tilde P$ in $\Cyc$ making the diagram
	\[ \begin{tikzcd}
	X \rar{\tilde \alpha} \dar{f} & N_c(\tilde O) \dar{N_c(f')} \\
	Y \rar{\tilde \beta} & N_c(\tilde P)
	\end{tikzcd} \]
	commute.
\end{theorem}

In order to prove this theorem, we need to delve a bit into how the operad structure of $O$ is manifested in the dendroidal set $N_d(O)$.

\newcommand{\retrieveprof}{\mathscr P}

Note that $\cstar_{n+1}$ from Example \ref{unrooted tree examples} is a rooted tree (which in the dendroidal setting is usually called $C_n \in \Omega$).
Let $i : \eta \to \cstar_{n+1}$ be the (oriented) map with image the single edge $i$; if $X \in \Set^{\Xi^{op}}$ write $\retrieveprof = \prod_{i=0}^n i^* : X_{\cstar_{n+1}} \to \prod_{i=0}^n X_\eta$.
If $O$ is a colored operad and $W = N_d(O)$, then $W_\eta = \colset(O)$ and $W_{\cstar_{n+1}} = \coprod_{c, c_1, \dots, c_n} O(c_1, \dots, c_n; c)$ with $\retrieveprof$ retrieving the profile $(c, c_1, \dots, c_n)$.
Now that we have the elements of $O$, let us examine the composition. For that, we will utilize the following trees.

\newcommand{\pgraft}{{\mathbf Z}}

\begin{figure}
\labellist
\small\hair 2pt
 \pinlabel {$a$} [B] at 40.5 57
 \pinlabel {$b$} [B] at 94.5 57
 \pinlabel {$a$} [B] at 212 57
 \pinlabel {$b$} [B] at 266 57
 \pinlabel {$b$} [B] at 381 57
 \pinlabel {$a$} [B] at 435 57
 \pinlabel {$e$} at 67.5 67
 \pinlabel {$e$} at 239 67
 \pinlabel {$e$} at 408 67
 \pinlabel {$4$} at 25.5 97
 \pinlabel {$3$} at 102.5 101
 \pinlabel {$2$} at 137.5 52
 \pinlabel {$1$} at 102.5 19
 \pinlabel {$0$} at 24.5 24
 \pinlabel {$0$} at 196 97
 \pinlabel {$4$} at 274 101
 \pinlabel {$3$} at 309 52
 \pinlabel {$2$} at 274 19
 \pinlabel {$1$} at 197 24
 \pinlabel {$1$} at 367 97
 \pinlabel {$0$} at 443 101
 \pinlabel {$4$} at 478 52
 \pinlabel {$3$} at 443 19
 \pinlabel {$2$} at 366 24
\endlabellist
\centering
\includegraphics[width=\textwidth]{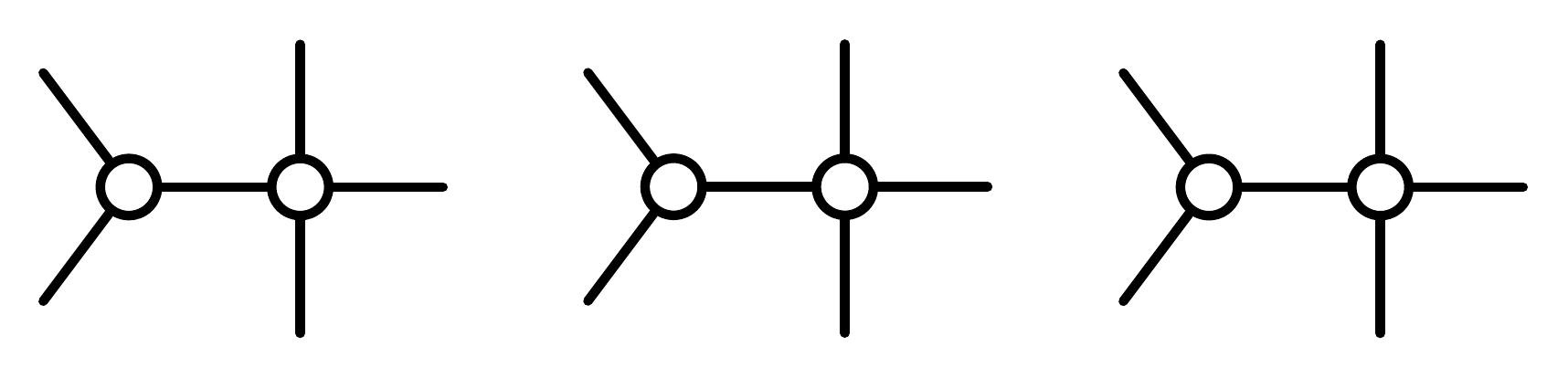}
	\caption{The rooted trees $\pgraft_{2,3}^1$, $\pgraft_{2,3}^2$, and $\pgraft_{3,2}^1$}\label{figure pgrafts}
\end{figure}

\begin{definition}
	Suppose that $m\geq 1$, $n\geq 0$, and $1\leq i \leq m$.
	Define a (rooted) tree $\pgraft_{m,n}^i$ with two vertices $a,b$, $\legs(\pgraft_{m,n}^i) = 0, 1, \dots, m+n-1$ (in that order), and one internal edge $e$.
	We further declare that the vertex neighborhoods are the following ordered sets
	\begin{align*}
		\nbhd(a) &= 0, 1, \dots, i-1, e, n+i, \dots, m+n-1 \\
		\nbhd(b) &= e, i, i+1, \dots, n+i-1.
	\end{align*}
	See Figure~\ref{figure pgrafts} for examples.
	This graph admits one inner coface map $\delta^e : \cstar_{n+m} \to \pgraft_{m,n}^i$ with $\delta^e(k) = k$ for all edges $k$, and two outer cofaces
	\begin{align*}
		\delta^a &: \cstar_{n+1} \to \pgraft_{m,n}^i & \delta^a(k) &= \begin{cases}
			e & \text{if $k=0$} \\
			i + k - 1 &\text{otherwise}
		\end{cases} & \delta^a(v) &= \cstar_b \\
		\delta^b &: \cstar_{m+1} \to \pgraft_{m,n}^i & \delta^b(k) &= \begin{cases}
			k & \text{if $0\leq k < i$} \\
			e & \text{if $k=i$} \\
			n - 1 + k &\text{if $i < k \leq m$.}
		\end{cases} & \delta^b(v) &= \cstar_a. 
	\end{align*}
All three of these cofaces are oriented.
\end{definition}

\begin{remark}\label{operadic mult}
	Let $O$ be a $\fC$-colored operad and $W = N_d(O)$ its dendroidal nerve. Then the operadic multiplication $\circ_i$ is recovered from
	\[
\begin{tikzcd}
	O(c_1, \dots, c_m; c_0) \times O(d_1, \dots, d_n; c_i)  \rar[hook] \arrow[dd, dashed, "\circ_i"] & 
	W_{\cstar_{m+1}} \times_{W_\eta} W_{\cstar_{n+1}} \dar["\cong","(d_b \times d_a)^{-1}" swap] 
\\& W_{\pgraft_{m,n}^i} \dar["d_e" swap]  
\\
	O(c_1, \dots, c_{i-1}, d_1, \dots, d_n, c_{i+1}, \dots, c_m; c_0) \rar[hook] & 
	W_{\cstar_{m+n}}
\end{tikzcd}\]
where the pullback in the top right is obtained from the maps $\cstar_{m+1} \overset{i}\hookleftarrow \eta \overset{0}\hookrightarrow \cstar_{n+1}$. 
\end{remark}

Given $\sigma \in \Sigma_{n}^+ = \Aut\{0,1,\dots, n\}$, there is a morphism $\psi_\sigma : \cstar_{n+1} \to \cstar_{n+1}$ determined by $(\psi_\sigma)_0 = \sigma$ on edges. The diagram 
\[ \begin{tikzcd}
X_{\cstar_{n+1}} \rar{\psi_\sigma^*}  \dar{\retrieveprof} & X_{\cstar_{n+1}}  \dar{\retrieveprof}   \\
\prod\limits_{i=0}^n X_{\eta} \rar{(-)\cdot \sigma} & \prod\limits_{i=0}^n X_{\eta}
\end{tikzcd} \]
commutes.
In particular, $\psi_\sigma^*$ restricts to a function
\[
	\retrieveprof^{-1}(x_0, \dots, x_n) \to \retrieveprof^{-1}(x_{\sigma(0)}, \dots, x_{\sigma(n)}).
\]

For the purposes of the next lemma, let us fix some notation.
For $q\geq 1$, let $\tau_q \in \Aut \{0, 1, \dots, q-1 \}$ be given by $\tau_q(k) = k +1 \mod q$.
Let $\psi^{q} = \psi_{\tau_q}: \cstar_{q} \to \cstar_{q}$ be the associated map.
If $1 \leq i \leq m-1$, define a map $\phi^{m,n} : \pgraft_{m,n}^i \to \pgraft_{m,n}^{i+1}$ by
\begin{align*}
	e &\mapsto e & a &\mapsto \cstar_a \\
	k &\mapsto \tau_{m+n}(k) & b &\mapsto \cstar_b.
\end{align*}
Likewise, if $n \geq 1$, define a map $\phi^{m,n} : \pgraft_{m,n}^m \to \pgraft_{n,m}^{1}$ by 
\begin{align*}
	e &\mapsto e & a &\mapsto \cstar_b \\
	k &\mapsto \tau_{m+n}(k) & b &\mapsto \cstar_a.
\end{align*}

The following lemma is a diagrammatic repackaging of Definition~\ref{defn cyclic operad}\eqref{equation cyclic operad condition}.

\begin{lemma}\label{psi and phi lemma}
	If $1 \leq i \leq m-1$, then the diagram 
\[ \begin{tikzcd}[column sep=small, row sep=small]
\cstar_{m+1} \arrow[dr, "\delta^b"] \arrow[dd, "\psi^{m+1}"] & & 
\cstar_{n+1} \arrow[dl, "\delta^a", swap] \arrow[dd, "\id" near start] 
\\ & 
\pgraft_{m,n}^i \arrow[dd, "\phi^{m,n}" description] & & 
\cstar_{m+n} \arrow[ll, crossing over, "\delta^e" near start] \arrow[dd, "\psi^{m+n}"] 
\\
\cstar_{m+1} \arrow[dr, "\delta^b"] & & 
\cstar_{n+1} \arrow[dl, "\delta^a", swap] 
\\ & 
\pgraft_{m,n}^{i+1} & & 
\cstar_{m+n} \arrow[ll, "\delta^e" near start] \\
\end{tikzcd} \]
	commutes.
If $i=m$ and $n\geq 1$, then the diagram
\[ \begin{tikzcd}[column sep=small, row sep=small]
\cstar_{m+1} \arrow[dr, "\delta^b"] \arrow[dd, "\psi^{m+1}"] & & 
\cstar_{n+1} \arrow[dl, "\delta^a", swap] \arrow[dd, "\psi^{n+1}" near start] 
\\ & 
\pgraft_{m,n}^m \arrow[dd, "\phi^{m,n}" description] & & 
\cstar_{m+n} \arrow[ll, crossing over, "\delta^e" near start] \arrow[dd, "\psi^{m+n}"] 
\\
\cstar_{m+1} \arrow[dr, "\delta^a"] & & 
\cstar_{n+1} \arrow[dl, "\delta^b", swap] 
\\ & 
\pgraft_{n,m}^{1} & & 
\cstar_{m+n} \arrow[ll, "\delta^e" near start] \\
\end{tikzcd} \]
	commutes.
\end{lemma}
\begin{proof}
	One merely needs to check that the composite functions are identical on edge sets, which is straightforward.
\end{proof}

\begin{proof}[Proof of Theorem \ref{theorem creating cyclic operads}]
	We begin by defining the cyclic structure on $O$.
	Let $c_0, c_1, \dots, c_n \in \colset(O)$, and write $x_i := \alpha^{-1} (c_i) \in X_\eta$.
Then $\alpha$ determines the isomorphism
	\[ \begin{tikzcd}
 \retrieveprof^{-1}(x_0, \dots, x_n) \rar[hook] \dar[dashed,"\alpha", "\cong" swap]  & X_{\cstar_{n+1}} \dar["\alpha", "\cong" swap]  \rar{\retrieveprof} & \prod\limits_{i=0}^n X_{\eta} \dar["\alpha", "\cong" swap]  \\
 O(c_1, \dots, c_n; c_0) \rar[hook] & N_d(O)_{C_n} \rar{\retrieveprof}  & \prod\limits_{i=0}^n N_d(O)_{\eta} 
	\end{tikzcd} \]
on the left.
We define, for $\sigma \in \Sigma_n^+$,
\[
	(-)\cdot \sigma : O(c_1, \dots, c_n; c_0) \to O(c_{\sigma(1)}, \dots, c_{\sigma(n)}; c_{\sigma(0)})
\]
as the restriction of the endomorphism $\alpha^{-1} \circ \psi_\sigma^* \circ \alpha$ of $N_d(O)_{\cstar_{n+1}}$.
When $\sigma$ fixes $0$, this is precisely the usual symmetric group action.

Now suppose that we are working with $\tau_{n+1}$ and consider
\[ (-)\cdot \tau_{n+1} : O(c_1, \dots, c_n; c_0) \to O(c_2, \dots, c_n, c_0; c_1). \]
which is the restriction of $\alpha^{-1} \circ \psi_{n+1} \circ \alpha$.
By Theorem~\ref{dendroidal nerve theorem}, we simply need to show that $\{ (-) \cdot \tau_{n+1} \}$ are compatible with operadic multiplication; using the notation from Remark \ref{operadic mult}, we have that $\circ_i = d_e \circ (d_b \times d_a)^{-1}$ on $W = N_d(O)$.

In the diagrams in Figure~\ref{figure big cube one} and Figure~\ref{figure big cube two}, we have $W_\eta = \colset(O)$ and \[ W_{\cstar_{m+1}} = O_m = \coprod_{W_\eta^{\times m + 1}} O(c_1, \dots, c_m; c_0) \] since $W = N_d(O)$.
By Lemma \ref{psi and phi lemma}, the diagram in Figure~\ref{figure big cube one}
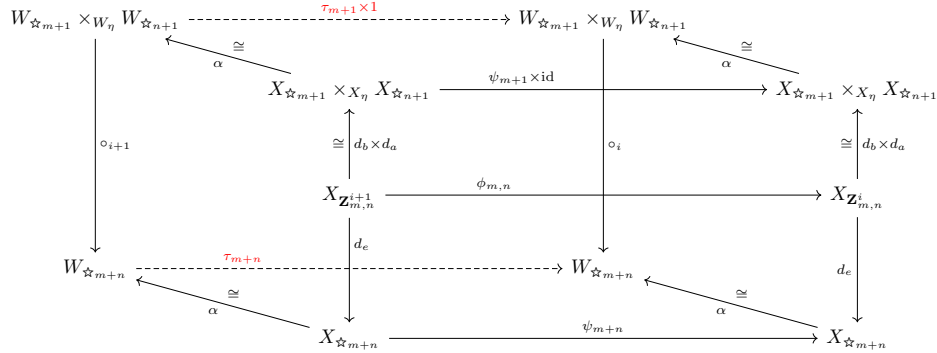
\begin{figure}[htb]
\[
\resizebox{\textwidth}{!}{%
\begin{tikzcd}[column sep=large,ampersand replacement=\&] 
W_{\cstar_{m+1}} \times_{W_\eta} W_{\cstar_{n+1}} \arrow[rr, dashed, "\tau_{m+1} \times 1" color=red] \arrow[dddd, "\circ_{i+1}"] \& \& W_{\cstar_{m+1}} \times_{W_\eta} W_{\cstar_{n+1}} \arrow[dddd, "\circ_i"]
\\
\& X_{\cstar_{m+1}} \times_{X_\eta} X_{\cstar_{n+1}}  \arrow[rr,"\psi_{m+1} \times \id" near start]
	\arrow[ul, "\alpha", "\cong" swap]
\& \& 
X_{\cstar_{m+1}} \times_{X_\eta} X_{\cstar_{n+1}} 
\arrow[ul, "\alpha", "\cong" swap]
\\ \\
\& X_{\pgraft_{m,n}^{i+1}} \arrow[dd,"d_e" near start]  \arrow[uu,"\cong","d_b \times d_a" swap]  \arrow[rr,"\phi_{m,n}" near start]  \& \&  
X_{\pgraft_{m,n}^i} \arrow[dd,"d_e" swap]  \arrow[uu,"\cong","d_b \times d_a" swap]  
\\ 
W_{\cstar_{m+n}} \arrow[rr, dashed, "\tau_{m+n}" {color=red, near start}]\& \&
	W_{\cstar_{m+n}}
\\
\& X_{\cstar_{m+n}} \arrow[rr,"\psi_{m+n}"] \arrow[ul, "\alpha", "\cong" swap]\& \&
	X_{\cstar_{m+n}} \arrow[ul, "\alpha", "\cong" swap]
\end{tikzcd}
} \]
\caption{Compatibility with operadic composition}\label{figure big cube one}
\end{figure}
commutes for $1\leq i \leq m-1$
and the diagram in Figure~\ref{figure big cube two}
\begin{figure}[htb]
\[
\resizebox{\textwidth}{!}{%
\begin{tikzcd}[column sep=large,ampersand replacement=\&] 
W_{\cstar_{n+1}} \times_{W_\eta} W_{\cstar_{m+1}} \arrow[rr, dashed, "(\tau_{m+1} \times \tau_{n+1})\circ \sigma" color=red] \arrow[dddd, "\circ_1"] \& \& W_{\cstar_{m+1}} \times_{W_\eta} W_{\cstar_{n+1}} \arrow[dddd, "\circ_m"]
\\
\& X_{\cstar_{n+1}} \times_{X_\eta} X_{\cstar_{m+1}}  \arrow[rr,"(\psi_{m+1} \times \psi_{n+1})\circ \sigma" near start]
	\arrow[ul, "\alpha", "\cong" swap]
\& \& 
X_{\cstar_{m+1}} \times_{X_\eta} X_{\cstar_{n+1}} 
\arrow[ul, "\alpha", "\cong" swap]
\\ \\
\& X_{\pgraft_{n,m}^{1}} \arrow[dd,"d_e" near start]  \arrow[uu,"\cong","d_b \times d_a" swap]  \arrow[rr,"\phi_{m,n}" near start]  \& \&  
X_{\pgraft_{m,n}^m} \arrow[dd,"d_e" swap]  \arrow[uu,"\cong","d_b \times d_a" swap]  
\\ 
W_{\cstar_{m+n}} \arrow[rr, dashed, "\tau_{m+n}" {color=red, near start}]\& \&
	W_{\cstar_{m+n}}
\\
\& X_{\cstar_{m+n}} \arrow[rr,"\psi_{m+n}"] \arrow[ul, "\alpha", "\cong" swap]\& \&
	X_{\cstar_{m+n}} \arrow[ul, "\alpha", "\cong" swap]
\end{tikzcd}
} \]
\caption{Compatibility with operadic composition}\label{figure big cube two}
\end{figure}
commutes for $i=m$ (where $\sigma$ interchanges the two factors).
This shows that the proposed $\Sigma_n^+$ actions really turn $O$ into a cyclic operad.
Uniqueness of the structure is clear from the requirement that $\alpha$ lifts to an isomorphism of cyclic dendroidal sets.
Finally, the existence of a unique $f'$ in the second part is just a characterization of cyclic operad maps as operad maps which respect the extra symmetry, the fact that $N_d$ is fully-faithful, and that the $f'$ guaranteed by fullness of $N_d$ preserves the extra symmetry by inspection.
\end{proof}

We are now ready to prove the main theorem about the cyclic dendroidal nerve.
Note that our nerve theorem does not fit into the general monadic framework for nerve theorems from \cite{weber,bmw}, as, for the reasons outlined in the introduction, we have chosen a non-full subcategory $\Xi$ of $\Cyc$ as our indexing category (contrast with the paragraph after \cite[Definition 2.3]{bmw}, where the indexing category $\Theta_T$ is always a full subcategory).

\begin{theorem}\label{cyclic dendroidal nerve theorem}
The functor $N_c : \Cyc \to \Set^{\Xi^{op}}$ is fully-faithful.
Further, if $X\in \Set^{\Xi^{op}}$ is a cyclic dendroidal set, then the following are equivalent.
	\begin{enumerate}[(i)]
		\item $X\cong N_c(O)$ for some cyclic operad $O$.
		\item $X_S = \hom(\Xi[S], X) \to \hom(\Lambda^\delta\Xi[S], X)$ is a bijection for every inner coface map $\delta$.
		\item $X_S = \hom(\Xi[S], X) \to \hom(\Sc[S], X)$ is a bijection for every rooted tree $S$. 
	\end{enumerate}
\end{theorem}
\begin{proof}
That $N_c$ is fully-faithful follows from the existence and uniqueness of $f'$ in the second part of Theorem \ref{theorem creating cyclic operads}.

The remainder of the proof relies on Theorem \ref{dendroidal nerve theorem}, and goes by comparing properties of $X$ to properties of $\iota^*X$.
Notice, for instance, that Theorem \ref{theorem creating cyclic operads} implies that $\iota^*X$ satisfies (i') if and only if $X$ satisfies (i).

Since every tree $S\in \Xi$ is isomorphic to a rooted tree $T\in \Omega$, it suffices to examine the remaining condition only for rooted trees $T$. Likewise, in (ii), we may assume that the inner coface map $\delta: R \to T$ is an oriented map.
So, to see that (ii) holds for $X$ if and only if (ii') holds for $\iota^*X$, we merely need to note that the commutative diagram
\[ \begin{tikzcd}
\hom(\Omega[T], \iota^* X) \rar \dar{\cong} & \hom(\Lambda^\delta\Omega[T], \iota^* X)  \dar{\cong} \\
\hom(\iota_!\Omega[T], X) \rar \dar{\cong}& \hom(\iota_!\Lambda^\delta\Omega[T], X) \dar{\cong}\\
\hom(\Xi[T], X) \rar & \hom(\Lambda^\delta\Xi[T], X)
\end{tikzcd} \]
has the indicated isomorphisms by Lemma \ref{coface interchange}. A similar argument shows that (iii) holds for $X$ if and only if (iii') holds for $\iota^*X$, this time applying Lemma \ref{segal inclusion different category lemma}.
\end{proof}

\section{Berger--Moerdijk Reedy model structure}\label{section bmr model structure}

In this section, we first recall the Berger--Moerdijk Reedy model structure on categories of diagrams indexed by a generalized Reedy category $\Rr$ (Definition \ref{definition greedy}) and investigate properties of this model structure.
The basic reference is the paper \cite{bm}, while many of these results in the case of classical Reedy categories may be found in \cite[Chapter 15]{hirschhorn}.
We then turn to the special case of simplicial $\Rr$-presheaves, and consider the full subcategory $\SSet^{\Rr^{op}}_\ast \subseteq \SSet^{\Rr^{op}}$ of diagrams which are a point in degree zero.
For certain $\Rr$, this category admits a model structure (Theorem~\ref{reedy type model structure for discrete}) which is simplicial (Theorem~\ref{existence simplicial structure discrete}), left proper, and cellular (Proposition~\ref{reedy type localizability}).
The results of this section are likely well-known among experts, at least in certain special cases.

A model category $\mathcal M$ is called \emph{$\Rr$-projective} if, for every $r\in \Rr$, the category of (right) $\Aut(r)$-equivariant maps $\mathcal M^{\Aut(r)}$ admits the model structure where weak equivalences and fibrations are detected in $\mathcal M$.
This occurs, for instance, if $\mathcal M$ is cofibrantly generated (see \cite[11.6.1]{hirschhorn}) or if $\Aut(r) = \{e\}$ for all $r\in \Rr$.

If $r\in \Rr$, let $\Rr^+(r)$ be the category whose objects are \emph{non-invertible} maps in $\Rr^+$ with codomain $r$, and whose morphisms $\alpha \to \beta$ are commutative triangles
\[ \begin{tikzcd}[column sep=tiny]
s \arrow{rr}{\sigma} \arrow{dr}[swap]{\alpha} & & s' \arrow{dl}{\beta}\\
& r
\end{tikzcd} \]
in $\Rr^+$.
This is a full subcategory of $\Rr^+ \downarrow r$.
Similarly, for each $r\in \Rr$, there is a full subcategory $\Rr^-(r) \subseteq r \downarrow \Rr^-$ whose objects are non-invertible morphisms with domain $r$.

If $Z\in \mathcal M^{\Rr}$, define, for each $r\in \Rr$, \emph{latching} and \emph{matching} objects
\begin{align*}
		L_r Z &= \colim_{
	\substack{
	\mathbb R^+(r) \\
	\alpha \colon s \to r
	}}
	Z_s & 
	M_r Z &= \lim_{
	\substack{
	\mathbb R^-(r) \\
	\alpha \colon r \to s
	}}
	Z_s
\end{align*}
which come equipped with maps
\[
	L_r Z \to Z_r \to M_r Z
\]
in $\mathcal M^{\Aut(r)}$.

\begin{theorem}[Theorem 1.6, \cite{bm}]\label{bm_reedy_ms}
If $\Rr$ is a generalized Reedy category and $\mathcal M$ is an $\Rr$-projective model category, then the diagram category $\mathcal M^{\Rr}$ admits a model structure where $f : X \to Y$ is a
\begin{itemize}
	\item weak equivalence if, for each $r\in \Rr$, $f_r : X_r \to Y_r$ is a weak equivalence in $\mathcal M$;
	\item  cofibration if, for each $r\in \Rr$, \[ X_r \amalg_{L_rX} L_rY \to Y_r\] is a cofibration in $\mathcal M^{\Aut(r)}$; and
	\item fibration if, for each $r\in \Rr$, \[X_r \to  M_r X \times_{M_r Y} Y_r\] is a fibration in $\mathcal M$.
\end{itemize}
\end{theorem}

We will call this model structure the \emph{Berger--Moerdijk Reedy} model structure on $\mathcal M^\Rr$.
It would be convenient to know when this model structure is left proper and cellular, so that we can guarantee the existence of left Bousfield localizations \cite[4.1.1]{hirschhorn}.

\begin{theorem}\label{left properness}
	If $\Rr$ is a generalized Reedy category and $\mathcal M$ is left proper and $\Rr$-projective, then the Berger--Moerdijk Reedy model structure on $\mathcal M^{\Rr}$ is left proper.
\end{theorem}

A key component of the proof is that Reedy cofibrations are levelwise cofibrations (Lemma~\ref{reedy levelwise lemma}); in establishing this lemma, we will utilize the fact that $\mathcal M^G \to \mathcal M$ preserves cofibrations when $G$ is a finite group.

\begin{remark}\label{remark preservation of cofibrations}
	Suppose that $G$ is any discrete group and $\mathcal M$ is a category so that products indexed by $G$ exist. The forgetful functor $f: \mathcal M^G \to \mathcal M$ has a right adjoint $r$ so that the composite $\mathcal M \xrightarrow{r} \mathcal M^G \xrightarrow{f} \mathcal M$ takes an object $X$ to $\prod_{g\in G} X$, and similarly for morphisms.
	If $\mathcal M$ is a model category, then the composite $fr$ takes acyclic fibrations to acyclic fibrations.
	Suppose further that $\mathcal M^G$ admits the projective model structure, that is, the model structure so that $f$ creates fibrations and weak equivalences.
	In this case, $r$ preserves acyclic fibrations, hence the left adjoint $f: \mathcal M^G \to \mathcal M$ preserves cofibrations.
\end{remark}

\begin{lemma}\label{reedy levelwise lemma}
If $\Rr$ is a generalized Reedy category, $\mathcal M$ is $\Rr$-projective, and $f: A \to B$ is a Reedy cofibration, then for each $r\in \Rr$ the map 
$f_r : A_r \to B_r$ is a cofibration in $\mathcal M$.
\end{lemma}
\begin{proof}
A variation on \cite[Lemma 5.3]{bm} shows that, for each $r\in \Rr$, the map $L_rA \to L_rB$ is a cofibration in $\mathcal M$.
As it is a bit simpler than that lemma, let us give a sketch of the argument.
The category $\mathbb S = \Rr^+(r)$ is again generalized Reedy, with $\mathbb S = \mathbb S^+$ and the domain functor $\varphi : \mathbb S \to \Rr$ is a morphism of generalized Reedy categories.
For any $s\to r$ in $\mathbb S$, the diagram
\[ \begin{tikzcd}
L_{s\to r} (\varphi^*A) \rar{\cong} \dar  & L_s(A) \dar \\
(\varphi^*A)_{s\to r} \rar{=} & A_s
\end{tikzcd} \]
commutes where the top isomorphism is a consequence of \cite[Lemma 4.4(i)]{bm}.
Thus, $\varphi^*(f)$ is a cofibration in $\mathcal M^{\mathbb S}$.
Further, since $\colim_{\mathbb S} (\varphi^*A) = L_r(A)$, we have
\[
	\left(L_rA \xrightarrow {L_rf} L_rB\right) = \colim_{\mathbb S} \left(\varphi^* A \xrightarrow{\varphi^*(f)} \varphi^* B\right)
\]
which is a cofibration in $\mathcal M$ by \cite[Corollary 1.7]{bm}.

The map $A_r \to A_r \amalg_{L_rA} L_r B$
is a pushout of a cofibration 
\[ \begin{tikzcd}
L_rA \dar[tail] \rar & A_r \dar \\
L_rB \rar & A_r \amalg_{L_rA} L_r B
\end{tikzcd} \]
hence is a cofibration \cite[Proposition 7.2.12(a)]{hirschhorn}. 
Since $f$ is a Reedy cofibration,
\[ A_r \amalg_{L_rA} L_r B \to B_r\]
is a cofibration in $\mathcal M^{\Aut(r)}$.
By Remark~\ref{remark preservation of cofibrations}, it is also a cofibration in $\mathcal M$.
Now $f_r : A_r \to B_r$ is the composite of two cofibrations in $\mathcal M$
\[
	A_r \to A_r \amalg_{L_rA} L_r B \to B_r, 
\]
hence is also a cofibration in $\mathcal M$.
\end{proof}

\begin{proof}[Proof of Theorem \ref{left properness}]
Suppose we have a pushout diagram
\begin{equation}\label{left proper pushout}
\begin{tikzcd}
A \dar{\simeq} \rar[tail] & B \dar \\
X \rar & X \amalg_A B
\end{tikzcd} \end{equation}
in $\mathcal M^{\mathbb R}$; we wish to show that $B \to X \amalg_A B$ is a weak equivalence.
Evaluating \eqref{left proper pushout} at $r\in \mathbb R$ gives the pushout square
\[ \begin{tikzcd}
A_r \dar{\simeq} \rar[tail] & B_r \dar \\
X_r \rar & X_r \amalg_{A_r} B_r.
\end{tikzcd} \] 
By Lemma \ref{reedy levelwise lemma}, $A_r \to B_r$ is a cofibration in $\mathcal M$; since $\mathcal M$ is left proper, this implies that $B_r \to X_r \amalg_{A_r} B_r$ is a weak equivalence.
Since $r$ was arbitrary, $B \to X \amalg_A B$ is a weak equivalence.
\end{proof}

\begin{proposition}\label{cellularity}
	If $\mathcal M$ is cellular, then so is $\mathcal M^\Rr$.
\end{proposition}
\begin{proof}
	The proof given in \cite[15.7]{hirschhorn} goes through, with the caveat that in the proof of \cite[15.7.1]{hirschhorn}, one should use Lemma \ref{reedy levelwise lemma} in place of \cite[15.3.11]{hirschhorn}.
\end{proof}

\begin{remark}\label{remark on generating cofibrations}
	Implicit in the explanation for Proposition~\ref{cellularity} is the fact that the generating (acyclic) cofibrations of $\mathcal{M}^{\Rr}$ are described as in \cite[Definition 15.6.23]{hirschhorn}. Indeed, the proof of cofibrant generation in the first paragraph of \cite[Theorem 7.6]{bm} only relies on $\Rr$ being a dualizable generalized Reedy category and $\mathcal M$ being cofibrantly generated. We will need this explicit description of the generating (acyclic) cofibrations for the proof of Theorem~\ref{reedy type model structure for discrete} (in the special case when $\mathcal M = \SSet$).
\end{remark}

\subsection{Reduced presheaves in simplicial sets}\label{reduced presheaves}

Let $\Rr$ be a generalized dualizable Reedy category. 
We assume that $\Rr$ has a unique object $\eta$ of degree zero and, for any $r\in \Rr$, the set $\Rr(r,\eta)$ is either empty or has exactly one element.
Further, we assume that if $\Rr(r,\eta) = \{f\}$, then $\Rr(\eta, r) \neq \varnothing$ as well; any inhabitant of $\Rr(\eta, r)$ is automatically a section of $f$.
In symbols, this says 
\begin{equation}\label{bound on r eta}
	|\Rr(r,\eta)| \leq \min (1, |\Rr(\eta,r)|) \qquad \forall r\in \Rr
\end{equation}
which we take as a standing assumption for all that follows.
Examples of such categories include $\Delta$, $\Omega$, and $\Xi$.

Let $\SSet^{\Rr^{op}}_\ast$ be the full subcategory of $\SSet^{\Rr^{op}}$ consisting of those $X$ so that $X_\eta = \Delta[0]$.
As this category has been thoroughly analyzed in \cite{MR2342822} in the case of $\Rr = \Delta$ and \cite{bh1} in the case of $\Rr = \Omega$, many of the arguments and constructions in the remainder of this section should look familiar.
Write
\[
	\incl \colon \SSet^{\Rr^{op}}_\ast \hookrightarrow \SSet^{\Rr^{op}}
\]
for the inclusion.
This functor admits a left adjoint 
\[
	\red \colon \SSet^{\Rr^{op}} \to \SSet^{\Rr^{op}}_\ast
\]
which we now describe explicitly.

\begin{definition} We define the \emph{reduction} of an object $X\in \SSet^{\Rr^{op}}$.
\begin{itemize}
\item
A map $\Rr[\eta] \times X_\eta \to X$ is given in degree $r$ by either $X_\eta \xrightarrow{f^*} X_r$ if $\Rr(r,\eta) = \{f\}$ or else the unique map $\varnothing \to X_r$ if $\Rr(r,\eta) = \varnothing$.
Since such an $f$ admits a section, this map is a monomorphism, and we write $X^{(0)} \subseteq X$ for its image.
Notice that there is a unique map $X^{(0)} \to \Rr[\eta]$ and that 
\[
	X^{(0)}_r \cong \begin{cases}
		X_\eta & \text{if }\Rr(r,\eta) = * \\
		\varnothing & \text{if }\Rr(r,\eta) = \varnothing.
	\end{cases}
\]
\item
Define $\red(X)$ as the pushout
\[ \begin{tikzcd}
X^{(0)} \rar \dar & X \dar \\
\Rr[\eta] \rar & \red(X).
\end{tikzcd} \]
\end{itemize}
\end{definition}

Suppose $X \in \SSet^{\Rr^{op}}$.  
If there is not a map $r\to \eta$, then $\red(X)_r = X_r$. 
Suppose $Z \in \SSet^{\Rr^{op}}_\ast$. If there is a map from $r$ to $\eta$ then it is unique, $\Rr(r,\eta) = \{ f \}$; this implies that $Z_r$ has a natural basepoint given by $f^*(t)$, where $t\in Z_\eta$ is the unique element.

\begin{remark}
	Many results of this section hold with a weaker assumption than \eqref{bound on r eta}, namely one could just assume $|\Rr(r,\eta)| \leq 1$ for all $r\in \Rr$. 
	In this case, $\Rr[\eta] \times X_\eta \to X$ may not be a monomorphism, and we should define $X^{(0)}$ to be $\Rr[\eta] \times X_\eta$ instead of its image in $X$.
	We refrain from giving further details, as our main application is to the categories $\Xi$ and $\Omega$ (which do satisfy \eqref{bound on r eta}).
\end{remark}

The category $\SSet^{\Rr^{op}}_\ast$ is bicomplete:
Limits and directed colimits are computed in the larger category $\SSet^{\Rr^{op}}$, while finite coproducts are given by the formula
\begin{equation}\label{coproduct reduced}
	\left(X\amalg Y \right)_r  = \begin{cases}
		X_r \coprod Y_r & \Rr(r,\eta) = \varnothing \\
		X_r \bigvee Y_r & \Rr(r,\eta) = *.
	\end{cases}
\end{equation}
Equivalently, $X \amalg Y = \red(\incl(X) \amalg \incl(Y))$. Though $\incl$ does not preserve coproducts, it does preserve pushouts.

\begin{lemma}\label{pushout lemma}
	The inclusion functor $\incl : \SSet^{\Rr^{op}}_\ast \to \SSet^{\Rr^{op}}$ preserves connected colimits.
\end{lemma}
\begin{proof}
Suppose $F : \mathcal{C} \to \SSet^{\Rr^{op}}_\ast$ is a functor from a connected category, and $A = \colim_{\mathcal{C}} \incl F$.
Then
\[
	A_\eta = (\colim_{c\in \mathcal{C}} \incl F(c))_\eta = \colim_{c\in \mathcal{C}} (\incl F(c)_\eta) = \colim_{c\in \mathcal{C}} \Delta[0] = \Delta[0].
\]
	Since $A$ is already in $\SSet^{\Rr^{op}}_\ast$ and $\incl$ is fully-faithful, $A$ is also the colimit of $F$ in $\SSet^{\Rr^{op}}_\ast$.
\end{proof}

Given an object $r\in \Rr$, recall that the \emph{boundary of $r$} is the subobject $\partial \Rr[r] \subseteq \Rr[r] = \Rr(-,r)$ so that $\partial \Rr[r]_s \subseteq \Rr(s,r)$ consists of those maps $s \to r$ which factor through an object of degree less than $d(r)$.
In particular, if $d(s) < d(r)$, $\partial \Rr[r]_s = \Rr[r]_s$. 

\begin{definition}\label{generating cofibrations}
We define two sets of maps in $\SSet^{\Rr^{op}}$.
The set $I$ consists of all inclusions
\[
	(\partial \Rr[r] \times \Delta[n]) \bigcup_{\partial \Rr[r] \times  \partial \Delta[n]} (\Rr[r] \times \partial \Delta[n]) \to \Rr[r] \times \Delta[n]
\]
for $n\geq 0$ and $r \in \Rr$.
The set $J$ consists of all inclusions 
\[
	(\partial \Rr[r] \times \Delta[n]) \bigcup_{\partial \Rr[r] \times  \Lambda^k[n]} (\Rr[r] \times \Lambda^k[n]) \to \Rr[r] \times \Delta[n].
\]
The set $I$ (resp.\ $J$) is a set of generating (acyclic) cofibrations for the Berger--Moerdijk Reedy model structure on $\SSet^{\Rr^{op}}$ (see Remark~\ref{remark on generating cofibrations}).
\end{definition}

We now make a further restriction on our generalized Reedy category $\Rr$, which will imply that all (co)domains of elements of $I\cup J$ are small relative to the whole category.

\begin{lemma}\label{smallness lemma}
All objects of $\SSet^{\Rr^{op}}$ and $\SSet^{\Rr^{op}}_\ast$ are small \cite[Definition 10.4.1]{hirschhorn}.
\end{lemma}
\begin{proof}
In any locally presentable category, every object is small in this sense. 
Any diagram category $\Set^{\mathcal{C}}$ (with $\mathcal{C}$ small) is locally presentable \cite[Example 1.12]{AdamekRosicky:LPAC}, hence $\SSet^{\Rr^{op}} \cong \Set^{\Delta^{op} \times \Rr^{op}}$ is.
Finally, $\SSet^{\Rr^{op}}_\ast$ is a full, reflective subcategory of $\SSet^{\Rr^{op}}$ which is further closed under $\lambda$-directed colimits ($\lambda$ a regular cardinal) by Lemma~\ref{pushout lemma}, hence is locally presentable by \cite[Representation Theorem 1.46]{AdamekRosicky:LPAC}.
\end{proof}

\begin{lemma}\label{zero skeleta lemma}
Let $A$ be a domain or codomain of an element of $I\cup J$, with $A \subseteq \Rr[r] \times \Delta[n]$.
If $d(r) > 0$, then $A^{(0)} = \Rr[r]^{(0)} \times \Delta[n] \cong \coprod_{\Rr(\eta, r)} \Rr[\eta] \times \Delta[n]$. 
If $r= \eta$, then $A^{(0)} = A$, hence $\red(A) = \Rr[\eta]$.
\end{lemma}
\begin{proof}
The important point is that 
\[
	(\partial \Rr[r])^{(0)} = \begin{cases}
		\Rr[r]^{(0)} & d(r) > 0 \\
		\varnothing & d(r) = 0.
	\end{cases}
\]
If $d(r) > 0$, then 
\[
	\Rr[r]^{(0)} \times \Delta[n] = (\partial \Rr[r] \times \Delta[n] )^{(0)} \subseteq A^{(0)} \subseteq (\Rr[r] \times \Delta[n])^{(0)} = \Rr[r]^{(0)} \times \Delta[n].
\]
If $r = \eta$, then $A$ is one of $\Rr[\eta] \times \partial \Delta[n]$, $\Rr[\eta]\times \Lambda^k[n]$, or $\Rr[\eta]\times \Delta[n]$, hence $A^{(0)} = A$.
\end{proof}

\begin{proposition}\label{reduction of generator one}
Suppose that $j : A \to \Rr[r] \times \Delta[n]$ is in $I$ or $J$ and $d(r) > 0$.
If $j\in I$, then $\red(j)$ is a cofibration in $\SSet^{\Rr^{op}}$, while if $j\in J$, then $\red(j)$ is an acyclic cofibration in $\SSet^{\Rr^{op}}$.
\end{proposition}
\begin{proof}
	The cube with front and back squares pushouts
\[ \begin{tikzcd}[row sep=small, column sep=small]
A^{(0)}
	\arrow[rrrr] 
	\arrow[dr, "=", "\ref{zero skeleta lemma}" swap] 
	\arrow[ddd] &&&& 
A
	\arrow[dr] 
	\arrow[ddd]
\\ & 
(\Rr[r] \times \Delta[n])^{(0)}
	\arrow[rrrr,  crossing over] &&&& 
\Rr[r] \times \Delta[n]
\\ \\
\Rr[\eta]	
	\arrow[rrrr] 
	\arrow[dr, "="] &&&& 
\red(A)
	\arrow[dr] 
\\ & 
\Rr[\eta]
	\arrow[rrrr] 
	\arrow[from=uuu, crossing over] 
	&&&& 
\red(\Rr[r] \times \Delta[n])
	\arrow[from=uuu, crossing over] 
\end{tikzcd} \]
reduces to a rectangle 
\[ \begin{tikzcd}
A^{(0)} \rar\dar  & A \rar\dar  & \Rr[r] \times \Delta[n]\dar \\
\Rr[\eta] \rar & \red(A) \rar & \red(\Rr[r] \times \Delta[n])
\end{tikzcd} \]
where the big rectangle and the left square are pushouts, hence so is the right square.

The result now follows since the class of (acyclic) cofibrations is closed under pushouts.
\end{proof}

\begin{proposition}\label{reduction of generator two}
Suppose that $j : A \to \Rr[\eta] \times \Delta[n]$ is in $I$ or $J$. Then $\red(j)$ is the identity map on $\Rr[\eta]$.
\end{proposition}
\begin{proof}
This is a direct consequence of the second part of Lemma \ref{zero skeleta lemma} and the fact that $\hom(\Rr[\eta], \Rr[\eta]) = \Rr(\eta, \eta) = *$.
\end{proof}

\begin{lemma}\label{levelwise acyclic cof}
If $j\in J$, then $\red(j)_s$ is an acyclic cofibration in $\SSet$ for every $s\in \Rr$.
If $j\in I$, then $\red(j)_s$ is a cofibration in $\SSet$ for every $s\in \Rr$.
\end{lemma}
\begin{proof}
By the previous two propositions, we know that $\red(j)$ is an (acyclic) cofibration in the Berger--Moerdijk Reedy model structure on $\SSet^{\Rr^{op}}$.
The result then follows from Lemma \ref{reedy levelwise lemma}.
\end{proof}

\begin{theorem}\label{reedy type model structure for discrete}
	Suppose that $\Rr$ is a generalized Reedy category with unique element $\eta$ in degree zero and $|\Rr(r,\eta)| \leq \min (1, |\Rr(\eta,r)|)$ for all $r \in \Rr$.
	Then the Berger--Moerdijk Reedy model structure lifts along the adjunction \[ \red : \SSet^{\Rr^{op}} \rightleftarrows \SSet^{\Rr^{op}}_\ast  : \incl. \]
	In other words, $\SSet^{\Rr^{op}}_\ast$ admits a cofibrantly-generated model structure with 
	weak equivalences (resp.\ fibrations) those maps which are weak equivalences (resp.\ fibrations) in the larger category $\SSet^{\Rr^{op}}$.
\end{theorem}
\begin{proof}
We apply Kan's lifting theorem \cite[11.3.2]{hirschhorn}.
We know that $\red I$ and $\red J$ satisfy the small object argument by Lemma \ref{smallness lemma}, hence condition (1) is satisfied.
For condition (2) we know that $\red(j)_s$ is an acyclic cofibration in $\SSet$ for any $s\in \Rr$ by Lemma \ref{levelwise acyclic cof}.
Then given any relative $\red J$-cell complex $X\to Y$ in $\SSet^{\Rr^{op}}_\ast$, we also have that $X_s \to Y_s$ is an acyclic cofibration for each $s\in \Rr$. Thus $\incl X \to \incl Y$ is a weak equivalence, and condition (2) is established.
\end{proof}

Notice that the inclusion functor does not admit a right adjoint (as it does not preserve finite coproducts), hence cannot be a left Quillen functor. Nevertheless, we have the following, which hinges on the fact that $\incl$ preserves pushouts and transfinite composition (Lemma~\ref{pushout lemma}).
\begin{proposition}\label{cofibration preservation}
	The inclusion functor $\incl : \SSet^{\Rr^{op}}_\ast \to \SSet^{\Rr^{op}}$ preserves (acyclic) cofibrations.
\end{proposition}
\begin{proof}
The inclusion functor preserves all weak equivalences, so it is enough to show that it preserves cofibrations.
By Proposition \ref{reduction of generator one} and Proposition \ref{reduction of generator two}, we know that $\incl$ takes generating cofibrations in $\red I$ to cofibrations in $\SSet^{\Rr^{op}}$.
The inclusion functor preserves pushouts (Lemma \ref{pushout lemma}) and transfinite composition, so takes relative $\red I$-cell complexes to cofibrations.
Since every cofibration in $\SSet^{\Rr^{op}}_\ast$ is a retract of a relative $\red I$-cell complex, every cofibration in $\SSet^{\Rr^{op}}_\ast$ is again a cofibration in $\SSet^{\Rr^{op}}$.
\end{proof}

The following has precursors elsewhere in special cases, for instance in the proof of \cite[4.3]{bh1}.

\begin{proposition}\label{reedy type localizability}
With the hypotheses of Theorem \ref{reedy type model structure for discrete}, the model structure on $\SSet^{\Rr^{op}}_\ast$ is left proper and cellular.
\end{proposition}
\begin{proof}
For left properness, it is enough to show that the pushout of a weak equivalence along a \emph{generating} cofibration $\red(i) \in \red I$ is again a weak equivalence; this follows from Lemma \ref{levelwise acyclic cof}, left properness of $\SSet$, and the fact that weak equivalences are levelwise weak equivalences.

We now turn to cellularity, and aim to verify the three conditions from \cite[12.1.1]{hirschhorn}. We know that (2) holds by Lemma \ref{smallness lemma}.
Recall from Proposition \ref{cellularity} that $\SSet^{\Rr^{op}}$ is a cellular model category.
Cofibrations in $\SSet^{\Rr^{op}}_\ast$ are also cofibrations (Proposition \ref{cofibration preservation}) in the cellular model category $\SSet^{\Rr^{op}}$, hence are effective monomorphisms. Thus (3) holds.

It remains to show (1).
All elements of the form $\Rr[r] \times \Delta[n]$ are compact relative to $I$ since they are codomains of elements of $I$ and $\SSet^{\Rr^{op}}$ is cellular. 
If $A$ is a domain or a codomain of an element of $I$, then 
$A^{(0)}$ is compact by Lemma~\ref{zero skeleta lemma}. 
It follows that $\red(A)$ is compact (relative to $I$) by \cite[10.8.8]{hirschhorn}, since $A$, $A^{(0)}$ and $\Rr[\eta] = \Rr[\eta] \times \Delta[0]$ are compact.
This shows that the set $\red(I)$ is a set of cofibrations of $\SSet^{\Rr^{op}}$ with compact domains.
By \cite[11.4.9]{hirschhorn}, if $W\in \SSet^{\Rr^{op}}$ is compact relative to $I$, then $W$ is compact relative to $\red(I)$.
In particular, $\red(A)$ is compact relative to $\red(I)$.
Presented relative $\red(I)$-cell complexes in $\SSet^{\Rr^{op}}$ with $X_0 \in \SSet^{\Rr^{op}}_\ast$ are the same thing as presented relative $\red(I)$-cell complexes in $\SSet^{\Rr^{op}}_\ast$, hence all domains and codomains of elements in $\red(I) \cup \red(J)$ are compact in $\SSet^{\Rr^{op}}_\ast$.
Thus \cite[12.1.1(1)]{hirschhorn} holds, and we conclude that $\SSet^{\Rr^{op}}_\ast$ is a cellular model category.
\end{proof}

\subsection{Simplicial model structures}\label{simplicial model structures}

As in the case of an ordinary Reedy category, the Berger--Moerdijk Reedy model structure on $\SSet^{\Rr^{op}}$ is a \emph{simplicial model structure} (see \cite[Ch.~9]{hirschhorn}).
The structure is given as follows:
\begin{definition}
	Suppose that $X, Y \in \SSet^{\Rr^{op}}$ and $K \in \SSet$.
	\begin{itemize}
		\item The object $X\otimes K$ is defined on objects by
		\[
			(X\otimes K)_r = X_r \times K \in \SSet.
		\]
		\item The object $Y^K$ is defined on objects by
		\[
			(Y^K)_r = \map_{\SSet}(K, Y_r) \in \SSet
		\]
		where $\map_{\SSet}$ is the simplicial mapping space, that is, $\map_{\SSet}(A,B)_n = \hom_{\SSet} (A\times \Delta[n], B)$.
		\item The (simplicial) mapping spaces $\map(X,Y)$ are defined in degree $n$ by
		\[
			\map(X,Y)_n = \hom(X\otimes \Delta[n], Y).
		\]
	\end{itemize}
\end{definition}
The simplicially-enriched category $\SSet^{\Rr^{op}}$ satisfies the two conditions to be a simplicial model category \cite[9.1.6]{hirschhorn}, namely
\begin{description}
	\item[M6] For every two objects $X, Y$ and every $K\in \SSet$
	\[
		\map(X\otimes K, Y) \cong \map_{\SSet}(K, \map(X,Y)) \cong \map(X, Y^K)
	\]
	(with isomorphisms natural in the three variables).
	\item[M7] If $A\to B$ is a cofibration and $X\to Y$ is a fibration, then
	\begin{equation}\label{M7 map}
		\map(B,X) \to \map(A,X) \times_{\map(A,Y)} \map(B,Y)
	\end{equation}
	is a fibration of simplicial sets. 
	If either map is a weak equivalence, then \eqref{M7 map} is also a weak equivalence.
\end{description}

Assuming M6, condition M7 is equivalent (see, e.g. \cite[Remark A.3.1.6]{htt} or \cite[\S 3]{MR1852091}) to the following statement.
\begin{description}
	\item[M7$'$]\label{M7 prime map}
	If $X \to Y$ is a fibration in our model category and $K\to L$ is a cofibration of simplicial sets, then
\begin{equation}\label{eq pullback corner cotensor}
	X^L \to X^K \times_{Y^K} Y^L
\end{equation}
is a fibration. If either map is a weak equivalence, then so is \eqref{eq pullback corner cotensor}.
\end{description}

Any full subcategory of a simplicially-enriched category is simplicially-enriched; we will now work towards Theorem \ref{existence simplicial structure discrete}, where we show that $\SSet^{\Rr^{op}}_\ast$ is a simplicial model category.

\newcommand{\otimesM}{\otimes_{\mathcal M}}
\newcommand{\otimesN}{\otimes_{\mathcal N}}
\newcommand{\mapM}{\map_{\mathcal M}}
\newcommand{\mapN}{\map_{\mathcal N}}

\begin{notation}
	For the remainder of this section, we will write
	\begin{align*}
		\mathcal M &= \SSet^{\Rr^{op}} & \mathcal N &= \SSet^{\Rr^{op}}_\ast
	\end{align*}
	for these two model categories. 
	For $X, Y \in \mathcal M$, $K \in \SSet$, we will write $X\otimesM K := X \otimes K$ and $\mapM(X,Y) := \map(X,Y)$.	
\end{notation}

Recall that we have a Quillen adjunction
\[
	\red : \mathcal M \rightleftarrows \mathcal N : \incl.
\]

\begin{definition}\label{N simplicial structure}
	Suppose that $X, Y \in \mathcal N = \SSet^{\Rr^{op}}_\ast$ and $K \in \SSet$.
	\begin{itemize}
		\item The object $X\otimesN K $ is defined to be $X\otimesN K = \red(\incl(X)\otimesM K)$.
		\item The object $Y^K$ is defined as $Y^K = \red \left[ (\incl Y)^K \right]$.
		\item We define $\mapN(X,Y) = \mapM(\incl X,\incl Y)$.
	\end{itemize}
\end{definition}

\begin{remark}\label{exponent already discrete}
The object $Z = (\incl Y)^K$ already has $Z_\eta = \Delta[0]$, which explains why we've elected not to distinguish between the exponential in the two categories.
Indeed,
\[
	Z_\eta = ((\incl Y)^K)_\eta = \map_{\SSet}(K, Y_\eta) = \map_{\SSet}(K, \Delta[0]) = \Delta[0].
\]
\end{remark}

\begin{lemma}\label{enriched adjunction}
There is an isomorphism
\[
	\mapN(\red(-), -) \cong \mapM(-,\incl(-))
\]
of bifunctors $\mathcal N \times \mathcal M \to \SSet$.
\end{lemma}
\begin{proof}
	Let $Z \in \mathcal N$, $Y\in \mathcal M$, and $n\geq 0$.
We have
\begin{align*}
	\mapN(\red (Z),Y)_n &= \mapM(\incl \red (Z), \incl Y)_n \\
	&= \hom_{\mathcal M}(\incl \red (Z) \otimesM \Delta[n], \incl Y) \\
	&= \hom_{\mathcal M}(\incl \red (Z), (\incl Y)^{\Delta[n]}) \\
	&= \hom_{\mathcal M}(\incl \red (Z), \incl\red (\incl Y)^{\Delta[n]}) & \text{Remark \ref{exponent already discrete}} \\
	&= \hom_{\mathcal N}(\red (Z), \red (\incl Y)^{\Delta[n]}) \\
	&= \hom_{\mathcal M}(Z, \incl \red (\incl Y)^{\Delta[n]}) \\
	&= \hom_{\mathcal M}(Z, (\incl Y)^{\Delta[n]}) = \mapM(Z,\incl Y)_n.
\end{align*}
with all isomorphisms natural in $Z, Y$, and $n$.
\end{proof}

\begin{theorem}\label{existence simplicial structure discrete}
With the structure from Definition \ref{N simplicial structure}, $\SSet^{\Rr^{op}}_\ast$ is a simplicial model category.
\end{theorem}
\begin{proof}
By Remark~\ref{exponent already discrete}, the fact that M7$'$ holds for $\mathcal M =  \SSet^{\Rr^{op}}$, and the fact that $\incl$ creates (acyclic) fibrations, M7$'$ holds for $\mathcal N = \SSet^{\Rr^{op}}_\ast$. Thus it is enough to check M6.

Let $X,Y \in \mathcal N$ and $K\in \SSet$.
First,
\[
	\mapN(X\otimesN K, Y) = \mapN(\red (\incl (X) \otimesM K), Y) = \mapM(\incl (X)\otimesM K, \incl Y)
\]
by Lemma \ref{enriched adjunction}.
Thus, using M6 for $\mathcal M$, the simplicial set $\mapN(X\otimesN K, Y)$ is isomorphic to, on the one hand, 
\[
	\mapM(\incl (X), (\incl Y)^K) = \mapM(\incl X, \incl \red (\incl Y)^K) = \mapN(X,Y^K)
\]
and on the other to
\[
	\map_{\SSet}(K, \mapM(\incl X, \incl Y)) = \map_{\SSet}(K, \mapN(X,Y)).
\]
\end{proof}

\begin{remark}
We are grateful to an anonymous referee for observing that the above results reflect a general pattern.
The adjunction $\incl \dashv \red$ is monadic since $\red$ is conservative and preserves connected colimits; thus $\SSet^{\Rr^{op}}_\ast$ is equivalent to the category of $\incl\red$-algebras. 
Remark~\ref{exponent already discrete} should come as no surprise, as the simplicial cotensorings are an enriched limit, hence should be computed in the ground category $\SSet^{\Rr^{op}}$.
The rest of the structure (simplicial tensorings and simplicial homs) are then forced by the two-variable adjunctions of M6.
Provided these exist, the simplicial structure is guaranteed since M7 is equivalent to M7$'$.
\end{remark}

\section{Segal cyclic operads}\label{section segal cyclic}

In this section we define Segal cyclic operads as certain fibrant objects in $\SSet^{\Xi^{op}}_\ast$ which satisfy a Segal condition (Definition \ref{def segal cyclic}). 
The Segal cyclic operads may be identified as the fibrant objects after we have localized the model structure on $\SSet^{\Xi^{op}}_\ast$ with respect to Segal core inclusions.

We begin this section by specializing the work of Section \ref{section bmr model structure} to the cases $\Rr = \Omega$ or $\Xi$.
We show that the Quillen adjunctions from Theorem~\ref{reedy type model structure for discrete} fit into a diagram \eqref{first diagram of adjunctions}.
Afterward, we discuss the left Bousfield localizations, and show that after localization we still have a diagram of Quillen adjunctions.
Finally, we check in Proposition \ref{not q equiv} that the homotopy theory for Segal cyclic operads is distinct from that for Segal operads, and speculate on the possibility of rigidification theorems.

\begin{proposition}\label{BMR adjunction}
Using the Berger--Moerdijk Reedy model structures, 
the adjunction
	\[
		\iota_! :  \SSet^{\Omega^{op}} \rightleftarrows \SSet^{\Xi^{op}}  : \iota^*
	\]
	is a Quillen adjunction.
\end{proposition}
\begin{proof}
The map $\iota^*$ preserves weak equivalences since those are defined levelwise, hence it is enough to show that $\iota^*$ preserves fibrations.
If $T$ is a rooted tree, 
we will show that
\begin{equation} \label{functor to show is initial} \begin{tikzcd}[row sep=small]
(\Omega^{op})^-(T) \rar \dar[equal] & (\Xi^{op})^-(\iota T) \dar[equal] \\
(\Omega^+(T))^{op} & (\Xi^+(\iota T))^{op}
\end{tikzcd} \end{equation}
is an initial functor (see \cite[\S IX.3]{maclane}). This implies that 
the natural map
\[
	M_{\iota T} X = \lim_{(\Xi^{op})^-(\iota T)} X_S \to \lim_{(\Omega^{op})^-(T)} (\iota^*X)_{T'} = M_T(\iota^*X)
\]
is an isomorphism. Hence if $X \to Y$ is a map in $\SSet^{\Xi^{op}}$ and $T\in \Omega$, the map on the right of the commutative diagram
\[ \begin{tikzcd}[row sep=small]
(\iota^* X)_T \rar \dar[equal] & M_T(\iota^* X) \times_{M_T(\iota^* Y)} (\iota^* Y)_T \\
X_{\iota T} \rar & M_{\iota T}(X) \times_{M_{\iota T}(Y)} Y_{\iota T} \uar
\end{tikzcd} \]
is an isomorphism.
In particular, if $X\to Y$ is a fibration, then so is $\iota^*(X\to Y)$.

As promised, we will now show that \eqref{functor to show is initial} is an initial functor;
this is equivalent to the induced functor
\begin{equation}\label{eq Omega T to Xi T}
\iota_T : \Omega^+(T) \to \Xi^+(\iota T)
\end{equation}
being final. 
Suppose that $S\overset\phi\to \iota T$ is an object of $\Xi^+(\iota T)$,
that is, $\phi$ is an element of $\Xi^+(S, \iota T)\setminus \Iso(\Xi)$.
Our goal is to show that $\phi \downarrow \iota_T$ is nonempty and connected.
We first explain the case when $\phi$ is not constant (that is, when $S \neq \eta$), and indicate later the changes for the simpler case when $S=\eta$.
Write $s_0 = \findroot_{t_0}(\phi)$, where $t_0$ is the root of $T$; we have a morphism
\[
	\lifting_{t_0}(\phi) \colon \rootify(S,s_0) \to \rootify(\iota T, t_0) = T.
\]
Using the structure map $f : \iota \rootify(S,s_0) \overset\cong\to S$ from Definition \ref{treeificiation}, the commutative diagram
\[ \begin{tikzcd}
S
\arrow{dr}[swap]{\phi} \arrow{rr}{f^{-1}} 
& & \arrow{dl}{\iota \lifting_{t_0}(\phi)} \iota \rootify(S,s_0) \\
& \iota T
\end{tikzcd} \]
in $\Xi$ constitutes an object
\[ \phi \xrightarrow{f^{-1}} \iota_T(\lifting_{t_0}(\phi))  \]
in $\phi \downarrow \iota_T$.
To show that this category is connected, suppose that we have an arbitrary object 
\[
	\phi \xrightarrow{\gamma} \iota_T\left( R\overset\alpha\to T \right) 
\]
of $\phi \downarrow \iota_T$, where $\alpha \in \Omega^+(T)$. 
Such an object corresponds to a commutative diagram
\[ \begin{tikzcd}
S
\arrow{dr}[swap]{\phi} \arrow{rr}{\gamma} 
& & \arrow{dl}{\iota \alpha} \iota R \\
& \iota T
\end{tikzcd} \]
with $\gamma \in \Xi^+$.
Lifting all maps to $\Omega$, we have the diagram
\[ \begin{tikzcd}
\rootify (S,s_0)
\arrow[dr, "\lifting_{t_0}(\phi)" swap] \arrow{rr}{\lifting_{r_0}(\gamma)} 
& & \arrow[dl, "\lifting_{t_0}(\iota \alpha)" description] \rootify(\iota R, r_0) \rar[equal] & R \arrow{dl}{\alpha}\\
& \rootify(\iota T, t_0) \rar[equal] & T
\end{tikzcd} \]
in $\Omega$, which commutes by Proposition \ref{functoriality of lifting}.
Commutativity of the diagram
\[ \begin{tikzcd}[row sep=large, column sep=large]
S \arrow{dr}[swap]{\phi} \arrow[rr, bend left, "\gamma" swap] 
\rar{f^{-1}}
& \iota \rootify(S,s_0) \dar["\iota(\lifting_{t_0}(\phi))" description] 
\rar{\iota(\lifting_{r_0} (\gamma))}& \iota R \arrow[dl, "\iota(\alpha)"] \\
& \iota T
\end{tikzcd} \]
in $\Xi$ shows that $\iota(\lifting_{r_0} (\gamma))$ constitutes a morphism $f^{-1} \to \gamma$ in $\phi \downarrow \iota_T$; thus this category is connected.

A similar proof holds when $\phi$ is constant, that is, when $S=\eta$. In this case, $\rootify (S,s_0)$ should be replaced by $\eta$, $f$ should be taken to be the identity map, and $\lifting_{t_0}(\phi)$ (resp.\ $\lifting_{r_0}(\gamma)$) should be replaced by the unique lift $\eta \to T$ of $\phi$ (resp.\ the unique lift $\eta \to R$ of $\gamma$).  
Since $\phi \downarrow \iota_T$ is connected for every object $\phi \in \Xi^+(\iota T)$, \eqref{eq Omega T to Xi T} is a final functor.
\end{proof}

\begin{corollary}\label{BMR adjunction corollary}
The adjunction
	\[
		\iota_! :  \SSet^{\Omega^{op}}_\ast \rightleftarrows \SSet^{\Xi^{op}}_\ast  : \iota^*
	\]
	is a Quillen adjunction.
\end{corollary}
\begin{proof}
We have a commutative diagram
\[ \begin{tikzcd}
\SSet^{\Omega^{op}}  & \SSet^{\Xi^{op}} \lar{\iota^*}
\\
\SSet^{\Omega^{op}}_\ast  \uar{\incl} & \SSet^{\Xi^{op}}_\ast \lar{\iota^*} \uar{\incl}
\end{tikzcd} \]
of right adjoints
where all but the bottom map $\iota^*$ are known to be right Quillen functors.
Suppose that $X\to Y$ is an (acyclic) fibration in $\SSet^{\Xi^{op}}_\ast$, which implies $\incl \iota^* (X\to Y) = \iota^* \incl (X\to Y)$ is an (acyclic) fibration in $\SSet^{\Omega^{op}}$.
Since $\incl: \SSet^{\Omega^{op}}_\ast \to \SSet^{\Omega^{op}}$ detects fibrations and weak equivalences we known that $\iota^*(X\to Y$) is an (acyclic) fibration, implying $\iota^* \colon \SSet^{\Xi^{op}}_\ast \to \SSet^{\Omega^{op}}_\ast$ is a right Quillen functor.
\end{proof}

We now have a diagram of Quillen adjunctions
\begin{equation}\label{first diagram of adjunctions} \begin{tikzcd}
\SSet^{\Omega^{op}} \rar[shift left, "\iota_!"] \dar[shift right, "\red" swap] & \SSet^{\Xi^{op}} \dar[shift right, "\red" swap] \lar[shift left, "\iota^*"] \\
\SSet^{\Omega^{op}}_\ast \rar[shift left, "\iota_!"]  \uar[shift right, "\incl" swap] & \SSet^{\Xi^{op}}_\ast  \lar[shift left, "\iota^*"] \uar[shift right, "\incl" swap]
\end{tikzcd} \end{equation}
and we wish to localize each of these model structures.

\subsection{Localizations}
Roughly speaking, a left localization of a model category $\mathcal M$ at a set of maps $C$ is a left Quillen functor from $\mathcal M$ which is initial among all left Quillen functors which take elements of $C$ to weak equivalences.
Recall from \cite[3.1.4]{hirschhorn} that an object $W$ of $\mathcal M$ is called $C$-local if $W$ is fibrant and for each $f : A \to B$ which is an element of $C$, the map $\map^h(B, W) \to \map^h(A,W)$ is a weak equivalence of simplicial sets. A map $g: X \to Y$ is called a $C$-local equivalence if $\map^h(Y, W) \to \map^h(X, W)$ is a weak equivalence for every $C$-local $W$.
The left Bousfield localization of $\mathcal M$ with respect to $C$ \cite[3.3.1]{hirschhorn}, denoted $\mathcal L_C\mathcal M$, is then a model structure (which may or may not exist) on $\mathcal M$ with weak equivalences the $C$-local equivalences and cofibrations the ordinary cofibrations in $\mathcal M$.
The fibrant objects in this model structure (if it exists) are precisely the $C$-local objects, and the identity functor $\mathcal M \to \mathcal L_C\mathcal M$ is a left localization of $\mathcal M$.

In order to show that the diagram \eqref{first diagram of adjunctions} remains a diagram of Quillen adjunctions after localization, we will apply the following lemma several times.
\begin{lemma}\label{qa on localizations}
Let $L : \mathcal M \rightleftarrows \mathcal N : R$ be a Quillen adjunction.
Suppose that $C\subseteq \mathcal M$ and $D\subseteq \mathcal N$ are sets of maps with the domain and codomain of each element of $C$ cofibrant.
Suppose further that the left Bousfield localizations $\mathcal L_C \mathcal M$ and $\mathcal L_D \mathcal N$ exist.
If, for each $c \in C$, the map $L(c) \in \mathcal N$ is isomorphic to some $d\in D$, then 
\[
	L : \mathcal L_C \mathcal M \rightleftarrows \mathcal L_D \mathcal N : R
\]
is a Quillen adjunction.
\end{lemma}
\begin{proof}
There is a left Quillen functor
\[
	F : \mathcal M \to \mathcal N \to \mathcal L_D \mathcal N.
\]
If $c\in C$, then $c$ is a cofibrant approximation to itself.
Further, $F(c) \cong d \in D$ is a weak equivalence, hence $F$ takes any cofibrant approximation of $c$ to a weak equivalence by \cite[8.1.24(1)]{hirschhorn}.
By \cite[3.3.18(1)]{hirschhorn}, the functor $F$ is then a left Quillen functor when regarded as a functor $\mathcal L_C \mathcal M \to \mathcal L_D \mathcal N.$
\end{proof}

Let $\Rr$ be either $\Omega$ or $\Xi$.
Define $\Seg_\Rr$ to be the set of Segal core inclusions
\[ \Seg_\Rr = \{ \Sc[r] \to \Rr[r] \}_{r\neq \eta} \]

\begin{remark}\label{homotopy function complex remark}
According to \S\ref{simplicial model structures},  $\SSet^{\Rr^{op}}$ and $\SSet^{\Rr^{op}}_\ast$ are \emph{simplicial model categories}.
Thus, if $A$ is cofibrant and $X$ is fibrant, we may use the simplicial mapping space $\map(A,X)$ as a model for the homotopy mapping space $\map^h(A,X)$ (by, for example, \cite[Example 17.2.4]{hirschhorn}).
Since $\Sc[r], \Rr[r], \red(\Sc[r])$, and $\red(\Rr[r])$ are all cofibrant,\footnote{To 
	establish cofibrancy of $\Xi[S]$, simply pick a rooted tree with $\iota T = S$. Then since $\Sc[T]$ is cofibrant \cite[Corollary 1.7]{cm-ho}, so is $\iota_!(\Sc[T]) \cong \Sc[\iota T] \cong \Sc[S]$ by Lemma~\ref{segal inclusion different category lemma} and Proposition~\ref{BMR adjunction}.
} 
it suffices to work with $\map$ rather than $\map^h$ when discussing $\Seg_\Rr$ or $\red(\Seg_\Rr)$ locality.
\end{remark}

Since $\SSet^{\Rr^{op}}$ is left proper and cellular by Theorem \ref{left properness} and Proposition \ref{cellularity}, the left Bousfield localization
\[
	\mathcal L_{\Seg_\Rr} \SSet^{\Rr^{op}}
\]
exists by \cite[4.1.1]{hirschhorn}.
Since $\SSet^{\Rr^{op}}_\ast$ is left proper and cellular by Proposition \ref{reedy type localizability}, we can take the left Bousfield localization 
with respect to the set of maps $\red(\Seg_\Rr)$.
We call the resulting model structure
\[
	\mathcal L_{\red(\Seg_\Rr)} \SSet^{\Rr^{op}}_\ast.
\]
For $\Rr = \Omega$, these two model structures appear in \cite[Definition 5.4]{cm-ds} and \cite[Proposition 4.3]{bh1}, respectively.

\begin{proposition}\label{same category adj}
Let $\Rr$ be either $\Omega$ or $\Xi$. Then 
the Quillen adjunction 
\[
	\red \colon \SSet^{\Rr^{op}} \rightleftarrows \SSet^{\Rr^{op}}_\ast : \incl
\]
induces a Quillen adjunction 
\[
		\red \colon \mathcal L_{\Seg_\Rr} \SSet^{\Rr^{op}} \rightleftarrows \mathcal L_{\red(\Seg_\Rr)} \SSet^{\Rr^{op}}_\ast : \incl
\]
after taking left Bousfield localization.
\end{proposition}
\begin{proof}
If $r \neq \eta$ is an object of $\Rr$, then both $\Sc[r]$ and $\Rr[r]$ are cofibrant.
Thus Lemma \ref{qa on localizations} applies.
\end{proof}

The following is a variation on Lemma \ref{segal inclusion different category lemma}. 

\begin{proposition}\label{reduced segal different category}
	If $T\in \Omega$ is a rooted tree, then
	\[
		\iota_!\Big(\red (\Sc[T]) \to \red (\Omega[T])\Big) \cong \Big(\red (\Sc[\iota T]) \to \red (\Xi[\iota T])\Big).
	\]
\end{proposition}
\begin{proof}
We have
\begin{align*}
	\iota_!\red(\Sc[T] \to \Omega[T]) &= \red\iota_!(\Sc[T] \to \Omega[T]) & \eqref{first diagram of adjunctions} \\
	&\cong \red(\Sc[\iota T] \to \Xi[\iota T]) & \text{Lemma \ref{segal inclusion different category lemma}.}
\end{align*}
\end{proof}

\begin{proposition}\label{adjunctions after localizing}
The diagram \eqref{first diagram of adjunctions} gives a diagram
\[ \begin{tikzcd}
\mathcal L_{\Seg_\Omega}\SSet^{\Omega^{op}} \rar[shift left, "\iota_!"] \dar[shift right, "\red" swap] & \mathcal L_{\Seg_\Xi}\SSet^{\Xi^{op}} \dar[shift right, "\red" swap] \lar[shift left, "\iota^*"] \\
\mathcal L_{\red(\Seg_\Omega)}\SSet^{\Omega^{op}}_\ast \rar[shift left, "\iota_!"]  \uar[shift right, "\incl" swap] & \mathcal L_{\red(\Seg_\Xi)}\SSet^{\Xi^{op}}_\ast  \lar[shift left, "\iota^*"] \uar[shift right, "\incl" swap]
\end{tikzcd} \]
of Quillen adjunctions after localizing.
\end{proposition}
\begin{proof}
In light of Proposition \ref{same category adj}, we only need to show that the horizontal adjunctions are Quillen adjunctions.
The objects $\Sc[T]$ and $\Omega[T]$ are cofibrant in $\SSet^{\Omega^{op}}$, so the top adjunction is a Quillen adjunction by Lemma \ref{qa on localizations} and Lemma \ref{segal inclusion different category lemma}.
Since $\red$ is a left Quillen functor, $\red(\Sc[T])$ and $\red(\Omega[T])$ are cofibrant in $\SSet^{\Omega^{op}}_\ast$.
Thus the bottom adjunction is a Quillen adjunction by Lemma \ref{qa on localizations} and Proposition \ref{reduced segal different category}.
\end{proof}

\begin{definition}\label{def segal cyclic}
A \emph{Segal cyclic operad} is a fibrant object in the model category $\mathcal L_{\red(\Seg_\Xi)}\SSet^{\Xi^{op}}_\ast$.
\end{definition}
Notice that every Segal cyclic operad has an underlying Segal operad via the functor $\iota^*$.

The following example gives two cyclic operad structures on the same underlying operad. In fact, this hints at a whole class of examples: if $A$ is an abelian group, then cyclic structures on the operad given by $A$ are in bijection with the order $1$ and $2$ elements of $\Aut(A)$.

\begin{example}\label{cyclic example}
Let $M$ be the group $\mathbb Z / 2 \times \mathbb Z /2 = \{0, 1\} \times \{0, 1\}$, whose elements are written as $00, 01, 10, 11$.
Then $M$ determines a (monochrome) operad $O$ with 
\[
	O (n) = \begin{cases}
		\{00, 01, 10, 11\} & n= 1 \\
		\varnothing & n \neq 1
	\end{cases}
\]
with the operadic multiplication given by the the group operation.

The operad $O$ admits distinct cyclic structures $C$ and $C'$. For the first, the action of $\Sigma^+_1 = \Sigma_2$ interchanges $01$ and $10$ and fixes $00$ and $11$, while in the second, the action fixes every element.
These two structures are not isomorphic as cyclic operads because the $\Sigma_2$-sets $C(1) = \Sigma_2 \amalg * \amalg *$ and $C'(1) = * \amalg * \amalg * \amalg *$ are not isomorphic.
\end{example}

\begin{proposition}\label{not q equiv}
	The Quillen adjunction
	\[ \iota_! : \mathcal L_{\red(\Seg_\Omega)}\SSet^{\Omega^{op}}_\ast \rightleftarrows \mathcal L_{\red(\Seg_\Xi)}\SSet^{\Xi^{op}}_\ast  : \iota^* \]
	does not induce an equivalence of homotopy categories.
	In particular, this adjunction is not a Quillen equivalence.
\end{proposition}
\begin{proof}
Consider the two cyclic operads $C, C'$ in $\Set$ from Example \ref{cyclic example}; recall from that example that $UC = UC'$.
Let $X = N_c(C)$ and $X' = N_c(C')$, and note that $\iota^*X = \iota^* X'$.
Additionally, let $A$ be the operad with
\[
	A(n) = \begin{cases}
		\mathbb{Z}/2 = \{ e, x \mid x^2 = e \} & n= 1 \\
		\varnothing & n \neq 1.
	\end{cases}
\]
The operad $A$ admits a unique cyclic operad structure where the $\Sigma_2$ action on $A(1)$ fixes $x$.
There are exactly two maps of cyclic operads $A \to C$, while there are four maps $A \to C'$; we will show that this remains true once we pass to the homotopy category of $\mathcal L_{\red(\Seg_\Xi)}\SSet^{\Xi^{op}}_\ast$.
Write $W = N_c(A)$, and note that $W_S = \varnothing$ if $S$ is non-linear, while $W_{L_m}$ is the set of words of length $m$ in the alphabet $e,x$. 
The objects $W, X, X'$ are fibrant in $\mathcal L_{\red(\Seg_\Xi)}\SSet^{\Xi^{op}}_\ast$ by Theorem~\ref{reedy type model structure for discrete}, Theorem \ref{cyclic dendroidal nerve theorem}(iii), and the fact that all maps between discrete simplicial sets are Kan fibrations.

Let $E\Sigma_2 \to \Delta[0]$ be a cofibrant resolution of the terminal object of $\SSet^{\Sigma_2}$; there is a cofibrant resolution of $W$ in the unlocalized model structure $\SSet^{\Xi^{op}}_\ast$ that is isomorphic levelwise to the tensor product $W\otimes E\Sigma_2$ from Definition \ref{N simplicial structure}.
The object $W\otimes E\Sigma_2$ does not take into account the $\Sigma_2$-structure on $E\Sigma_2$, so we must modify the presheaf structure a bit.
As this seems interesting in its own right, we discuss for $K\in \SSet^{\Sigma_2}$ tensoring ($-\circledast K$) and cotensoring ($K \pitchfork -$) in detail in Appendix~\ref{appendix cotensorings flippy sets}. 
We now show that a cofibrant resolution of $W$ is given by $\widetilde W = W \circledast E\Sigma_2 \to W$.
We have $(W \circledast E\Sigma_2)_S = (W\otimes E\Sigma_2)_S$ for every tree $S$.
To see that $\widetilde W \simeq W$, notice that at $L_m$ we have
\[
	\widetilde W_{L_m} = \Big( (W_{L_m} \setminus \{e^{\times m}\}) \times E\Sigma_2 \Big)_+ \simeq \Big( (W_{L_m} \setminus \{e^{\times m}\}) \Big)_+ = W_{L_m}
\]
and $\widetilde W_S = \varnothing$ if $S$ is non-linear. 

To see that $\widetilde W$ is cofibrant, notice that $W$ and $\widetilde W$ admit a filtration with $W^{(k)}$ consisting of words which have $x$ appearing at most $k$ times and $\widetilde W^{(k)} = W^{(k)} \circledast E\Sigma_2$.
Then $\widetilde W^{(k)}$ is the pushout in $\SSet^{\Xi^{op}}_\ast$,
\[ \begin{tikzcd}
\red(\partial \Xi[L_k])\circledast E\Sigma_2 \rar \dar & \widetilde W^{(k-1)} \dar \\
\red(\Xi[L_k])\circledast E\Sigma_2 \rar & \widetilde W^{(k)}.
\end{tikzcd} \]
By Proposition~\ref{circledast preserves cof} and Proposition~\ref{creates cofibrations}, $\widetilde W^{(k-1)} \to \widetilde W^{(k)}$ is a cofibration in $\SSet^{\Xi^{op}}_\ast$.
Since $\widetilde W^{(0)} = W^{(0)} = \Xi[\eta]$ is the initial object of this category, it follows that $\widetilde W = \colim \widetilde W^{(k)}$ is cofibrant.

We now have
\[ \hom(\widetilde W, X) = \hom(N_cA \circledast E\Sigma_2, N_cC) = \hom(N_cA, E\Sigma_2 \pitchfork N_cC); \]
but $N_cC$ is levelwise discrete, so $E\Sigma_2 \pitchfork N_cC \cong N_cC$. 
Using Theorem \ref{cyclic dendroidal nerve theorem}, we see \[ \hom(\widetilde W, X) = \hom(N_cA, N_cC) = \hom(A,C) \]
and similarly $\hom(\widetilde W, X') = \hom(A, C')$.
But these sets are especially easy to understand.
Maps of cyclic operads from $A$ to any other cyclic operad are determined by where we send $x$, and we have
\begin{align*}
\hom(\widetilde W, X) & = \hom(A,C) = \{f_{00}, f_{11} \} \\
\hom(\widetilde W, X') & = \hom(A,C') = \{f_{00}, f_{01}, f_{10}, f_{11}\} 
\end{align*}
where $f_a(x) = a$.

Since $\widetilde W$ is cofibrant in the unlocalized model structure and the objects $X, X'$ are fibrant in the localized model structure, we can compute the homotopy classes of maps from $\widetilde W$ to $X$ (or $X'$) in either the unlocalized or localized model structure and will get the same answer by \cite[3.5.2]{hirschhorn}.
We now work in the unlocalized model structure $\SSet^{\Xi^{op}}_\ast$, which is a simplicial model category.
Since our objects are levelwise discrete, $\widetilde W$ is cofibrant, and $X$ is fibrant, we have 
\[
	\hom(\widetilde W, X) = \pi_0 \hom(\widetilde W, X) = \pi_0 \map(\widetilde W, X) = \pi_0 \map^h(\widetilde W, X) = \pi(\widetilde W, X)
\]
by Remark~\ref{homotopy function complex remark}, \cite[Proposition 9.5.3]{hirschhorn}, and \cite[Proposition 9.5.24]{hirschhorn} (see \cite[\S 7.5]{hirschhorn} for $\pi$-notation). 
Similarly, $\hom(\widetilde W, X') = \pi(\widetilde W, X')$.
We thus have $|\pi(\widetilde W, X')| = 4 > 2 = |\pi(\widetilde W, X)|$, which shows that $X$ and $X'$ are not isomorphic in the homotopy category of $\mathcal L_{\red(\Seg_\Xi)}\SSet^{\Xi^{op}}_\ast$.
\end{proof}

There is a model structure on the category of (monochrome) simplicial cyclic operads where the weak equivalences and fibrations are defined as those maps which are levelwise weak equivalences.
This follows by considering either the multi-sorted algebraic theory of cyclic operads or the colored operad controlling cyclic operads (for the latter, see \cite[\S 1.6.4]{lukacs}), and then applying \cite[Theorem 4.7]{MR2263055} or \cite[Theorem 2.1]{bmresolution}.

\begin{conjecture}
	The model structure for simplicial cyclic operads is Quillen equivalent to $\mathcal L_{\red(\Seg_\Xi)}\SSet^{\Xi^{op}}_\ast$.
\end{conjecture}

Analogous results for simplicial monoids and for simplicial operads appear in \cite{MR2342822} and \cite{bh1}, respectively.

\appendix

\section{Tensoring and cotensoring with \texorpdfstring{$\Sigma_2$}{Σ₂} simplicial sets}\label{appendix cotensorings flippy sets}

As inspiration for this section, recall the cartesian closed structure on the category $\Set^{\Sigma_2}$ of involutive sets. The cartesian product $X\times Y$ has the diagonal $\Sigma_2$ action, and the internal hom is just the set of ordinary functions, together with the conjugation $\Sigma_2$ action $(\sigma \cdot f)(x) = \sigma \cdot f(\sigma \cdot x)$. 
The fixed points of the action on the internal hom are precisely the $\Sigma_2$-equivariant functions $X\to Y$.

Let $\nabla$ be the full subcategory of $\Xi$ spanned by the objects $L_n$ (Example~\ref{unrooted tree examples}), so that $\nabla$ is equivalent to $\Xi \downarrow \eta$.
If $n > 0$, write $\vartheta : L_n \to L_n$ for the unique non-oriented isomorphism, and if $n=0$ write $\vartheta = \id_{L_0}$.
Each morphism in $\nabla$ is of the form
\begin{equation}\label{nabla decomposition}
	L_n \overset{\vartheta^i}\to L_n \overset{\gamma}\to L_m
\end{equation}
where $\gamma$ is oriented and $i\in \{0,1\}$.
Moreover, the data $(i,\gamma)$ for the decomposition \eqref{nabla decomposition} of a map is unique if and only if that map is not constant (otherwise there are two such decompositions, see Corollary~\ref{structure from rooting linear}).
Left Kan extension along $\nabla \hookrightarrow \Xi$ gives a functor $\SSet^{\nabla^{op}}_\ast \to \SSet^{\Xi^{op}}_\ast$ which allows us to identify the former category as being equivalent to the full subcategory consisting of those $X\in \SSet^{\Xi^{op}}_\ast$  so that $X_R = \varnothing$ whenever $R$ is non-linear.
Notice that if $X \in \SSet^{\nabla^{op}}_\ast$, then $X_{L_n}$ may be considered as an object of $\SSet^{\Sigma_2}$ using the action of $\vartheta^*$.

The category $\SSet^{\nabla^{op}}$ is tensored and cotensored over $\SSet^{\Sigma_2}$.
For our purposes we are interested only in the reduced case, so we will reserve notation for that.
We have, for $X\in \SSet^{\nabla^{op}}_\ast$ and $K\in \SSet^{\Sigma_2}$ that $X\circledast K$ is given on objects as the pushout
\[ \begin{tikzcd}
 K = X_{L_0} \times K \rar \dar & X_{L_n} \times K \dar \\
\Delta[0] \rar &	(X \circledast K)_{L_n} 
\end{tikzcd} \]
in $\SSet^{\Sigma_2}$, where the top map is induced from the unique morphism $f_n : L_n \to L_0$.
The presheaf structure is given for a morphism \eqref{nabla decomposition} by the induced maps on pushouts coming from the commutative diagram
\[
	\begin{tikzcd}
		\Delta[0] \dar{=} & X_{L_0} \times K \lar \rar{f_m^* \times \id} \dar{\id \times \sigma^i} & X_{L_m} \times K  \dar{(\gamma\vartheta^i)^* \times \sigma^i}\\
		\Delta[0] & X_{L_0} \times K \lar \rar{f_n^* \times \id} & X_{L_n} \times K.
	\end{tikzcd}
\]
Note that this is well-defined on constant maps, though the decomposition \eqref{nabla decomposition} need not be unique in this case.

We write $K \pitchfork - $ for the right adjoint to $-\circledast K$.
If $Y\in \SSet^{\nabla^{op}}_\ast$ 
then \[ (K \pitchfork Y)_{L_n}  = \map_{\SSet}(K, Y_{L_n});\] in particular, if $n=0$ this is just $\Delta[0]$.
Given a morphism \eqref{nabla decomposition}, the map $(K \pitchfork Y)_{L_m} \to (K \pitchfork Y)_{L_n}$ is given by 
\[
	\map(\sigma^i , (\gamma\vartheta^i)^*) : \map(K, Y_{L_m}) \to \map(K, Y_{L_n}).
\]

As simplicial sets, we have (natural in $X,Y$)
\begin{align*}
	(X \circledast K)_{L_n} &= (X \otimes K)_{L_n} \\
	(K \pitchfork Y)_{L_n} &= (Y^K)_{L_n}
\end{align*}
where the objects on the right are those from Definition~\ref{N simplicial structure} after forgetting the $\Sigma_2$ action on $K$.
Note that if $\gamma : L_n \to L_m$ is an \emph{oriented} map, then the above definitions also give equalities of the maps $\gamma^*$:
\begin{align*}
	(X \circledast K)_{\gamma} &= (X \otimes K)_{\gamma} & 
	(K \pitchfork Y)_{\gamma} &= (Y^K)_{\gamma}.
\end{align*}
Where these functors differ is exactly on the $\vartheta^*$.
\begin{proposition}\label{circledast preserves cof}
If $K \in \SSet^{\Sigma_2}$, then the functor $- \circledast K : \SSet^{\nabla^{op}}_\ast \to \SSet^{\nabla^{op}}_\ast$ preserves Reedy cofibrations. 
\end{proposition}
\begin{proof}
To prove this statement, it is of course equivalent to prove that the right adjoint $K \pitchfork -$ preserves acyclic fibrations. 
Notice by M7$'$ and the fact that every simplicial set is cofibrant, we have that $(-)^K$ preserves (acyclic) fibrations.
Since $(-)^K$ preserves acyclic fibrations, $((-)^K)_{L_m} \cong(K\pitchfork -)_{L_m}$ for all $m\geq 0$, and Reedy weak equivalences are levelwise, we see that $K \pitchfork -$ sends acyclic fibrations to weak equivalences.

We will conclude by showing that that $K\pitchfork -$ preserves Reedy fibrations.
Recall the definition of fibration from Theorem~\ref{bm_reedy_ms}. 
The key point to check is that we have an isomorphism of matching objects
\[
	M_{L_m}(K \pitchfork Y) \cong M_{L_m}(Y^K)
\]
for every $m\geq 0$.
For an arbitrary presheaf $Z$, we have 
\[
	M_{L_m} Z = \lim_{
	\substack{
	(\nabla^{op})^-(L_m) \\
	\alpha \colon L_m \to L_n
	}}
	Z_{L_n} =
	\lim_{
	\substack{
	\nabla^+(L_m) \\
	\beta \colon L_n \to L_m
	}}
	Z_{L_n}.
\]
Let $\Delta^+(m) \subseteq \nabla^+(L_m)$ be the full subcategory whose objects are the \emph{oriented} maps $L_n \to L_m$ in $\nabla^+ \subseteq \Xi^+$.
Note that if we have a morphism
\[
	\begin{tikzcd}[column sep=tiny]
		L_n \arrow[rr, "\phi"] \arrow[dr, "\gamma"] & & L_p \arrow[dl, "\gamma'" swap]\\
		& L_m
	\end{tikzcd}
\]
with $\gamma, \gamma' \in \Delta^+(m)$, then $\phi$ is an oriented map.
Thus every morphism in $\Delta^+(m)$ is an oriented map as well, which explains the choice of notation.
(This argument is much like one appearing in the proof of Lemma~\ref{coface interchange}.)
Further, we have that $\Delta^+(m) \subseteq \nabla^+(L_m)$ is essentially surjective: a non-oriented map $\gamma \vartheta$ in $\nabla^+(L_m)$ is isomorphic to $\gamma$.
Thus the inclusion map is an equivalence of categories.

It follows that 
\[
	M_{L_m} Z \cong 
	\lim_{
	\substack{
	\Delta^+(m) \\
	\beta \colon L_n \to L_m
	}}
	Z_{L_n}.
\]
Since $K \pitchfork Y$ and $Y^K$ are equal when applied to oriented maps, it follows that we have, for each $K$, an isomorphism $M_{L_m}(K \pitchfork Y) \cong M_{L_m}(Y^K)$, natural in $Y$.
In fact, we have
\[
	\Big( (K\pitchfork Y)_{L_m} \to M_{L_m}(K \pitchfork Y) \Big) \cong \Big( (Y^K)_{L_m} \to M_{L_m}(Y^K) \Big)
\]
so the fact that $(-)^K$ preserves fibrations implies that $K\pitchfork -$ also preserves fibrations.
\end{proof}

\begin{proposition}\label{creates cofibrations}
Suppose that $\SSet^{\nabla^{op}}_\ast$ and $\SSet^{\Xi^{op}}_\ast$ are endowed with the model structures from Theorem~\ref{reedy type model structure for discrete}.
Consider the adjunction
\[ \SSet^{\nabla^{op}}_\ast \rightleftarrows \SSet^{\Xi^{op}}_\ast \]
induced from the full-subcategory inclusion $\nabla \to \Xi$. Then the left adjoint preserves and detects cofibrations.
\end{proposition}
\begin{proof}
We identify $\SSet^{\nabla^{op}}_\ast$ as the full subcategory whose objects $A$ satisfy $A_R = \varnothing$ if $R$ is non-linear.
By Theorem~\ref{reedy type model structure for discrete}, Proposition~\ref{cofibration preservation}, and $\red\incl\cong\id$, the functor $\incl: \SSet^{\Rr^{op}}_\ast \to \SSet^{\Rr^{op}}$ creates cofibrations (where $\Rr = \nabla,\Xi$). Thus it is enough to work with Reedy cofibrations in the unreduced categories.

Consider the indexing category for the latching object of $A$.
If $R$ is any object of $\Xi$, then $(\Xi^{op})^+(R) = \Xi^-(R)$ has objects of the form $R\to S$ in $\Xi^- \setminus \Iso(\Xi)$.
Given a map $R \to S$, if $S$ is linear, so is $R$.
It follows that 
\[ L_R^\Xi(A) = \colim_{
	\substack{
	\Xi^-(R) \\
	R \to S
	}}
	A_S   =  \varnothing \] when $R$ is a non-linear tree. 
If $R = L_m$, then 
\[
 L_{L_m}^\Xi(A) = \colim_{
	\substack{
	\Xi^-(L_m) \\
	L_m \to S
	}}
	A_S  
	= \colim_{
	\substack{
	\nabla^-(L_m) \\
	L_m \to L_n
	}}
	A_{L_n} = L_{L_m}^\nabla(A) \]
since $A_S = \varnothing$ when $S$ is not linear.
Thus a map $A \to B$ in $\SSet^{\nabla^{op}}_\ast$ is a cofibration in $\SSet^{\nabla^{op}}_\ast$ if and only if it is a cofibration in $\SSet^{\Xi^{op}}_\ast$.
\end{proof}

\bibliographystyle{amsplain}
\bibliography{cyc}
\end{document}